\documentclass[a4paper,twoside]{article}
\usepackage{makeidx}
\usepackage{amssymb,amsfonts,amsmath,amsbsy,theorem,amscd}
\usepackage{dina42e}
\usepackage{epsfig,psfrag,graphicx,xcolor}
\usepackage{stmaryrd}
\usepackage{mathrsfs}
\definecolor{CMUrot}{RGB}{128,18,18}
\definecolor{Gold}{RGB}{238,180,34}
\allowdisplaybreaks[4]

\newcommand{\ol}[1]{\overline{#1}}

\numberwithin{equation}{section}

\newcommand{\R}{\ensuremath{\mathbb{R}}}

\newcommand{\Rn}{\ensuremath{\mathbb{R}^2}}

\newcommand{\Om}{\ensuremath{\Omega}}

\newcommand{\N}{\ensuremath{\mathbb{N}}}
\newcommand{\J}{\ensuremath{\mathbf{J}}}

\newcommand{\SD}{\ensuremath{\mathcal{S}}}

\newcommand{\dist}{\operatorname{dist}}
\newcommand{\sdist}{\operatorname{sdist}}

\newcommand{\sd}{{\rm d}}

\newcommand{\eps}{\ensuremath{\varepsilon}}
\newcommand{\weight}[1]{\langle #1\rangle}

\newcommand{\Div}{\operatorname{div}}

\newcommand{\T}{\ensuremath{\mathbb{T}}}

\newcommand{\no}{\mathbf{n}}

\newcommand{\tn}[1]{\mathbf{#1}}
\def\nn{\mathbf{n}}
\let\e=\varepsilon
\newcommand{\ve}{\mathbf{v}}
\newcommand{\we}{\mathbf{w}}
\newcommand{\ue}{\mathbf{u}}

\newcommand{\btau}{{\boldsymbol{\tau}}}
\newtheorem{thm}{THEOREM}[section]
\newtheorem{cor}[thm]{Corollary}
\newtheorem{lem}[thm]{Lemma}
\newtheorem{defn}[thm]{Definition}
\newtheorem{theorem}[thm]{Theorem}
\newtheorem{prop}[thm]{Proposition}

\newtheorem{claim*}{Claim}

\theorembodyfont{\rmfamily}
\newtheorem{rem}[thm]{Remark}

\newenvironment{proof*}[1]{{\bf Proof
#1:}}{\hspace*{\fill}\rule{1.2ex}{1.2ex}\\ }
\newenvironment{proof}{{\bf
Proof:\,}}{\hspace*{\fill}\rule{1.2ex}{1.2ex}\\ }

\global\long\def\cte{\tilde{c}^{\epsilon}}

\newcommand{\p}{\partial}
\newcommand{\G}{\Gamma}
\def\({\left(}
\def\){\right)}
\renewcommand{\O}{\Omega}
\newcommand{\Lgrad}{L^{\nabla}}
\newcommand{\Ldelta}{L^{\Delta}}
\newcommand{\Lt}{L^{t}}

\newcommand{\divtau}{\Div_\btau}
\newcommand{\tc}{\hat{c}}
\newcommand{\tv}{\hat{\ve}}
\newcommand{\tp}{\hat{p}}
\newcommand{\tr}{\hat{r}}

\newcommand{\order}{N}
\newcommand{\zg}{\zeta\circ d_\Gamma}

\newcommand{\Diff}{u}
\newcommand{\NSt}[1]{#1}
\newcommand{\myeta}{\theta_0}
\begin{document}
\begin{titlepage}
\title{Sharp Interface Limit for a   Navier-Stokes/Allen-Cahn System with Different Viscosities }
\author{Helmut Abels and Mingwen Fei}
\end{titlepage}

\maketitle
\abstract{
We discuss the sharp interface limit of a coupled Navier-Stokes/Allen-Cahn system in a two dimensional, bounded and smooth domain,
when a parameter $\varepsilon>0$ that is proportional to the thickness of the diffuse interface
tends to zero, rigorously. We
prove convergence of the solutions of the Navier-Stokes/Allen-Cahn system to solutions
of a sharp interface model, where the interface evolution is given by the mean curvature flow with an additional convection term coupled to a two-phase Navier-Stokes system with surface tension. This is done by constructing an approximate solution from the limiting system via matched asymptotic expansions together with a novel Ansatz for the highest order term, and then estimating its difference with the real solution with the aid of a refined spectral estimate of the linearized Allen-Cahn operator near the approximate solution.
}

\bigskip

{\small\noindent
}

\bigskip

\section{Introduction and Main Result}

The study of two-phase flows of macroscopically immiscible fluids is a challenging and important problem  with many applications in the sciences and
engineering applications. There are two classes of mathematical models to deal with two-phase flows: so-called sharp interface models and diffuse interface or phase field models. In sharp interface models  the interface separating two fluids is treated as a sufficiently smooth surface
with zero width and the equations of motion that hold in each fluid are supplemented by boundary conditions at the sharp interface. In diffuse interface models a mixing of the macroscopically immiscible fluids on a small length scale is taken into account. Hence  the interface is treated as a transition layer of finite (but small) width $\varepsilon>0$, which is described by
a suitable scalar function (the so-called order parameter or phase field) $c_\eps$. Typically it is related to the concentrations or volume fractions of the fluids.
Diffuse interface models have many advantages in numerical simulations of the
interfacial motion since it can describe topological singularities  of interfaces such as pinch-off and reconnection. One of the important and natural problems is to investigate whether the diffuse interface model can be related to the corresponding
sharp interface model in the limit in which the interfacial width $\varepsilon$ tends to zero (i.e., by the so called the sharp-interface limit).
This  sharp interface limit is in fact a
question about the consistency of sharp and diffuse interface models.

In this paper we consider the singular limit $\eps\to 0$ of the following Navier-Stokes/Allen-Cahn system:
\begin{alignat}{2}\label{eq:NSAC1}
  \partial_t \ve_\eps +\ve_\eps\cdot \nabla \ve_\eps-\Div(2\nu(c_\eps)D\ve_\eps)  +\nabla p_\eps & = -\eps \Div (\nabla c_\eps \otimes \nabla c_\eps)&\quad & \text{in}\ \Omega\times(0,T_0),\\\label{eq:NSAC2}
  \Div \ve_\eps& = 0&\quad & \text{in}\ \Omega\times(0,T_0),\\\label{eq:NSAC3}
 \partial_t c_\eps +\ve_\eps\cdot \nabla c_\eps & =\Delta c_\eps - \frac1{\eps^2} f'(c_\eps)&\quad & \text{in}\ \Omega\times(0,T_0),\\
 \label{eq:NSAC4}
 (\ve_\eps,c_\eps)|_{\partial\Omega}&= (0,-1)&\quad & \text{on }\partial\Omega\times (0,T_0),\\
 \label{eq:NSAC5}
 (\ve_\eps,c_\eps) |_{t=0}& = (\ve_{0,\eps},c_{0,\eps})&\quad& \text{in }\Omega.
\end{alignat}
Here $\ve_\eps,p_\eps$ are the velocity and the
pressure of the fluid mixture, $c_\eps$ is the order parameter, which is related to  the
concentration difference of the fluids, $\nabla c_\eps \otimes \nabla c_\eps$ models the extra reactive stress, $\nu(c_\eps)$ describes the viscosity in dependence on $c_\eps$, and $f$ is a suitable smooth double well potential, e.g., $f(c)=\frac{1}{8}(c^2-1)^2$. Moreover, $D\ve_\eps=\frac 12(\nabla\ve_\eps+(\nabla\ve_\eps)^T)$, $\nabla$= $\nabla_x$, $\Div$= $\Div_x$,
 $\Delta$= $\Delta_x$ are always taken with respect to $x\in\Omega$. Furthermore, $\Omega \subseteq \R^2$ is a bounded domain with smooth boundary. We note that this model was suggested by Liu and Shen in \cite{LiuShenModelH} as an alternative approximation of a classical sharp interface model for a two-phase flow of viscous, incompressible, Newtonian fluids. Later the model was derived by Jiang et al.~\cite{TwoPhaseVariableDensityJiangEtAl} in a more general version for fluids with different densities and phase transitions. Moreover, existence of weak solutions and long time behaviour of the model was studied. The longtime behavior of solutions of this model was studied by Gal and Grasselli~\cite{GalGrasselliDCDS}. Recent analytic results on a mass-conserving Navier-Stokes/Allen-Cahn system and further reference can be found in Giorgini et al.~\cite{GiorginiGrasselliWu}.

We will prove the convergence of  \eqref{eq:NSAC1}-\eqref{eq:NSAC5} to the following sharp interface limit system:
\begin{alignat}{2}
\label{eq:Limit1}
    \p_t \ve_0^\pm+\ve_0^\pm\cdot\nabla\ve_0^\pm-\nu^\pm\Delta \ve_0^\pm  +\nabla p^\pm_0 &= 0 &\qquad &\text{in }\Omega^\pm (t), t\in (0,T_0),\\\label{eq:Limit2}
    \Div \ve_0^\pm &= 0 &\qquad &\text{in }\Omega^\pm (t), t\in (0,T_0),\\\label{eq:Limit3}
    \llbracket 2\nu^\pm D\ve_0^\pm -p_0^\pm \tn{I}\rrbracket\no_{\Gamma_t} &=- \sigma H_{\Gamma_t}\no_{\Gamma_t} && \text{on }\Gamma_t, t\in (0,T_0),\\ \label{eq:Limit4}
    \llbracket\ve_0^\pm \rrbracket &=0 && \text{on }\Gamma_t, t\in (0,T_0),\\
 \label{eq:Limit5}
    V_{\Gamma_t} -\no_{\Gamma_t}\cdot \ve_0^\pm &= H_{\Gamma_t} && \text{on }\Gamma_t, t\in (0,T_0),\\
    \ve_0^-|_{\partial\Omega}&= 0&&\text{on }\partial\Omega\times (0,T_0)\\
    \label{eq:Limit6}
    (\ve_0^\pm, \Gamma_t)|_{t=0} &= (\ve_{0,0}^\pm, \Gamma_0),
\end{alignat}
where $\nu^\pm=\nu(\pm 1)$,
$\Omega$ is the disjoint union of $\Omega^+(t), \Omega^-(t)$, and $\Gamma_t$ for every $t\in[0,T_0]$, $\Omega^\pm(t)$ are smooth domains, $\Gamma_t=\partial\Omega^+(t)$, and  $\no_{\Gamma_t}$ is the interior normal of $\Gamma_t$ with respect to  $\Omega^+(t)$. Moreover,
\begin{equation*}
  \llbracket u\rrbracket(p,t)= \lim_{h\to 0+} \left[u(p+\no_{\Gamma_t}(p)h)- u(p-\no_{\Gamma_t}(p)h)\right]
\end{equation*}
is the jump of a function $u\colon \Omega\times [0,T_0]\to \R^2$ at $\Gamma_t$ in direction of $\no_{\Gamma_t}$, $H_{\Gamma_t}$ and $V_{\Gamma_t}$ are the curvature and the normal velocity of $\Gamma_t$,  both with respect to $\no_{\Gamma_t}$. Furthermore, $D\ve=\frac 12(\nabla\ve+(\nabla\ve)^T)$ and $\sigma= \int_{\R}\theta_0'(\rho)^2\sd \rho$, where $\theta_0$ is the so-called optimal profile that is the unique solution of
\begin{alignat}{1}\label{eq:OptProfile1}
  -\theta_0''(\rho)+f'(\theta_0(\rho))&=0\qquad \text{for all }\rho\in\R,\\\label{eq:OptProfile2}
  \lim_{\rho\to\pm \infty}\theta_0(\rho)&=\pm 1,\qquad \theta_0(0)=0.
\end{alignat}
Using the method of formally matched asymptotic expansion, this sharp interface limit was formally discussed in \cite{AbelsConvectiveAC} together with arguments in \cite{AbelsGarckeGruen2}. In the present contribution we will verify this limit rigorously for well-prepared initial data as long as a sufficiently smooth solution of the limit system \eqref{eq:Limit1}-\eqref{eq:Limit6}. Here it is assumed that the interfaces $\Gamma_t$ and $\partial\Omega$ do not intersect for all $t\in [0,T_0]$. {We note that existence of strong solutions of \eqref{eq:Limit1}-\eqref{eq:Limit6} for sufficiently smooth initial data was proved by Moser and the first author in \cite{AbelsMoserNSAC}. By standard parabolic theory one can show that the solution is indeed smooth for smooth initial data satisfying the necessary compatibility conditions. Existence of weak solutions for a non-Newtonian variant of \eqref{eq:Limit1}-\eqref{eq:Limit6} was proved by Liu et al.~\cite{LiuSatoTonegawa2}.} 
The result {of the present contribution} generalizes a corresponding result by the first author and Liu in \cite{StokesAllenCahn}, where $\nu^+=\nu^-\equiv \nu$ was assumed and the terms $\partial_t \ve_\eps +\ve_\eps\cdot \nabla \ve_\eps$, $\p_t \ve_0^\pm+\ve_0^\pm\cdot\nabla\ve_0^\pm$, respectively, on the left-hand sides of the Navier-Stokes equation were neglected and a quasi-stationary Stokes flow was considered. Moreover, convergence was only verified for sufficiently small times $T_0>0$. A more detailed comparison will be given below.

In order to put our result into the relevant scientific context, we briefly summarize some important related results and models.
A fundamental phase field model was introduced by Cahn and Hilliard to model the process of phase separation and coarsening of a binary
alloy at a fixed temperature \cite{CahnHilliard}. In the Cahn-Hilliard equation, there is no coupling between diffusion and fluid mechanics.
In the case of fluids with same densities this coupling was first treated in the so-called model H, which was derived in \cite{GurtinTwoPhase,HohenbergHalperin}, and leads to a
Navier-Stokes/Cahn-Hilliard (is the so-called model H) system. This model describes the flow of two
macroscopically immiscible, viscous, incompressible Newtonian fluids. A Boussinesq approximation
of the model H in a Hele-Shaw cell is the so-called Hele-Shaw-Cahn-Hilliard system \cite{LeeLowengrub1}. Existence of solutions for the model H was e.g. shown in \cite{ModelH,BoyerModelH}. One can see \cite{Fei,WangZhangHSCH} for the existence of solutions for the Hele-Shaw-Cahn-Hilliard system. 
 Based on the method of formally matched asymptotic expansion, the sharp interface limit  for the model H was studied in \cite{AbelsGarckeGruen2}. The existence of solutions for the
sharp interface limit system was given in \cite{NSMS,StrongNSMS}. Despite the formal analysis  for
the sharp interface limit, there are only few rigorous results on proving the sharp interface
limit for the model H. Using the notion of varifold solutions as discussed in \cite{ChenSharpInterfaceLimit} such results for
large times were shown in \cite[Appendix of]{NSMS} for the model H, in \cite{Fei} for the Hele-Shaw-Cahn-Hilliard system, and in \cite{ChenSharpInterfaceLimit} also for a generalization of the model H for fluids with different densities, which was derived in \cite{AbelsGarckeGruen2}. But in sense of varifold solutions the convergence is rather
weak and holds only for a suitable
subsequence and no rates of convergence were obtained.

For the Allen-Cahn equation, i.e., \eqref{eq:NSAC3} with $\ve_\eps \equiv 0$, De~Mottoni and Schatzman~\cite{DeMottoniSchatzman} proved convergence to the mean curvature flow $V_{\Gamma_t}= H_{\Gamma_t}$ for well-prepared initial data as long as a smooth solution of the mean curvature flow exists.  They used the matched
asymptotic expansion method to construct an approximate solution, and then estimated the difference between the approximate solution and the real solution with the aid of a spectral estimate of the linearized Allen-Cahn operator.  The result also provides convergence rates in strong norms. This result was extended to the case of constant contact angle at the boundary of $\pi/2$ and close to it in \cite{AbelsMoser90Degree,AbelsMoserClose90Degree}. Alternative approaches to the sharp interface limit of the Allen-Cahn equation can be found in \cite{ESS,KKR} (based on viscosity solutions), \cite{Ilmanen,MizunoTonegawa,Kagaya} (based on varifold solutions), \cite{LauxSimon} (conditional result in a $BV$-setting) and in \cite{FischerLauxSimon} (relative entropy method, locally in time). Recently, the latter result was extended by {Hensel and Liu~\cite{HenselLiuModelH} to the Navier-}Stokes/Allen-Cahn system with constant viscosity {in space dimension $d=2,3$}.

Moreover, Alikakos, Bates,
and Chen in \cite{AlikakosLimitCH} proved that classical solutions of the Cahn-Hilliard equation tend to solutions of the Mullins-Sekerka problem (also
called the Hele-Shaw problem) with a modification of the method used by De~Mottoni and Schatzman. However, only few results with convergence rates in strong norms are known for two-phase flow models in fluid mechanics. In \cite{StokesAllenCahn} and \cite{AbelsMarquardt1,AbelsMarquardt2} the authors considered
a coupled Stokes/Allen-Cahn system and  Stokes/Cahn-Hilliard system in two dimensions, respectively. It is shown that smooth solutions of the
diffuse interface systems converge for short times to solutions of the corresponding sharp interface model, the so-called two-phase Stokes system and two-phase Stokes-Hele-Shaw system respectively, where
the evolution of the free surface is governed by a convective mean curvature flow and  the jump of the stress tensor,
accounting for capillary forces, which complies the Young-Laplace law. Finally, in \cite{PreprintRemarks} Jiang, Su, Xie extended recently the result of \cite{StokesAllenCahn} to the case of the instationary Stokes/Allen-Cahn system with some improvements in the error estimates. 

Our main result can be stated as follows:
\begin{thm}\label{thm:main}
  Let $N\geq 3$, $N\in\N$, $(\ve_0^\pm,\Gamma)$ be a smooth solution of \eqref{eq:Limit1}-\eqref{eq:Limit6} for some $T_0\in (0,\infty)$.
  Then there are smooth $c_{A,0}\colon \Omega\to \R$ and $\ve_{A,0}\colon \Omega\to \R^2$, depending on $\eps\in (0,1)$, such that the following is true: Let $(\ve_\eps,c_\eps)$ be strong solutions of \eqref{eq:NSAC1}-\eqref{eq:NSAC5} with initial values $c_{0,\eps}\colon \Omega\to [-1,1]$, $\ve_{0,\eps}\colon \Omega\to \R^2$, $0<\eps\leq 1$, satisfying
\begin{equation}\label{initial assumption}
  \|c_{0,\eps}-c_{A,0}\|_{L^2(\Omega)}+ \varepsilon^2\|\nabla(c_{0,\eps}-c_{A,0})\|_{L^2(\Omega)}+ \|\ve_{0,\eps}-\ve_{A,0}\|_{L^2(\Omega)}\leq C\eps^{\order+\frac12}
\end{equation}
for all $\eps\in (0,1]$ and some $C>0$.
 Then there are some $\eps_0 \in (0,1]$, $R>0$, and $c_A\colon \Om\times [0,T_0]\to \R$, $\ve_A\colon \Om\times [0,T_0]\to \R^2$ (depending on $\eps$) such that
\begin{subequations}\label{assumptions'}
  \begin{align}
 \sup_{0\leq t\leq T_0} \|c_\eps(t) -c_A(t)\|_{L^2(\Omega)}+\|\nabla (c_\eps -c_A)\|_{L^2(\Omega\times (0,T_0)\setminus\Gamma(\delta))}  &\leq R\eps^{\order+\frac12},\\
    \|\nabla_\btau(c_\eps -c_A)\|_{L^2(\Omega\times(0,T_0)\cap \Gamma(2\delta))}+ \eps \|\partial_\no(c_\eps -c_A)\|_{L^2(\Omega\times(0,T_0)\cap \Gamma(2\delta))} &\leq R\eps^{\order+\frac12},\\
    \|\nabla(c_\eps -c_A)\|_{L^\infty(0,T_0;L^2(\Omega))}+\|\nabla^2(c_\eps -c_A)\|_{L^2(\Omega\times(0,T_0))} &\leq R\eps^{\order-\frac32}
\end{align}
\end{subequations}
and
\begin{equation}
  \label{eq:convVelocityb}
  \|\ve_\eps -\ve_A\|_{L^\infty(0,T_0;L^2(\Om))}+ \|\ve_\eps -\ve_A\|_{L^2(0,T_0;H^1(\Om))} \leq C(R)\eps^{N+\frac12} 
\end{equation}
 hold true for all $\eps \in (0,\eps_0]$ and some $C(R)>0$. {Here $\Gamma(\delta), \Gamma(2\delta)$ are as in Section~\ref{subsec:Coordinates} below for some $\delta>0$ depending only on $\Gamma$.} Moreover,
\begin{equation*}
  \lim_{\eps\to 0} c_A= \pm 1 \qquad \text{uniformly on compact subsets of } \Omega^\pm= \bigcup_{t\in [0,T_0]} \Omega^\pm (t) \times \{t\}
\end{equation*}
and
\begin{equation*}
  \ve_A=\ve_0^\pm + O(\eps) \qquad \text{in }L^\infty(\Om\times (0,T_0))\text{ as }\eps\to 0.
\end{equation*}
\end{thm}
\begin{rem}
In principle one can get convergence in $L^\infty(\Om\times (0,T_0))$ and any $C^k$-norm, $k\geq 1$, if one chooses $N$ sufficiently large and uses an interpolation argument as in the proof of Theorem 2.3 in \cite{AlikakosLimitCH}.
\end{rem}
\begin{rem}
  Here  $c_{A,0}= c_A|_{t=0}$ and $\ve_{A,0}= \ve_A|_{t=0}$, where the construction of $(c_A, \ve_A)$ is discussed in Section~\ref{sec:ApproxSolutions} below.
  In particular we will have $c_{A,0}= c_{A,0}^0 + O(\eps^2)$ with
\begin{equation*}
  \begin{split}
    c_{A,0}^0(x)&=\zeta(d_{\Gamma_0}(x))\theta_0\left(\tfrac{d_{\Gamma_0}(x)}\eps\right)+(1-\zeta(d_{\Gamma_0}(x)))\left(
  \chi_{\Omega^+(0)}(x)-\chi_{\Omega^-(0)}(x)\right) \quad \text{for all }x\in \Om,
  \end{split}
\end{equation*}
where $d_{\Gamma_0}=d_\Gamma|_{t=0}$ is the signed distance function to $\Gamma_0$ and $\zeta\in C^\infty(\R)$ such that
\begin{equation}\label{eq:1.34}
  \zeta(z)=1~\text{if}~|z|\leq\delta; ~\zeta(z)=0~\text{if}~|z|\geq 2\delta;~ 0\leq  -z\zeta'(z) \leq 4~\text{if}~ \delta\leq |z|\leq 2\delta.
\end{equation}
Finally, we remark that every sufficiently smooth solution of \eqref{eq:NSAC1}-\eqref{eq:NSAC5} satisfies the energy identity
\begin{equation*}
  \frac{d}{dt}\left(\int_\Omega \tfrac12{|\ve_\eps|^2} \,dx
    +\int_\Omega (\tfrac{\eps}2|\nabla c_\eps|^2 + \tfrac1\eps f(c_\eps))\, dx   \right) = - \int_\Omega(|D\ve_\eps|^2 + \tfrac1\eps|\mu_\eps|^2)\,dx
\end{equation*}
for all $t\in (0,T_0)$, where $\mu_\eps = -\eps \Delta c_\eps + \tfrac1{\eps} f'(c_\eps) $. In particular,
\begin{align}\nonumber
  &\sup_{t\in [0,T_0]}\int_\Omega \tfrac12{|\ve_\eps(t)|^2} \,dx
    +\int_\Omega (\tfrac{\eps}2|\nabla c_\eps(t)|^2 + \tfrac1\eps f(c_\eps(t)))\, dx\\\label{eq:EnergyEstim}
  &\qquad\qquad  + \int_0^{T_0}\int_\Omega (|D\ve_\eps|^2 + \tfrac1\eps|\mu_\eps|^2)\, dx \, dt \leq E_{0,\eps},
  \end{align}
  where
  \begin{equation*}
    E_{0,\eps} := \int_\Omega \tfrac12{|\ve_{0,\eps(t)}|^2} \,dx
    +\int_\Omega (\tfrac{\eps}2|\nabla c_{0,\eps}(t)|^2 + \tfrac1\eps f(c_{0,\eps}(t)))\, dx.
  \end{equation*}
  Hence the left-hand side in \eqref{eq:EnergyEstim} is uniformly bounded in $\eps \in (0,1)$ if $\sup_{\eps\in (0,1)} E_{0,\eps}<\infty$. From the form of $c_A,\ve_A$ given in Section~\ref{sec:ApproxSolutions} and a Taylor expansion of $f(c_\eps)$ it is easy to see that this is the case under the assumption \eqref{initial assumption}.
\end{rem}

Compared to \cite{StokesAllenCahn}  we consider the coupling of the Allen-Cahn and the Navier-Stokes
system with general, possibly non-equal viscosities $\nu^+,\nu^-, \nu(c)>0$. This brings new  difficulties coming from the the nonlinear terms $\ve_\eps\cdot \nabla \ve_\eps$ and $-\Div (2\nu(c_\eps)D\ve_\eps)$. In particular, estimates of errors in the velocity $\ve_\eps$ are more involved {since the non-constant viscosity limits them to estimates in $L^\infty(0,T_0;L^2(\Omega)^d)\cap L^2(0,T_0;H^1(\Omega)^d)$ obtained by an energy method, cf.~Theorem~\ref{thm:weEstim} below. E.g.\ estimates in $L^r(0,T_0;L^q(\Omega))$ as in \cite[Proposition~3.6]{StokesAllenCahn} are not available.} Moreover, in \cite{StokesAllenCahn} only the convergence rate $N=2$ was considered and smallness of the time intervall had to be assumed. Using an approximation of higher order $N\geq 3$ we are able to prove the result as long as a smooth solution of the limit system exists. In order to deal with these difficulties, we will construct  suitable approximate solution to \eqref{eq:NSAC1}-\eqref{eq:NSAC3} with the method of matched asymptotic expansions used in \cite{StokesAllenCahn} and \cite{AbelsMarquardt1,AbelsMarquardt2}, see Theorem~\ref{thm:approx} below. But there is a crucial difference in the construction compared to the latter contribution concerning the highest order term of the height function $h_{N+\frac12}$ needed to dealing with the error in the velocity in \eqref{eq:NSAC3}. To this end we use a novel ansatz, which does not follow standard ways for formally matched asymptotics. The idea is to use an ansatz based on linearization for the term instead of including it in the stretched variable, cf. Section~\ref{sec:ApproxSolutions} below for details.
This new ansatz simplifies a lot of technical difficulties compared to \cite{StokesAllenCahn} as well as \cite{AbelsMarquardt1,AbelsMarquardt2} since the most critical terms related to $h_{N+\frac12}$ appear only linearly. It might be helpful for future results on sharp interface limits as well.
Furthermore, we use ideas from \cite{AbelsMarquardt1,AbelsMarquardt2} to improve the estimate for the error in the velocity by a factor $\eps^{\frac12}$ compared to \cite{StokesAllenCahn}.

 The structure of this contribution is as follows: In Section 2 we will give several
preliminary results concerning  local coordinates near $\Gamma_t$, the definition of stretched variable, parabolic equations on evolving hypersurfaces, and some kind
of spectral estimate for Allen-Cahn operator uniformly in small $\varepsilon$. The approximate solution is constructed by the method of matched asymptotic expansions and a novel ansatz in the critical order in
Section 3. Finally, the main
result is proven in Section 4. Some lengthy but straight forward calculations related
to the matched asymptotic expansions are given in details in the appendix.

\section{Preliminaries}\label{sec:Prelim}
\subsection{Coordinates}\label{subsec:Coordinates}

We use the notation of \cite{StokesAllenCahn} and parametrize $(\Gamma_t)_{t\in[0,T_0]}$ with the aid of a family of smooth diffeomorphisms $X_0\colon \mathbb{T}^1\times [0,T_0]\to \Om$ such that $\partial_s X_0(s,t)\neq 0$ for all $s\in\T^1$, $t\in [0,T_0]$.
 Moreover, let
 \begin{equation*}
\btau(s,t)= \frac{\partial_{s} X_0(s,t)}{|\partial_{s} X_0(s,t)|},\quad \text{and}\quad \no(s,t)=
\begin{pmatrix}
  0 & -1\\ 1 & 0
\end{pmatrix}
\btau(s,t)
\end{equation*}
be the normalized tangent and normal vectors on $\Gamma_t$ at $X_0(s,t)$, where we choose the orientation of $\Gamma_t$ (induced by $X_0(\cdot,t)$) such that $\no(s,t)$ is the exterior normal with respect to $\Omega^-(t)$. Furthermore, we define
\begin{alignat}{2}\label{eq:1.58}
\no_{\Gamma_t}(x)&:= \no (s,t)&\quad &\text{for all}~ x=X_0(s,t)\in \Gamma_t,\\\label{eq:1.57}
   V(s,t)&:= V_{\Gamma_t}(X_0(s,t)),\ H(s,t):= H_{\Gamma_t}(X_0(s,t))&\quad& \text{for all }s\in\T^1, t\in [0,T_0],
 \end{alignat}
 where $V_{\Gamma_t}$ and $H_{\Gamma_t}$ are the normal velocity and (mean) curvature of $\Gamma_t$ (with respect to $\no_{\Gamma_t}$). In particular,
\begin{equation*}
  V_{\Gamma_t}(X_0(s,t))=V(s,t)= \partial_t X_0(s,t)\cdot \no(s,t)\qquad \text{for all }(s,t)\in \T^1\times [0,T_0].
\end{equation*}
We use tubular neighborhoods of $\Gamma_t$: For $\delta>0$ sufficiently small,  the orthogonal projection $P_{\Gamma_t}(x)$ of all
\begin{equation*}
x\in \Gamma_t(3\delta) =\{y\in \Omega: \dist(y,\Gamma_t)<3\delta\}
\end{equation*}
is well-defined and smooth for all $t\in[0,T_0]$. We choose $\delta$ so small that $\dist(\partial\Omega,\Gamma_t)>3\delta$ for every $t\in [0,T_0]$. Every $x\in\Gamma_t(3\delta)$ has a unique decomposition
\begin{equation*}
x=P_{\Gamma_t}(x)+r\no_{\Gamma_t}(P_{\Gamma_t}(x)),
\end{equation*}
 where $r=\sdist(\Gamma_t,x)$. Here
\begin{equation*}
  d_{\G}(x,t):=\sdist (\Gamma_t,x)=
  \begin{cases}
    \dist(\Omega^-(t),x) &\text{if } x\not \in \Omega^-(t),\\
    -\dist(\Omega^+(t),x) &\text{if } x \in \Omega^-(t).
  \end{cases}
\end{equation*}
For the following we define for $\delta'\in (0,3\delta]$
\begin{equation*}
  \Gamma(\delta') =\bigcup_{t\in [0,T_0]} \Gamma_t(\delta') \times\{t\}, \qquad \Omega^\pm = \bigcup_{t\in [0,T_0]} \Omega^\pm (t)\times \{t\}.
\end{equation*}
 We will frequently use
\begin{equation*}
  \int_{\Gamma_t(\delta')} f(x)\,\sd x = \int_{-\delta'}^{\delta'}\int_{\Gamma_t} f(p+r\no_{\Gamma_t}(p))J(r,p,t)\sd \sigma(p)\sd r
\end{equation*}
for any $\delta'\in (0,3\delta]$, where $J\colon (-3\delta,3\delta)\times \Gamma \to (0,\infty)$ is a smooth function depending on $\Gamma$.
Furthermore, we use new coordinates in  $\Gamma(3\delta)$ with the aid of the mapping
\begin{equation*}
  X\colon  (-3\delta, 3\delta)\times \T^1 \times [0,T_0]\mapsto \Gamma(3\delta)~\text{by}~  X(r,s,t):= X_0(s,t)+r\no(s,t),
\end{equation*}
where
\begin{equation}\label{eq:1.42}
  r=\sdist(\Gamma_t,x), \qquad s= X_{0}^{-1}(P_{\Gamma_t}(x),t)=: S(x,t).
\end{equation}
Moreover, we use
  \begin{equation}\label{eq:1.26}
    \nabla d_{\G}(x,t)=\no_{\Gamma_t} (P_{\Gamma_t}(x)),~  \partial_t d_{\G}(x,t)=-V_{\Gamma_t} (P_{\Gamma_t}(x)),~\Delta d_\Gamma(q,t)=-H_{\Gamma_t}(q)
\end{equation}
for all $(x,t)\in \Gamma(3\delta)$, $(q,t)\in\Gamma$, resp., cf. Chen et al.~\cite[Section~4.1]{ChenHilhorstLogak}, and define
\begin{equation}\label{eq:1.50}
\partial_{\btau} u(x,t):= \btau(S(x,t),t)\nabla_x u(x,t),\quad   \nabla_\btau u(x,t):= \partial_{\btau} u(x,t)\btau(S(x,t),t)\quad
 \end{equation}
for all $(x,t)\in \Gamma(3\delta)$.

In the following we associate a function $\phi(x,t)$ to $\tilde{\phi}(r,s,t)$ such that
 \begin{equation}\label{eq:1.4}
   \phi(x,t)=\tilde{\phi}(d_{\G}(x,t),S(x,t),t)\quad\text{or}\quad\phi(X_0(s,t)+r\no(s,t),t)=\tilde{\phi}(r,s,t).
 \end{equation}
As in \cite{StokesAllenCahn} we have
\begin{equation}\label{Prelim:1.13}
  \begin{split}
    \partial_t \phi(x,t) &= -V_{\Gamma_t} (P_{\Gamma_t}(x)) \partial_r\tilde{\phi}(r,s,t) + \partial_{t}^\Gamma \tilde{\phi}(r,s,t), \\
  \nabla \phi(x,t) &= \no_{\Gamma_t} (P_{\Gamma_t}(x)) \partial_r\tilde{\phi}(r,s,t) + \nabla^ \Gamma  \tilde{\phi}(r,s,t), \\
 \Delta \phi(x,t) &= \partial_r^2\tilde{\phi}(r,s,t) + \Delta d_{\G_t}(x) \partial_r\tilde{\phi}(r,s,t) +  \Delta^{\Gamma} \tilde{\phi}(r,s,t),
  \end{split}
\end{equation}
where $r,s$ are as in \eqref{eq:1.42} and we use the notation
\begin{equation}\label{Prelim:1.12}
  \begin{split}
    \partial_{t}^\Gamma \tilde{\phi}(r,s,t) &= \partial_t \tilde{\phi}(r,s,t) + \p_t S(x,t)\cdot\partial_s \tilde{\phi}(r,s,t) ,\\
\nabla^{\Gamma} \tilde{\phi}(r,s,t) &=   \nabla  S(x,t) \p_{s} \tilde{\phi}(r,s,t),\\
\Delta^{\Gamma} \tilde{\phi}(r,s,t) &=   (\Delta S)(x,t)\cdot\partial_s \tilde{\phi}(r,s,t)+|\nabla S(x,t)|^2 \p_{s}^2  \tilde{\phi}(r,s,t) \  ,
  \end{split}
\end{equation}
 cf.\ \cite[Section~4.1]{ChenHilhorstLogak} for more details.
In \eqref{Prelim:1.12} $x$ is understood via $x=\no(s,t)r+X_0(s,t)$.
Note that $\nabla^\G g$ is a function of $(r,s,t)$:
\begin{equation}\label{eq:1.46}
  \nabla^\G g(r,s,t)=(\nabla S)(x,t)\partial_s g(s,t), \qquad \text{where }x=X(r,s,t).
\end{equation}
Therefore we  define for every $h\colon \T^1\times [0,T_0]\to \R$
\begin{equation}\label{eq:1.27}
\begin{split}
  (\nabla_\G h)(s,t):=(\nabla^\Gamma h)(0,s,t),\\
  (\Delta_\G h)(s,t):=(\Delta^\G h)(0,s,t),\\
  (D_t h)(s,t):=(\p_t^\G h)(0,s,t),
\end{split}
\end{equation}
and 
\begin{equation}\label{Prelim:1.11}
\begin{split}
  (\Lgrad h)(r,s,t):=(\nabla^\Gamma h)(r,s,t)-(\nabla_\Gamma h)(s,t),\\
   (\Ldelta h)(r,s,t):=(\Delta^\Gamma h)(r,s,t)-(\Delta_\Gamma h)(s,t),\\
  (\Lt h)(r,s,t):=(\p_t^\G h)(r,s,t)-(D_t h)(s,t),
 \end{split}
\end{equation}
where the coefficients of the latter operators vanish for $r=0$, which corresponds to $x\in\Gamma_t$.

Finally, we denote
\begin{alignat*}{2}
 (X_0^\ast u)(s,t)&:= u(X_0(s,t),t) &\qquad& \text{for all }s\in\T^1,t\in[0,T_0],\\
  (X_0^{\ast,-1} v)(p,t) & := v(X_0^{-1}(p,t),t) &\qquad& \text{for all }(p,t)\in\Gamma
\end{alignat*}
 if $u\colon \Gamma\to \R^N$  and $v\colon \G_0\times [0,T_0]\to \R^N$  for some $N\in\N$.

\begin{lem} \label{eq:3.6}
For any  $u\in H^1_0(\Gamma_t(2\delta))$, $\ve\in H^1_0(\Gamma_t(2\delta);\R^2)$ we have
\begin{equation*}
[\partial_\no, \nabla_\btau] u=\btau (\partial_{\btau} \no\cdot \nabla u)
\end{equation*}
and
   \begin{equation}\label{eq:1.23}
   \int_{\Gamma_t(2\delta)} u \Div_\btau  \ve\sd x= - \int_{\Gamma_t(2\delta)}  \nabla_\btau u\cdot \ve\sd x - \int_{\Gamma_t(2\delta)} \kappa \no\cdot \ve u\sd x,
 \end{equation}
 where $\kappa=-\Div \no$ and
 \begin{equation*}
  \nabla_\btau=(I-\no(S(\cdot),\cdot)\otimes\no(S(\cdot),\cdot))\nabla.
\end{equation*} \end{lem}
\begin{proof}
This lemma is a consequence of \cite[Corollary 2.7]{StokesAllenCahn}.  
\end{proof}

\subsection{Function Spaces}

 We denote by $L^p(U)$ the usual Lebesgue space with respect to the Lebesgue measure and the $L^p$-Sobolev space of order $m\in\N_0$ on an open set $U\subseteq\R^N$ is denoted by $W^m_p(U)$. Furthermore, $H^s(U)$ is the $L^2$-Sobolev space of order $s\in\R$ and $H^s_0(U)$ is the closure of $C_0^\infty(U)$ in $H^s(U)$. The $X$-valued variants are denoted by $W^m_p(U;X)$, $L^p(U;X)$, and $H^s(U;X)$, respectively.
As in \cite{StokesAllenCahn} we define
\begin{alignat*}{1}
  L^{p,\infty}(\Gamma_t(2\delta))&:=\left\{ f\colon \Gamma_t(2\delta)\to \R \text{ measurable }:\|f\|_{L^{p,\infty}(\Gamma_t(2\delta))}<\infty \right\},\quad \text{where }\\
  \|f\|_{L^{p,\infty}(\Gamma_t(2\delta))}&:=\left(\int_{\T^1}\operatorname{ess\, sup}_{|r|\leq 2\delta } |f(X_0(s,t)+r\no(s,t))|^p \sd s\right)^{\frac1p}.
\end{alignat*}

We will often use the embedding
\begin{equation} \label{L4inf} 
  H^1(\Gamma_t(2\delta))\hookrightarrow L^{4,\infty}(\Gamma_t(2\delta)),
\end{equation}
which is a consequence of the interpolation inequality
\begin{equation*}
  \|f\|_{L^\infty(-2\delta,2\delta)}\leq C\|f\|_{L^2(-2\delta,2\delta)}^{\frac12}\|f\|_{H^1(-2\delta,2\delta)}^{\frac12}\quad \text{for }f\in H^1(-2\delta,2\delta).
\end{equation*}

\subsection{The Stretched Variable and Remainder Terms}

In the following we will use a ``strechted variable'', which is defined by
\begin{equation}\label{eq:StretchedVariable}
  \rho= \frac{d_\Gamma(x,t)}\eps- h_\eps (s,t) \qquad \text{for } (x,t)\in \Gamma(3\delta), \eps \in (0,\eps_0),
\end{equation}
where $s=S(x,t)$ as in \eqref{eq:1.42}. Here we assume that $h_\eps\colon \T^1\times [0,T]\to \R$ is given and (sufficiently) smooth with bounded $C^k$-Norms for sufficiently large $k\in\N$ uniformly in $\eps\in (0,\eps_0)$ for some $\eps_0>0$.

As in \cite[Section~4.2]{ChenHilhorstLogak} a Taylor expansion of $\Delta d_\Gamma$ in the normal direction yields 
    \begin{align}\nonumber
      &\Delta d_{\Gamma}(x,t)= -H_{\Gamma_t}(s)-\eps d_\Gamma(x,t)\kappa_1(s,t)
        + \sum_{k=2}^{K-1}\kappa_i(s,t)d_\Gamma(x,t)^k+d_\Gamma(x,t) \tilde{\kappa}_{K}(x,t)\\\label{eq:ex1}
      &=-H_{\Gamma_t}(s)-\eps(\rho+h_\eps(s,t))\kappa_1(s,t)
        + \sum_{k=2}^{K-1}\eps^k \kappa_i(s,t)(\rho+h_\eps(s,t))^k+\eps^K \kappa_{K,\eps}(\rho,s,t),
    \end{align}
where $K$ is some large integer, $s$ is understood via \eqref{eq:1.42} and $$\kappa_1(s,t)=-\nabla d_{\Gamma}(X_0(s,t))\cdot\nabla\Delta d_{\Gamma}(X_0(s,t)),$$ $\{\kappa_j(s,t)\}_{2\leq j\leq K-1}$ are smooth
and $\kappa_{K,\eps}$ is a smooth function satisfying
\begin{equation}\label{eq:1.44}
  |\kappa_{K,\eps}(\rho,s,t)|\leq C|\rho+h_\eps(s,t)|^K\qquad \text{for all }\rho\in \R, s\in \mathbb{T}^1, t\in [0,T_0], \eps \in  (0,1).
\end{equation}
The following lemma follows from the chain rule and \eqref{Prelim:1.13}, cf.\ \cite[Section~4.2]{ChenHilhorstLogak}:
\begin{lem}\label{lem:ChainRule}
Let ${\hat{w}}\colon \R\times \Omega\times [0,T_0]\to \R$ be sufficiently smooth and let
\begin{equation*}
w(x,t)={\hat{w}}\(\rho(x,t),x,t\) \quad \text{for all }(x,t)\in \Gamma(3\delta).
\end{equation*}
Then for each $\eps>0$
  \begin{equation}\label{eq:formula1}
  \begin{split}
    \p_t w(x,t)=&-\(\tfrac{V_{\Gamma_t} (P_{\Gamma_t}(x))}\eps +\p_t^\Gamma h_\eps(r,s,t)\)\p_\rho {\hat{w}}(\rho,x,t) +\p_t {\hat{w}}(\rho,x,t),\\
    \nabla w(x,t)=& \(\tfrac{\no_{\Gamma_t} (P_{\Gamma_t}(x))}  \eps -\nabla^\Gamma h_\eps(r,s,t)\)\p_\rho{\hat{w}}(\rho,x,t) +\nabla{\hat{w}}(\rho,x,t),\\
    \Delta w(x,t)=& (\eps^{-2}+|\nabla^\Gamma h_\eps(r,s,t)|^2) \p^2_\rho{\hat{w}}(\rho,x,t)\\
    &+\(\eps^{-1}\Delta d_\Gamma (x,t) -\Delta^\Gamma h_\eps(r,s,t)\)\p_\rho{\hat{w}}(\rho,x,t)\\
    &+2\left({\frac{\no_{\Gamma_t}}\eps}-\nabla^\Gamma h_\eps(r,s,t)\right)\cdot\nabla \p_\rho{\hat{w}}(\rho,x,t)+\Delta {\hat{w}}(\rho,x,t),
  \end{split}
\end{equation}
where $\rho$ is as in \eqref{eq:StretchedVariable} and $(r,s)$ is understood via \eqref{eq:1.42}.
\end{lem}

For a systematic treatment of the remainder terms, we introduce:
\begin{defn}\label{eq:1.15}
  For any $k\in \R$ and $\alpha>0$,  $\mathcal{R}_{k,\alpha}$ denotes the vector space of families of continuous functions $\tr_\eps\colon \R\times \Gamma(2\delta) \to \R$, indexed by $\eps\in (0,1)$, which are continuously differentiable with respect to $\no_{\Gamma_t}$ for all $t\in [0,T_0]$ such that
  \begin{equation}\label{eq:EstimRkalpha}
    |\partial_{\no_{\Gamma_t}}^j \tr_\eps(\rho,x,t)|\leq Ce^{-\alpha |\rho|}\eps^k\qquad \text{for all }\rho\in \R,(x,t)\in\Gamma(2\delta), j=0,1, \eps \in (0,1)
  \end{equation}
for some $C>0$ independent of $\rho\in \R,(x,t)\in\Gamma(2\delta)$, $\eps\in (0,1)$.  $\mathcal{R}_{k,\alpha}^0$ is the subclass of all $(\tr_\eps)_{\eps\in (0,1)}\in \mathcal{R}_{k,\alpha}$ such that
$\tr_\eps(\rho,x,t)= 0$ for all  $\rho \in\R, x\in\Gamma_t, t\in [0,T_0]$.
\end{defn}
We remark that $\mathcal{R}_{k,\alpha}$ and $\mathcal{R}^0_{k,\alpha}$ are closed under multiplication and $\mathcal{R}_{k,\alpha}\subset \mathcal{R}_{k-1,\alpha}$.

\begin{lem}\label{lem:rescale}
  Let $0<\eps\leq \eps_0$, $h_\eps$ be defined by \eqref{eq:heps} and satisfy
  \begin{equation*}
    M:=\sup_{0<\eps <\eps_0, (s,t)\in \T^1\times [0,T_\eps]} |h_\eps (s,t)| <\infty
  \end{equation*}
for some  $T_\eps \in (0,T_0]$, $\eps_0\in (0,1)$,  and   $(\tr_\eps)_{0<\eps<1}\in \mathcal{R}_{k,\alpha}$ for some $\alpha>0$, $k\in\R$ and let $j=1$ if even $(\tr_\eps)_{0<\eps<1}\in \mathcal{R}_{k,\alpha}^0$ and $j=0$ else.
  Then there is some $C>0$, independent of $T_\eps,0<\eps\leq \eps_0$, $\eps_0\in (0,1)$  such that
    \begin{equation*}
      r_\eps (x,t):= \tr_\eps\left( \rho, x,t\right)\qquad \text{for all }(x,t)\in\Gamma(2\delta)
    \end{equation*}
  with $\rho$ as in \eqref{eq:StretchedVariable} satisfies
  \begin{align}\label{eq:RemEstim1}
    \left\|{\mathfrak{a}(P_{\Gamma_t}(\cdot))r_\eps \varphi} \right\|_{L^1(\G_t(2\delta))} &\leq C(1+M)^j\eps^{1+k+j}\|\varphi\|_{H^1(\Omega)}\|\mathfrak{a}\|_{L^2(\Gamma_t)},\\
  \label{eq:RemEstim2}
    \left\| {\mathfrak{a}(P_{\Gamma_t}(\cdot))r_\eps} \right\|_{L^2(\G_t( 2\delta))} &\leq C (1+M)^j\eps^{\frac 12+k+j} \|\mathfrak{a}\|_{L^2(\Gamma_t)}
  \end{align}
  uniformly for all $\varphi\in H^1(\Omega)$, $\mathfrak{a}\in L^2(\Gamma_t)$, $t\in [0,T_\eps]$, and $\eps\in (0,\eps_0]$.
\end{lem}
We refer to \cite[Corollary 2.7]{StokesAllenCahn} for the proof.

\begin{lem}\label{lem:DivergenceFreeRemainder}
  Let $f\in \SD(\R)$ such that $f'\in \mathcal{R}_{0,\alpha}$ for some $\alpha>0$. Then there is a constant $C>0$ such that for all $t\in [0,T_0]$, $a\in H^1(\T^1)$ and $\boldsymbol{\varphi}\in C_{0,\sigma}^\infty(\Omega)$ we have
  \begin{equation*}
    \left|\int_{\Gamma_t(2\delta)} f'(\rho(x,t))a(S(x,t))\no_{\Gamma_t}\otimes \no_{\Gamma_t} : \nabla \boldsymbol{\varphi} \, dx\right|\leq C \eps^{\frac32} \|a\|_{H^1(\T^1)}\|\boldsymbol{\varphi}\|_{H^1(\Omega)}.
  \end{equation*}
\end{lem}
\begin{proof}
  First of all, since $h_\eps$ is uniformly bounded, there is some $\eps_0\in (0,1]$ such that
  \begin{equation}\label{eq:expab}
    \left|\frac{\pm \delta}\eps - h_\eps(s,t)\right|\geq \frac{\delta}{2}\qquad \text{for all }s\in\T^1, t\in [0,T_0], \eps \in (0,\eps_0].
  \end{equation}
  We use that {$f'(\rho)= \eps \partial_{\no}f(\rho)$ and }$\no_{\Gamma_t}\otimes \no_{\Gamma_t} : \nabla \boldsymbol{\varphi}= - \Div_{\btau} \boldsymbol{\varphi}$ since $\Div \boldsymbol{\varphi}=0$.
  To treat the remaining integral, we may use Lemma~\ref{eq:3.6} 
to get
\begin{align*}
  &\left|\int_{\Gamma_t(2\delta)} f'(\rho(x,t))a(S(x,t))\no_{\Gamma_t}\otimes \no_{\Gamma_t} : \nabla \boldsymbol{\varphi} \, dx\right|\\
  & \leq \left|\int_{\Gamma_{t}(2\delta)}\eps \nabla_{\btau}\left(\partial_{\mathbf{n}}f(\rho(x,t))a(S(x,t))\right)\cdot\boldsymbol{\varphi}\, dx\right|
 +\left|\int_{\Gamma_{t}(2\delta)}\eps \partial_{\mathbf{n}}{f}(\rho)a(S(x,t))\boldsymbol{\varphi}\cdot\mathbf{n}\kappa(x,t)\, dx\right|\\
 & \quad+C\sum_{\pm}\int_{\mathbb{T}^{1}}\left|f'\left(\frac{\pm\delta}{\epsilon}-h_{\eps}(s,t)\right)a(s)\boldsymbol{\varphi}\left(X_0(\pm \delta,s,t)\right)\right|\, ds\\
 & :=J_{1}+J_{2}+J_{3}^++J_{3}^-.
\end{align*}
Now
\begin{align*}
{J}_{3}^{\pm} & \leq C_1e^{-\frac{\alpha \delta}{2\epsilon}}\int_{\mathbb{T}^{1}}\left|a(s)\right|\sup_{r\in\left[-\delta,\delta\right]}\left|\boldsymbol{\varphi}(X_0(r,s,t))\right|\, ds \leq C\epsilon^{\frac{3}{2}}\|a\|_{L^2(\T^1)}\left\Vert \boldsymbol{\varphi}\right\Vert _{H^{1}(\Gamma_{t}(2\delta))},
\end{align*}
where we used (\ref{eq:expab}) and $H^{1}(\Gamma_{t}(2\delta))\hookrightarrow L^{2,\infty}(\Gamma_{t}(2\delta))$
(cf. \eqref{L4inf}).  For the second term we use integration
by parts and get
\begin{align*}
{J}_{2} & \leq\left|\eps\int_{\Gamma_{t}(2\delta)}f(\rho) a (S(x,t))\partial_{\mathbf{n}}\left(\boldsymbol{\varphi}\kappa(x,t) \right)\cdot\mathbf{n}\left(S(x,t),t\right)\, dx\right|
+Ce^{-\frac{\alpha \delta}{2\epsilon}}\|a\|_{L^2 (\T^1)}\left\Vert \boldsymbol{\varphi}\right\Vert _{H^{1}(\Gamma_{t}(2\delta))}\\
 & \leq C\eps\left\Vert a\right\Vert _{L^{2}(\mathbb{T}^{1})}\left\Vert\boldsymbol{\varphi}\right\Vert _{H^{1}\left(\Gamma_{t}(2\delta)\right)}\epsilon^{\frac{1}{2}}\| f\| _{L^{2}\left(\mathbb{R}\right)}+Ce^{-C_{2}\frac{\delta}{2\eps}}\|a\|_{L^2 (\T^1)}\left\Vert \boldsymbol{\varphi}\right\Vert _{H^{1}(\Gamma_{t}(2\delta))}\\
 & \leq C\eps^{\frac{3}{2}}\left\Vert a\right\Vert _{L^{2}(\mathbb{T}^{1})}\left\Vert\boldsymbol{\varphi}\right\Vert _{H^{1}\left(\Gamma_{t}(2\delta)\right)},
\end{align*}
where the exponentially decaying term comes from the appearing boundary integral, which is estimated as before. Moreover, we used a change of
variables $r\mapsto\frac{r}{\epsilon}-h_{\epsilon}$.

Finally, we have
\begin{align*}
{J}_{1} & \leq\left|\eps \int_{\Gamma_{t}(2\delta)}\partial_{\mathbf{n}}\nabla_{\btau}\left(f\left(\rho(x,t)\right)a(S(x,t))\right)\cdot\boldsymbol{\varphi}(x)\, dx\right|\\
 & \quad+\left|\eps\int_{\Gamma_{t}(2\delta)}\left[\partial_{\mathbf{n}},\nabla_{\btau}\right]f(\rho(x,t))a(S(x,t))\cdot\boldsymbol{\varphi}(x)\, dx\right|\\
        & \leq C\left|\eps \int_{\Gamma_{t}(2\delta)}\nabla_{\btau}\left(f\left(\rho(x,t)\right)a(S(x,t))\right)\cdot\partial_{\mathbf{n}}\boldsymbol{\varphi}(x)\, dx\right|\\
  & \quad + C\eps^{\frac32}\|(f,f')\|_{L^2(\R)}\|a\|_{H^1(\T^1)}\|\boldsymbol{\varphi}\|_{L^2(\Gamma_t(2\delta))} +Ce^{-C_{6}\frac{\delta}{2\eps}}\left\Vert a\right\Vert _{L^2(\mathbb{T}^{1})}\left\Vert \boldsymbol{\varphi}\right\Vert _{H^{1}\left(\Gamma_{t}(2\delta)\right)}\\
 & \leq C\eps^{\frac{3}{2}}\left\Vert a\right\Vert _{H^{1}(\mathbb{T}^{1})}\left\Vert \boldsymbol{\varphi}\right\Vert _{H^{1}\left(\Gamma_{t}(2\delta)\right)}.
\end{align*}
Here we used that $\left[\partial_{\mathbf{n}},\nabla_\btau \right]$ is a differential operator in tangential direction and integration
by parts.
This finishes the proof.
\end{proof}

\subsection{Parabolic Equations on Evolving Hypersurfaces}\label{subsec:ParabolicEq}

Let $0<T<\infty$ be arbitrary.
We define
\begin{equation}\label{eq:1.31}
  X_T:=L^2(0,T;H^{5/2}(\T^1))\cap H^1(0,T;H^{1/2}(\T^1))
\end{equation}
equipped with the norm
\begin{equation*}
  \|u\|_{X_T} =\|u\|_{L^2(0,T;H^{5/2}(\T^1))}+\|u\|_{H^1(0,T;H^{1/2}(\T^1))}+ \|u|_{t=0}\|_{H^{3/2}(\T^1)}.
\end{equation*}
We note that
\begin{equation}\label{eq:EmbeddingE1}
X_T\hookrightarrow BUC([0,T]; H^{3/2}(\T^1))\cap L^4(0,T; H^2(\T^1))
\end{equation}
{and the operator norm of the embedding is uniformly bounded in $T$.}
 
\begin{thm}\label{thm:ParabolicEqOnSurface}
  Let $w\colon \T^1\times [0,T] \to \Rn$ and $a\colon \T^1\times [0,T]\to \R$ be smooth.
  For every $g\in L^2(0,T;H^{\frac 12}(\T^1))$ and $h_0\in H^{ \frac32 }(\T^1)$ there is a unique solution $h\in X_T$ of
  \begin{alignat}{2}\label{eq:h1}
    D_t h+ w\cdot \nabla_\G h -\Delta_\Gamma h + a h &= g&\qquad& \text{on }\T^1\times [0,T],\\\label{eq:h2}
    h|_{t=0} &=h_0 && \text{on } \T^1.
  \end{alignat}
\end{thm}
Proof: See \cite{StokesAllenCahn}.

\subsection{Spectral Estimate}

In this subsection we assume that $\Omega\subseteq \R^2$ is a bounded domain and $\Gamma_t\subseteq \Om$, $t\in [0,T_0]$, $T_0>0$, are given smoothly evolving closed and compact $C^\infty$-hypersurfaces, dividing $\Om$ in disjoint domains $\Omega^+(t)$ and $\Omega^-(t)$ as before, and
\begin{alignat*}{2}
    c_A(x)&= c_{A,0}(x)+ \eps^2c_{A,2+}(x), &\quad& \text{for all } x\in \Om,\\
    c_{A,0}(x)&=\zeta\circ d_\Gamma \theta_0(\rho)+(1-\zeta\circ d_\Gamma)\left(
  \chi_{\Omega^+(t)}-\chi_{\Omega^-(t)} \right)&\quad& \text{for all } x\in \Om,
\end{alignat*}
where $\zeta\in C^\infty(\R)$ is as in \eqref{eq:1.34}. 
Moreover, we assume that $\dist(\Gamma_t,\partial\Omega)> 2\delta$ for all $t\in [0,T_0]$.
For given  continuous functions $(\tilde{h}_\eps)_{0<\eps< 1} \colon \G\to \R$ with $\Gamma:=\bigcup_{t\in[0,T_0]} \Gamma_t \times {\{t\}}$, we define the stretched variable $\rho$ as in \eqref{eq:StretchedVariable} with $h_\eps(s,t)= \tilde{h}_\eps (X_0(s,t),t)$.
 Furthermore, we assume that
\begin{equation}\label{eq:BoundhcA}
\sup_{\eps \in (0,1)}\left(\sup_{(p,t)\in \Gamma} |\tilde{h}_\eps (p,t)|+ \sup_{x\in\Om, t\in[0,T_0]} |c_{A,2+} (x,t)|\right)\leq M
\end{equation}
for some $M>0$.  We will apply the results of this subsection to $\tilde{h}_\eps (p,t)= h_\eps(X_0^{-1}(p,t),t)$ for some $h_\eps\colon \T^1\times [0,T_0]\to \R$. 

The following spectral estimate due to  \cite[Theorem~2.13]{StokesAllenCahn} is a key ingredient for the proof of convergence.
\begin{thm}\label{thm:Spectral}
Let $c_A$ be as above and \eqref{eq:BoundhcA} be satisfied for some $M>0$. Then there are some $C_L,\eps_0>0$, independent of $\tilde{h}_\eps, c_A$, such that for every $\psi\in H^1(\Om)$, $t\in[0,T_0]$, and $\eps\in (0,\eps_0]$ we have
  \begin{equation*}
    \int_{\Om}\left(|\nabla\psi(x)|^2+ \frac{f''(c_A(x,t))}{\eps^2}\psi^2(x)\right)\sd x\geq -C_L\int_\Om \psi^2 \sd x + \int_{\Om\setminus \Gamma_t(\delta)} |\nabla \psi|^2\sd x +  \int_{\Gamma_t({\delta})} |\nabla_\btau \psi|^2\sd x.
  \end{equation*}
\end{thm}

The following refinement will be essential for our proof as well.
\begin{cor}\label{cor:SpectralDecomp}
 Let the previous assumptions hold true and let
$t\in\left[0,T\right]$, let $\psi\in H^{1}(\Gamma_{t}(\delta))$
and $\Lambda_{\eps}\in\mathbb{R}$ be such that
\begin{equation}
\label{def Lambda energy}
\int_{\Gamma_{t}(\delta)}\eps\left|\nabla\psi(x)\right|^{2}+\eps^{-1}f''\left({c_{A}}(x,t)\right)\psi(x)^{2}\sd x\leq\Lambda_{\eps}
\end{equation}
and denote $I_{\eps}^{s,t}:=\left(-\frac{\delta}{\eps}-h_{{\eps}}(s,t),\frac{\delta}{\eps}-h_{{\eps}}(s,t)\right)$.
Then, for $\eps>0$ small enough, there exist functions $Z\in H^{1}(\mathbb{T}^{1})$,
$\psi^{\mathbf{R}}\in H^{1}(\Gamma_{t}(\delta))$
and smooth $\Psi\colon I_{\eps}^{s,t}\times\mathbb{T}^{1}\to \R$ such that
\begin{equation}
\psi\left(X(r,s,t)\right)=\eps^{-\frac{1}{2}}Z(s)\left(\beta(s)\theta_{0}'(\rho(r,s))+\Psi(\rho(r,s),s)\right)+\psi^{\mathbf{R}}(r,s)\label{decompose u}
\end{equation}
for almost all $\left(r,s\right)\in\left(-\delta,\delta\right)\times\mathbb{T}^{1}$,
where $\rho(r,s)=\frac{r}{\eps}-h_{{\eps}}(s,t)$
and $\beta(s)=\left(\int_{I_{\eps}^{s,t}}\left(\theta_{0}'(\rho)\right)^{2}\sd\rho\right)^{-\frac{1}{2}}$.
Moreover,
\begin{equation}
\Vert \psi^{\mathbf{R}}\Vert _{L^{2}\left(\Gamma_{t}(\delta)\right)}^{2}\leq C\left(\eps\Lambda_{\eps}+\eps^{2}\left\Vert \psi\right\Vert _{L^{2}\left(\Gamma_{t}(\delta)\right)}^{2}\right),\label{f2u estimate-1}
\end{equation}
\begin{equation}\label{f2u estimate-2}
\Vert Z\Vert _{H^{1}\left(\mathbb{T}^{1}\right)}^{2}+\Vert \nabla^{\Gamma}\psi\Vert _{L^{2}(\Gamma_{t}(\delta))}^{2}+\Vert \psi^{\mathbf{R}}\Vert _{H^{1}\left(\Gamma_{t}(\delta)\right)}^{2}\le C\left(\left\Vert \psi\right\Vert _{L^{2}\left(\Gamma_{t}(\delta)\right)}^{2}+\frac{\Lambda_{\eps}}{\eps}\right),
\end{equation}
and
\begin{equation}\label{f1u estimate}
\sup_{s\in\mathbb{T}^{1}}\left(\int_{I_{\eps}^{s,t}}\left(\Psi(\rho,s)^{2}+\Psi_{\rho}(\rho,s)^{2}\right)J\left(\eps(\rho+h_{{\eps}}(s,t)),s\right)\sd\rho\right)\leq C\eps^{2}.
\end{equation}
\end{cor}
We refer to \cite[Corollary~2.12]{AbelsMarquardt1}  for the proof and note that it is easy to verify that Assumption~2.11 in \cite{AbelsMarquardt1} is satisfied in our situation.
\begin{rem}\label{rem:Decomp}
  In the following we will apply Corollary~\ref{cor:SpectralDecomp} to $\psi = \Diff= \frac{c_\eps-c_A}{\|c_\eps-c_A\|_{L^2(\Gamma_t(\delta))}}$ satisfying \eqref{assumptions'} and
  \begin{equation*}
    \int_0^T\int_{\Gamma_t(\delta)} \left({\eps}|\nabla \Diff|^2 +\eps^{-1} f''(c_A(x,t)) \Diff^2\right)\, dx\, dt \leq R^2\eps^{2N+2}.
  \end{equation*}
  Then we obtain that
  \begin{equation}\label{eq:SpectralDecomp}
    \tilde{\Diff}(r,s)=  \eps^{-\frac12} Z(s) \Psi_1(\tfrac{r}\eps - h_\eps (s),s)+ {{\psi}^{\mathbf{R}}}(r,s),
  \end{equation}
  where $\Psi_1(\rho,s)= \beta(s) \theta'_0(\rho)+\Psi(\rho,s)$ and
  \begin{equation*}
    \|Z\|_{L^2(0,T;H^1(\Gamma_t))} + \|\nabla^\Gamma \Diff\|_{L^2(\Gamma(\delta))}+ \|{\psi^{\mathbf{R}}}\|_{L^2(0,T;H^1(\Gamma_t(\delta)))}\leq CR\eps^{N+\frac12},
  \end{equation*}
  and
\begin{equation*}
  \|{\psi^{\mathbf{R}}}\|_{L^2(\Gamma(\delta))}\leq C\eps^{N+3/2}.
\end{equation*}

\end{rem}

\begin{rem}
  For $u\in H^1(\Gamma_t(\delta))$ let us introduce the $\eps$-dependent norms
  \begin{align*}
    \|u\|_{X_\eps} = &\inf\left\{ \|Z\|_{H^1(\Gamma_t)}+ \|v\|_{H^1(\Gamma_t(\delta))}+\eps^{-1}\|v\|_{L^2(\Gamma_t(\delta))}:\right.\\
    &\quad \qquad \left. \tilde{u}(\rho,s)= Z(s) \eps^{-\frac12} \theta_0'(\rho) + \tilde{v}(\rho,s), Z\in H^1(\Gamma_t), v\in H^1(\Gamma_t(\delta)) \right\}.
  \end{align*}
  Then, {choosing $\Lambda_\eps$ such that equality holds in \eqref{def Lambda energy},} Corollary~\ref{cor:SpectralDecomp} yields
  \begin{equation*}
    \|u\|_{X_\eps}^2+\|\nabla^\Gamma u \|_{L^2(\Gamma_t(\delta))}^2 \leq C\left(\int_{\Gamma_t(\delta)} \left(|\nabla u|^2 +\frac1{\eps^2} f''(c_A(x,t))u^2\right)\, dx + \|u\|_{L^2(\Gamma_t(\delta))}^2 \right).
  \end{equation*}
\end{rem}
As a consequence we obtain
\begin{lem}\label{lem:EstimMeanValueFree}
  Let the assumptions above hold true and $f\colon \Gamma_t(3\delta)\to \R$ be such that
    \begin{equation*}
    f(x,t)=
      a(\rho,s,t) w(s) \qquad \text{in } \Gamma_t(3\delta),\quad \text{where } x=X(r,s,t), \rho = \tfrac{r}\eps - h_\eps (s,t),
    \end{equation*}
    with $w\in L^2(\T^1)$ and $a\in \mathcal{R}_{0,\alpha}$
    \begin{equation}\label{eq:MeanValueFree}
    \int_\R a(\rho,s,t) \theta'_0(\rho) \, d\rho=0\qquad \text{for all }s\in \T^1, t\in [0,T_0].
  \end{equation}
  Then there are constants $C(T_0)$, $c_0>0$, $\eps_0>0$ independent of $t\in [0,T_0]$ and $w$ such that
  \begin{equation*}
    \|f\|_{X_\eps'} \leq C \eps^{\frac32}\|w\|_{L^2(\T^1)}
  \end{equation*}
  for every $\eps\in (0,\eps_0)$.
\end{lem}
\begin{proof}
  Let $u \in X_\eps$ with $\|u\|_{X_\eps}\leq 1$ and $Z\in H^1(\T^1)$, $v\in H^1(\Gamma_t(3\delta))$ with
  \begin{equation*}
    u(x)=\tilde{u}(r,s)= Z(s) \eps^{-\frac12} \theta_0'(\rho) + \tilde{v}(r,s),\quad \text{where } x=X(r,s,t), \rho = \tfrac{r}\eps - h_\eps (s,t),
  \end{equation*}
  $v(x)= \tilde{v}(r,s)$, and
  \begin{equation*}
    \|Z\|_{H^1(\T^1)}+ \|v\|_{H^1(\Gamma_t(3\delta))}+\eps^{-1}\|v\|_{L^2(\Gamma_t(3\delta))}\leq 2.
  \end{equation*}
  Then, using $J(r,s,t)= J(0,s,t)+ r \tilde{J}(r,s,t)$,
  \begin{align*}
    &\int_{\Gamma_t(3\delta)} f(x,t) u(x)\, dx = \int_{\T^1}\int_{-3\delta}^{3\delta} a(\tfrac{r}\eps - h_{{\eps}}(s,t),s,t){w(s)}\tilde{u}(r,s) J(r,s,t)\, dr\, ds \\
    &= \eps \int_{\T^1}\int_{-\tfrac{3\delta}\eps - h_{{\eps}}(s,t)}^{\tfrac{3\delta}\eps - h_{{\eps}}(s,t)} a(\rho,s,t) w(s) Z(s) \eps^{-\frac12} \theta_0'(\rho)J(0,s,t)\, d\rho\, ds\\
    &\quad +\eps^2 \int_{\T^1}\int_{-\tfrac{3\delta}\eps - h_{{\eps}}(s,t)}^{\tfrac{3\delta}\eps - h_{{\eps}}(s,t)} a(\rho,s,t)Z(s) \eps^{-\frac12} \theta_0'(\rho) w(s) (\rho+\eps h_{{\eps}}(s,t)) \tilde{J}(\rho+\eps h_{{\eps}}(s,t),s,t) \, d\rho\, ds \\
      &\quad +  \int_{\T^1}\int_{-3\delta}^{3\delta} a(\tfrac{r}\eps - h_{{\eps}}(s,t),s,t) w(s) \tilde{v}(r,s)J(r,s,t)\, d\rho\, ds\\
    &\equiv I_1+I_2 +I_3.
  \end{align*}
  Moreover, using \eqref{eq:MeanValueFree}, we can estimate
  \begin{align*}
    |I_1| &=  \eps \int_{\T^1}\int_{|\rho|\geq \tfrac{2\delta}\eps} \left|a(\rho,s,t)\theta_0'(\rho)d\rho\, w(s) Z(s) \eps^{-\frac12} J(0,s,t)\right|\, ds \\
    &\leq C \eps^{\frac12} e^{{-\tfrac{2\delta\alpha}{\eps}}}\|Z\|_{L^2(\T^1)}\|w\|_{L^2(\T^1)}\leq C \eps^{\frac32} \|w\|_{L^2(\T^1)}.
  \end{align*}
  Furthermore,
  \begin{align*}
    |I_{{2}}| &\leq \eps^{\frac32}  \int_{\R} \sup_{s\in \T^1, t\in [0,T_0]} |a(\rho,s,t)\theta_0'(\rho)(\rho+\eps h_{{\eps}}(s,t))| d\rho\ \|w\|_{L^2(\T^1)}\|Z\|_{L^2(\T^1)}\\
    &\leq C \eps^{\frac32} \|w\|_{L^2(\T^1)}
  \end{align*}
  and
  \begin{align*}
    |I_3| &\leq   \left(\int_{\T^1}\int_{-3\delta}^{3\delta} |a(\tfrac{r}\eps - h_{{\eps}}(s,t),s,t) w(s)|^2 J(r,s,t)\, dr\, ds\right)^{\frac12} \|v\|_{L^2(\Gamma_t(3\delta))}\\
          &\leq   C\left(\int_{-2\delta}^{2\delta} \exp(-\tfrac{\alpha |r|}\eps)\, dr\right)^{\frac12} \|w\|_{L^2(\T^1)} \|v\|_{L^2(\Gamma_t(3\delta))}\\
          &\leq C \eps^{\frac12} \|w\|_{L^2(\T^1)}\|v\|_{L^2(\Gamma_t(3\delta))}\leq C' \eps^{\frac32} \|w\|_{L^2(\T^1)}.
  \end{align*}
  for all $\eps\in (0,\eps_0)$ for some $\eps_0>0$ sufficiently small.
\end{proof}

\section{Construction of the Approximate Solutions}\label{sec:ApproxSolutions}

The main goal of this section is to prove:
\begin{thm}\label{thm:approx}
  Let $\ue=\ue(\eps)\in L^2(0,T_\eps;H^1(\Omega)^d)$, $\eps\in (0,1)$, be given such that for some $T_\eps \in (0,T_0]$,  $M>0$ and $\eps_0>0$ we have
\begin{equation*}
  \|\ue \|_{L^2(0,T_\eps; H^1(\Omega))}\leq M\qquad \text{for all }\eps \in (0,\eps_0).
\end{equation*}
Then there are smooth $c_A,p_A\colon \Omega\times [0,T_\eps]\to \R$, $\ve_A\colon \Omega\times [0,T_\eps]\to \R^2$  such that
  \begin{alignat}{2}\label{eq:ApproxS1}
    \partial_t \ve_A\NSt{+ \ve_A\cdot \nabla \ve_A} -\Div (2\nu(c_A)D \ve_A) +\nabla p_A&= -\eps \Div (\nabla c_A\otimes \nabla c_A) + R_\eps^1+  R_\eps^2,
\\\label{eq:ApproxS2}
      \Div \ve_A &=\,  
                        G_\eps,\\
      \partial_t c_A + \ve_A\cdot \nabla c_A+\eps^{N+\frac12}\ue|_{\Gamma} \cdot \nabla c_A &= \Delta c_A -\frac1{\eps^2} f'(c_A)
      + S_\eps,\\
      (\ve_A,c_A)|_{\partial\Omega}&= (0,-1),
\end{alignat}
where 
\begin{alignat*}{2}
  \|R_\eps^1\|_{L^2(0,T_\eps; (H^1(\Omega)^2)'))}&\leq {C'}\|\ue|_{\Gamma}\|_{L^2(0,T;L^2(\Gamma_t))} \eps^{N+\frac12},\\
  \|R_\eps^2\|_{L^2(0,T_\eps; (H^1(\Omega)^2)'))}&\leq C(M) \eps^{N+1},\\
  \|G_\eps\|_{H^{1/2}(0,T_\eps; L^2(\Omega))}+\|S_\eps\|_{L^2(0,T_\eps; (X_\eps)')}&\leq C(M)\eps^{N+1},\\
  \|S_\eps\|_{L^2(0,T_\eps; L^2(\Omega))}&\leq C(M)\eps^N.
  \end{alignat*}
  for some $C(M)$, $C'>0$ indpendent of $\eps, T_\eps$. {Here $c_A$ is of the form~\eqref{eq:innerexpan'} below for some $h_{N+\frac12}\in X_{T_\eps}$, which is determined by \eqref{eq:hN12a}-\eqref{eq:hN12b} below in dependence on $\ue$.} In particular, $c_A(.,t)\equiv \pm 1$ in $\Omega^\pm (t)\setminus \Gamma_t(2\delta)$ and $\nabla c_A$ is supported in $\ol{\Gamma_t(2\delta)}$ for $t\in [0,T_\eps]$. Moreover, $M\mapsto C(M)$ is increasing.
  \end{thm}

As before we introduce in $\Gamma(3\delta)$ the stretched variable
 \begin{equation}\label{eq:rho}
   \rho (x,t) = \frac{d_\Gamma(x,t)}\eps - h_\eps(S(x,t),t)
 \end{equation}
with
\begin{equation}\label{eq:heps}
  h_\eps(s,t)= \sum_{k=0}^{N} \eps^kh_{k+1}(s,t),
\end{equation}
where $\{h_k\}_{0\leq k\leq N-1}\subseteq C^\infty(\T^1\times [0,T_0])$ are smooth functions (independent of $\eps$). Moreover, we will use a function $h_{N+\frac12}\colon \T^1\times [0,T_0]\to \R$, which may  depend on $\eps$, such that $h_{N+\frac12}\in X_T$ is bounded (with respect to $\eps$),
where
 \begin{equation}\label{eq:2.0}
  X_T=L^2(0,T;H^{5/2}(\T^1))\cap H^1(0,T;H^{1/2}(\T^1)),
\end{equation}
normed by
\begin{equation*}
  \|h\|_{X_T}:= \|h\|_{L^2(0,T;H^{5/2}(\T^1))}+\|h\|_{H^1(0,T;H^{1/2}(\T^1))}+ \|h|_{t=0}\|_{H^{3/2}(\T^1)}.
\end{equation*}

We construct approximate solutions of the Navier-Stokes/Allen-Cahn system in the following form
\begin{alignat}{1}\label{eq:DefcA}
  c_{A}(x,t)&=\zg c_A^{in}(x,t)+(1-\zg )\(c_A^+
  \chi_+ +c_A^-\chi_- \),\\\label{eq:DefvA}
     \ve_A(x,t)&=\zg \ve_A^{in}(x,t)+(1-\zg )\left(\ve_A^{+}(x,t)
  \chi_+ +\ve_A^-(x,t)\chi_-\right){- {\mathbf{N}}\bar{a}_\eps(t)},\\\label{eq:DefpA}
  p_A(x,t)&=\zg p_A^{in}(x,t)+(1-\zg )
  \left(p_A^{+}(x,t)
  \chi_+ +p_A^{-}(x,t)\chi_- \right),
\end{alignat}
where $\zeta$ is as in \eqref{eq:1.34}, $c_{{A}}^{\pm}=\pm 1$, $\chi_\pm=\chi_{\O^\pm(t)}(x)$, $\mathbf{N}\colon \Omega \to \R^2$ is a smooth vector field such that $\mathbf{N}|_{\partial \Omega}= \no_{\partial\Omega}$ {and $\bar{a}_\eps\colon (0,T)\to \R$ is some suitable function related to the compatibility condition
   \begin{equation*}
     \int_\Omega \operatorname{div} \ve_A \,dx = \int_{\partial\Omega} \no_{\partial\Omega}\cdot   \ve_A \,d\sigma. 
   \end{equation*}}
  It will be essential that we use the following ansatz
\begin{equation}\label{eq:innerexpan'}
\begin{split}
  c_A^{in} (x,t)&=
  \tc_A^{in}(\rho,s,t)+\left(\eps^{N-\frac12} \theta_0'(\rho)+ \eps^{N+\frac32} \partial_\rho \hat{c}_2(\rho, S(x,t),t)\right)h_{N+\frac12}(S(x,t),t),
  \\
  \ve_A^{in} (x,t)&=\tv_A^{in}(\rho,x,t) {+ \eps^{N+\frac12}\hat{\we}(\rho,x,t)},\\
  p_A^{in} (x,t)&=\tp_A^{in}(\rho,x,t){+\eps^{{N-\frac12}}\hat{q}(\rho,x,t)}
\end{split}
\end{equation}
for the inner expansions of $c_A$, $\ve_A$, and $p_A$, where we use the following standard ansatz for the first terms $\tc_A$, $\tv_A$ and $\tp_A$
\begin{alignat}{1}\label{eq:AnsatzVin}
  \tc_A^{in}(\rho,s,t)&= \theta_0(\rho) + \sum_{k=2}^{N+2}\eps^k \tc_k(\rho,x,t),\\
  \tv_A^{in}(\rho,x,t)&= \sum_{k=0}^{N+2}\eps^k \tv_k(\rho,x,t),\\
  \tp_A^{in}(\rho,x,t)&=\sum_{k=-1}^{N+1}\eps^k \tp_k(\rho,x,t).
\end{alignat}
 Here $\tc_k$, $\tv_k$, and $\tp_k$ will be smooth functions that are independent of $\eps$. {Furthermore, $\hat{\we}$ and $\hat{q}$ will be chosen later with the property that $\no\cdot \hat{\we}=0$ in $\Gamma(3\delta)$.} For the following $\tilde{c}_A^{in}$, $\tilde{\ve}_A^{in}$, $\tilde{p}^{in}_A$ are defined as in \eqref{eq:innerexpan'}, but with $(h_{N+\frac12}, \hat{\we}, \hat{q})\equiv 0$. {Moreover, $(\tilde{c}_A, \tilde{\ve}_A, \tilde{p}_A)$ denote the corresponding approximate solution in $\Omega$ (with $(h_{N+\frac12}, \hat{\we}, \hat{q})\equiv 0$).}

The construction is done by the following scheme:
\begin{enumerate}
\item First we construct approximate solutions $(\tilde{c}_A, \tilde{\ve}_A, \tilde{p}_A)$ such that
  in $\Gamma(3\delta)$ we have
    \begin{alignat}{2}\nonumber
  \partial_t \tilde{\ve}_A^{in}\NSt{+ \tilde{\ve}_A^{in}\cdot \nabla \tilde{\ve}_A^{in}} -\Div (2\nu(\tilde{c}_A^{{in}})D \tilde{\ve}_A^{in}) +\nabla \tilde{p}_A^{in}&= -\eps \Div (\nabla \tilde{c}_A^{in}\otimes \nabla \tilde{c}_A^{in}) + R_\eps,
\\\nonumber
\Div \tilde{\ve}_A^{in} &=\,  
G_\eps,\\
\partial_t \tilde{c}_A^{in} + \tilde{\ve}_A^{in}\cdot \nabla \tilde{c}_A^{in} - \Delta \tilde{c}_A^{in} +\frac1{\eps^2} f'(\tilde{c}_A^{in})
&=r_\eps (\rho,s,t)
+ S_\eps
\end{alignat}
    with suitable estimates for the remainder terms.
The construction is done similarly as in \cite{PhDMarquardt}.
\item Now $h_{N+\frac12}$ and therefore the additional terms
  $$
  \left(\eps^{N-\frac12} \theta_0'(\rho)+ \eps^{N+\frac32} \partial_\rho \hat{c}_2(\rho, S(x,t),t)\right)h_{N+\frac12}(S(x,t),t)
  $$  are chosen such that they give
  $
  \eps^{N+\frac12}\ue|_{\Gamma} \cdot \nabla c_A^{in}
  $
  up to lower order terms, cf.\ Theorem~\ref{thm:hN12} below.
\item Finally, $\hat{\we}$ and $\hat{q}$ are chosen such that a leading error term on the right-hand side of the Navier-Stokes equation cancels and all additional terms give only lower order terms of order $O(\eps^{N+\frac12})$ in $L^2(0,T;(H^1(\Omega)^2)')$.
\end{enumerate}

Here the first step can be done in a standard manner. The outcome is summarized in the following theorem:
\begin{thm}\label{thm:Approx1}
Let $N\in \N$. Then there are smooth $(\tilde{c}_A^{in},\tilde{\ve}_A^{in},\tilde{p}_A^{in})$ defined in $\Gamma(3\delta)$ and $(c_A^\pm, \ve_A^\pm, p_A^\pm)$ defined on $\Omega\times [0,T_0]$, which are smooth,  such that:
\begin{enumerate}
\item \emph{Inner expansion:} In $\Gamma(3\delta)$ we have
    \begin{alignat}{2}\nonumber
      \partial_t \tilde{\ve}_A^{in}\NSt{+ \tilde{\ve}_A^{in}\cdot \nabla \tilde{\ve}_A^{in}} -\Div (2\nu(\tilde{c}_A^{in})D \tilde{\ve}_A^{in}) +\nabla \tilde{p}_A^{in}&= -\eps \Div (\nabla \tilde{c}_A^{in}\otimes \nabla \tilde{c}_A^{in}) + R_\eps,
      \\\nonumber
      \Div \tilde{\ve}_A^{in} &=\,  
      G_\eps,\\
      \partial_t \tilde{c}_A^{in} + \tilde{\ve}_A^{in}\cdot \nabla \tilde{c}_A^{in} &= \Delta \tilde{c}_A^{in} -\frac1{\eps^2} f'(\tilde{c}_A^{in})+ S_\eps,
\end{alignat}
where
\begin{alignat}{2}\label{eq:RemainderApprox1}
  \|(R_\eps, \partial_t G_\eps , S_\eps)\|_{L^\infty((0,T_0)\times\Omega)}&\leq C\eps^{N+1},\\\label{eq:RemainderApprox2}
 \|G_\eps \|_{L^\infty((0,T_0)\times\Omega)}&\leq C\eps^{N+2}.
  \end{alignat}
\item \emph{Outer expansion:}  In $\Omega^\pm$ we have $c_A^\pm \equiv \pm 1$ and
  \begin{alignat}{2}
  \partial_t \ve_A^\pm\NSt{+ \ve_A^\pm\cdot \nabla \ve_A^\pm} -\nu^\pm \Delta \ve_A^\pm +\nabla p_A^\pm&= R^\pm_\eps,
\\\nonumber
\Div \ve_A^\pm &=\, 0,\\
\ve_A^\pm{|_{\partial\Omega}} &=\, \ol{a}_\eps \no_{\partial\Omega},
\end{alignat}
where $\ol{a}_\eps\colon (0,T)\to \R$ is continuous and
\begin{equation*}
  \|R_\eps^\pm \|_{L^\infty(\Omega\times [0,T_0])}\leq C\eps^{N+2}\qquad \text{for all }\eps\in (0,1).
\end{equation*}
\item \emph{Matching condition:} We have
    \begin{equation*}
  \begin{split}
    &\| \p_x^{\beta}( \tilde\ve_A^{in}-\ve_A^{+}
  \chi_+-\ve_A^{-}\chi_-)\|_{L^\infty(\Gamma(3\delta)\setminus \Gamma(\delta) )}\leq C e^{-\frac{\alpha\delta}{2\eps}},\\
  &\| \p_x^{\beta}(  \tilde{p}_A^{in}-p_A^{+}
  \chi_+-p_A^{-}\chi_-)\|_{L^\infty(\Gamma(3\delta)\setminus \Gamma(\delta) )}\leq C e^{-\frac{\alpha\delta}{2\eps}},\\
  &\|   \p_x^{\beta}(\tilde{c}_A^{in}-c_A^{+}
  \chi_+-c_A^{-}\chi_-)\|_{L^\infty(\Gamma(3\delta)\setminus \Gamma(\delta) )}\leq C e^{-\frac{\alpha\delta}{2\eps}}
  \end{split}
\end{equation*}
for all $\eps \in (0,1)$ and $\beta\in \N_0^n$.
\end{enumerate}
  \end{thm}
  \begin{proof}
   The proof is done in the appendix.
 \end{proof}

Let us denote $u_A^{in}:= c_A^{in}- \tilde{c}_A^{in}$ and $\we_A^{in}:= \ve_A^{in}-\tilde{\ve}_A^{in}$. Then we have
\begin{align*}
  &  \partial_t c_A^{in}+ \ve_A^{in}\cdot \nabla c_A^{in}-\Delta c_A^{in} + \frac1{\eps^2}f'(c_A^{in})\\
  &= \partial_t \tilde{c}_A^{in}+ \tilde{\ve}_A^{in}\cdot \nabla \tilde{c}_A^{in}-\Delta \tilde{c}_A^{in} + \frac1{\eps^2}f'(\tilde{c}_A^{in}) +\partial_t u_A^{in}+ \ve_A^{in}\cdot \nabla u_A^{in}+ {\we}_A^{in}\cdot \nabla \tilde{c}_A^{in}-\Delta u_A^{in}\\
  &\qquad + \frac1{\eps^2}f''(\tilde{c}_A^{in}) u_A^{in}+\tilde{s}_A^\eps\\
  &= \partial_t u_A^{in}+ \ve_A^{in}\cdot \nabla u_A^{in}-\Delta u_A^{in} + \frac1{\eps^2}f''(\tilde{c}_A^{in}) u_A^{in}+S_\eps+\tilde{s}_A^\eps+ {\we}_A^{in}\cdot \nabla \tilde{c}_A^{in}
\end{align*}
in $\Gamma_t(3\delta)$, $t\in [0,T_\eps]$, where $\tilde{s}_A^\eps$ contains terms that are quadratic in $u_A^{in}$ times $\eps^{-2}$. Hence $\tilde{s}_A^\eps$ is $O(\eps^{2N-3+\frac12})=O(\eps^{N+\frac12})$ in $L^\infty(0,T_\eps;L^2(\Gamma_t(3\delta)))$ and $O(\eps^{N+1})$ in $L^\infty(0,T_\eps;(X_\eps)')$ if $N\geq 3$ due to \eqref{eq:RemEstim2}, \eqref{eq:RemEstim1}, resp., where the constants are uniform if $\|h_{N+\frac12}\|_{X_{T_\eps}}\leq M$ for some $M>0$. Moreover, one can show ${\we}_A^{in}\cdot \nabla \tilde{c}_A^{in}= O(\eps^{N+1})$ in $L^\infty(0,T_\eps;L^2(\Gamma_t(2\delta)))$ with the aid of Lemma~\ref{lem:rescale} in a straight forward manner since $\no\cdot \hat{\we}(\rho,x,t)=0$ by the construction below. {Here $\hat{\we}(\rho,x,t)$ is as in \eqref{eq:innerexpan'}.}

For the first terms we have:
\begin{thm}\label{thm:hN12}
  Let $T_\eps \in (0,T_0]$ and $h_{N+\frac12}\in X_{T_\eps}$ be the solution of
  \begin{alignat}{2}
    D_t  h_{N+\frac12} -X_0^\ast (\ve) \cdot  \nabla_\Gamma h_{N+\frac12}  - \Delta_\Gamma h_{N+\frac12}  -X_0^\ast(g_0) h_{N+\frac12} \label{eq:hN12a}
   &=-X_0^\ast(\no \cdot \mathbf{u})&\ & \text{on }\T^1\times [0,T_\eps],\\\label{eq:hN12b}
    h_{N+\frac12}|_{t=0} &= 0,
  \end{alignat}
  where $g_0\colon \Gamma\to \R$ is a smooth function, which is given by \eqref{formula:g0} in the appendix. Then for any $M>0$ there is some $C(M)>0$ such that, if $\|h_{N+\frac12}\|_{X_{T_\eps}}\leq M$, it holds
  \begin{align*}
    \|\partial_t u_A^{in}+ \ve_A^{in}\cdot \nabla u_A^{in}-\Delta u_A^{in} + \tfrac1{\eps^2}f''(\tilde{c}_A^{in}) u_A^{in}+ \eps^{N+\frac12}\mathbf{u}|_{\Gamma}\cdot \nabla c_A^{in}\|_{L^2(0,T_\eps;X_\eps')} \leq C(M)\eps^{N+1},\\
        \|\partial_t u_A^{in}+ \ve_A^{in}\cdot \nabla u_A^{in}-\Delta u_A^{in} + \tfrac1{\eps^2}f''(\tilde{c}_A^{in}) u_A^{in}+ \eps^{N+\frac12}\mathbf{u}|_{\Gamma}\cdot \nabla c_A^{in}\|_{L^2(0,T_\eps;L^2(\Gamma_t(2\delta)))} \leq C(M)\eps^{N}.
  \end{align*}
\end{thm}
\noindent
\begin{proof}  First of all, we have
  \begin{equation*}
    u_A^{in}(x,t)= \left(\eps^{N-\frac12} \theta_0'(\rho)+ \eps^{N+\frac32} \partial_\rho \hat{c}_2(\rho, S(x,t),t)\right)h_{N+\frac12}(S(x,t),t), \quad \text{where }\rho=\rho(x,t),
  \end{equation*}
  in $\Gamma_t$, $t\in [0,T_\eps]$.
  Because of Lemma~\ref{lem:ChainRule}, Lemma~\ref{lem:rescale} {and \eqref{eq:heps}}, we have
  \begin{align*}
    \partial_t u_A^{in} &= - \left(\tfrac{V_{\Gamma_t}(s)}\eps +\partial_t^\Gamma h_\eps(r,s,t)\right)\left(\eps^{N-\frac12}\theta_0''(\rho)+\eps^{N+\frac32} \partial_\rho^2 \hat{c}_2(\rho, S(x,t),t) \right)h_{N+\frac12}(s,t)\\
                        &\quad + \eps^{N-\frac12} \theta_0'(\rho)\partial_t^\Gamma h_{N+\frac12}(r,s,t)+ \eps^{N+\frac32} \partial_t^\Gamma\left(\partial_\rho \hat{c}_2(r,\rho,s,t) h_{N+\frac12}(s,t)\right)\\
                        &=  - \eps^{N-\frac32}V_{\Gamma_t}(s)\theta''_0(\rho) h_{N+\frac12}(s,t)\\
    &\quad -\eps^{N-\frac12}\theta_0''(\rho)\partial_t^\Gamma h_{{1}}(r,s,t) h_{N+\frac12}(s,t)+  \eps^{N-\frac12} \theta_0'(\rho)\partial_t^\Gamma h_{N+\frac12}(r,s,t)  + O(\eps^{N+1})
  \end{align*}
  in $L^\infty(0,T_\eps;L^2(\Gamma_t(2\delta)))$ and similarly
  \begin{align*}
    &\ve_A^{in}\cdot \nabla u_A^{in}\\
    &= \ve_A^{in}\cdot \left(\tfrac{\no_{\Gamma_t}(s)}\eps -\nabla ^\Gamma h_\eps(r,s,t)\right)\eps^{N-\frac12}\theta_0''(\rho)h_{N+\frac12} + \ve_A^{in}\cdot \eps^{N-\frac12} \theta_0'(\rho)\nabla^\Gamma h_{N+\frac12}+  O(\eps^{N+1})\\
    &= \eps^{N-\frac32} \ve_{\no}|_{\Gamma}\theta_0''(\rho)h_{N+\frac12}(s,t)-\eps^{N-\frac12} \Div_{\btau}\ve|_{\Gamma}\rho \theta_0''(\rho)h_{N+\frac12}(s,t)\\
    &\quad - \eps^{N-\frac12}\ve_0\cdot \nabla^\Gamma h_{{1}}(r,s,t)\theta''_0 (\rho)h_{N+\frac12} + \eps^{N-\frac12} \ve_0\cdot \nabla^\Gamma h_{N+\frac12}(r,s,t)\theta'_0 (\rho) +  O(\eps^{N+1})
  \end{align*}
in $L^\infty(0,T_\eps;L^2(\Gamma_t(2\delta)))$,  where $s=S(x,t)$, $r=d_\gamma(x,t)$, $\rho=\rho(x,t)$, and we have used
  \begin{equation*}
    \no_{\Gamma_t}\cdot \ve_0(\rho,x,t)= \ve_\no|_{\Gamma}(x,t) - \eps \Div_{\btau} \ve|_{\Gamma}(x,t)\rho + O(\eps^2),
  \end{equation*}
  cf.\ (4.17)ff.\ in \cite[proof of Lemma 4.4]{StokesAllenCahn}. Moreover,
 \begin{align*}
    \Delta u_A^{in}
   =& \eps^{N-\frac52}\theta_0'''(\rho)h_{N+\frac12}(s,t) +  \eps^{N-\frac32}\Delta d_\Gamma(x,t) \theta_0''(\rho) h_{N+\frac12}(s,t)\\
   & + |\nabla ^\Gamma h_{{1}}(r,s,t)|^2\eps^{N-\frac12}\theta_0'''(\rho)h_{N+\frac12}(s,t) + \eps^{N-\frac12} \partial_\rho^3 \hat{c} (\rho,s,t)h_{N+\frac12}(s,t)\\
    &  + \eps^{N-\frac12}\Delta^\Gamma h_{N+\frac12}\theta_0'(\rho) - \eps^{N-\frac12}\left(h_{N+\frac12}\Delta^\Gamma h_{{1}}+ 2 \nabla^\Gamma h_{{1}}\cdot \nabla^\Gamma h_{N+\frac12}\right)\theta_0''(\rho) +  O(\eps^{N+1})\\
 =& \eps^{N-\frac52}\theta_0'''(\rho)h_{N+\frac12}(s,t) -  \eps^{N-\frac32}H_\Gamma(s,t) \theta_0''(\rho) h_{N+\frac12}(s,t)\\
   & + |\nabla ^\Gamma h_{{1}}(r,s,t)|^2\eps^{N-\frac12}\theta_0'''(\rho)h_{N+\frac12}(s,t) + \eps^{N-\frac12} \partial_\rho^3 \hat{c} (\rho,s,t)h_{N+\frac12}(s,t)\\
    &  + \eps^{N-\frac12}\Delta^\Gamma h_{N+\frac12}\theta_0'(\rho) - \eps^{N-\frac12}\kappa_1(s,t) h_{N+\frac12}\rho\theta_0''(\rho)\\
   &- \eps^{N-\frac12}\left(h_{N+\frac12}\Delta^\Gamma h_{{1}}+ 2 \nabla^\Gamma h_{{1}}\cdot \nabla^\Gamma h_{N+\frac12}+\kappa_1h_{{1}}h_{N+\frac12}\right)\theta_0''(\rho) +  O(\eps^{N+1})
 \end{align*}
in $L^\infty(0,T_\eps;L^2(\Gamma_t(2\delta)))$  because of \eqref{eq:ex1}{,} and {lastly}
 \begin{align*}
   &\frac1{\eps^2}f''(\tilde{c}_A^{in}) u_A^{in} =  \eps^{N-\frac52}f''(\theta_0(\rho){)}\theta_0'(\rho)h_{N+\frac12}(s,t)\\
   &\quad +  \eps^{N-\frac12}\left(f'''(\theta_0(\rho)) \tc_2(\rho,s,t)\theta_0'(\rho) + f''(\theta_0(\rho)) \partial_\rho \tc_2(\rho,s,t) \right) h_{N+\frac12}(s,t)   +  O(\eps^{N+1})
 \end{align*}
 in $L^\infty(0,T_\eps;L^2(\Gamma_t(2\delta)))$.

 In the following we use that
 \begin{equation}\label{eq:hatd} 
   -\partial_\rho^2 \tc_2(\rho,s,t) + f''(\theta_0(\rho))\tc_2 = |\nabla_\Gamma h_{{1}}(s,t)|^2 \theta_0''(\rho)-\theta_0'(\rho)\rho g_0(s,t) 
 \end{equation}
 for all $s\in\T^1, t\in [0,T_\eps]$ and $\rho\in\R$ (cf.\ Remark~\ref{rem:B1} in the appendix below), which yields after differentiation with respect to $\rho$
 \begin{equation*}
   -\partial_\rho^3 \tc_2 + f''(\theta_0(\rho))\partial_\rho\tc_2 + f'''(\theta_0(\rho))\theta_0'(\rho) \tc_2 = |\nabla_\Gamma h_{{1}}|^2 \theta_0'''(\rho)-(\theta_0''(\rho)\rho +\theta_0'(\rho))g_0(s,t)
 \end{equation*}
 for all $s\in\T^1$, $t\in [0,T_\eps]$ and $\rho\in\R$.
 Using additionally \eqref{eq:Limit5},
 \begin{align*}
 -\theta_0'''(\rho) + f''(\theta_0(\rho){)}\theta_0'(\rho)=0\qquad \text{for all }\rho\in\R,
 \end{align*}
 and
 \begin{equation*}
   \partial_t^\Gamma  h_{N+\frac12} -\ve|_{\Gamma} \cdot  \nabla_\Gamma h_{N+\frac12}  - \Delta_\Gamma h_{N+\frac12}  -g_0 h_{N+\frac12} + \divtau \ve|_{\Gamma} h_{N+\frac12} = - {\no\cdot \ue}|_{\Gamma}
 \end{equation*}
 for {all} $s\in\Gamma_t$, $t\in [0,T_\eps]$,
 we obtain
 \begin{align*}
   \partial_t u_A^{in}+ \ve_A^{in}\cdot \nabla u_A^{in}-\Delta u_A^{in} + \frac1{\eps^2}f''(\tilde{c}_A^{in}) u_A^{in} &= - \eps^{N-\frac12}{\no\cdot \ue}|_{\Gamma} \theta_0'(\rho) + \eps^{N-\frac12} a(\rho,x,t) + O(\eps^{N+1})\\
   &{= -\eps^{N+\frac12}\ue|_{\Gamma} \cdot \nabla \tilde{c}_A^{in} + \eps^{N-\frac12} a(\rho,x,t) + O(\eps^{N+1})}
 \end{align*}
 in $L^\infty(0,T_\eps;L^2(\Gamma_t(2\delta)))$ since the $O(\eps^{N-\frac52})$- and $O(\eps^{N-\frac32})$-terms cancel, where
 \begin{align*}
   \int_{\R} a(\rho,s,t) \theta_0'(\rho) \, d\rho =0\qquad \text{for all }s\in\T^1, t\in [0,T_\eps]
 \end{align*}
 because of $\int_{\R}\theta_0''(\rho)\theta_0'(\rho)\, d\rho =0$.
 Next we use $a(\rho,x,t)= f_1(x,t)+f_2(x,t)$ with
 \begin{equation*}
 f_1(x,t)= \eps^{N-\frac12} a(\rho,S(x,t),t), \qquad f_2(x,t)= \eps^{N-\frac12} \left(a(\rho,x,t)- a(\rho,S(x,t),t) \right) + O(\eps^{N+1})
\end{equation*}
in $L^\infty(0,T_\eps;L^2(\Gamma_t(2\delta)))$. Then, because of Lemma~\ref{lem:EstimMeanValueFree} and Lemma~\ref{lem:rescale},
\begin{align*}
  \|f_1\|_{L^2(0,T_\eps;X_\eps')}&\leq  C\eps^{N-\frac12+\frac32}=C \eps^{N+1},\\
  \|f_1\|_{L^2(0,T_\eps;L^2(\Gamma_t(2\delta)))}&\leq  C\eps^{N},
\end{align*}
and
\begin{align*}
  \|f_2\|_{L^2(0,T_\eps;X_\eps')}&\leq C\|f_2\|_{L^2(0,T_\eps;L^2(\Gamma_t(2\delta)))}\leq  C'\eps^{N-\frac12+\frac32}= C\eps^{N+1}
\end{align*}
 because of \cite[Corollary 2.7]{StokesAllenCahn}. Altogether, the previous estimates  imply the statements of the theorem.
\end{proof}

  \noindent
  \begin{proof*}{of Theorem~\ref{thm:approx}}
    Let $(c_A, \ve_A, p_A)$ be as in \eqref{eq:DefcA}-\eqref{eq:DefpA}
    Moreover, we define
    \begin{align}\nonumber
    {\mathbf{a}(\rho,x,t) :=}& \partial_\rho (\theta_0''(\rho)\theta_0'(\rho))h_{N+\frac12}(s,t)(\nabla^\Gamma h_1) (r,s,t) + \theta_0''(\rho)\theta_0'(\rho)(\nabla^\Gamma h_{N+\frac12})(r,s,t)\\\label{eq:Defna}
                         & -2\theta_0''(\rho)\theta_0'(\rho) h_{N+\frac12}(s,t)\Div \left(\no \otimes \no\right)
    \end{align}
    with $s=S(x,t), r= d_{\Gamma_t}(x)$,
  and choose $\hat{\we}$ and $\hat{q}$ as solution of
  \begin{align*}
    -\partial_\rho(2\nu(\theta_0(\rho))\partial_\rho \hat{\we}(\rho,x,t))+ \partial_\rho \hat{q}(\rho,x,t){\no}&= \mathbf{a}(\rho,x,t),\\
    \partial_\rho \hat{\we}(\rho,x,t)\cdot \no &= 0
  \end{align*}
  for all $\rho \in\R$, $x\in \Gamma_t(3\delta), t\in [0,T_\eps]$,
  where we use that $\int_\R \mathbf{a}(\rho,x,t)\, d\rho=0$ and Lemma~\ref{modified version of Lemma 2.4} in the appendix.
  Actually, $\hat{q}$ is determined such that
  \begin{align*}
    \partial_\rho \hat{q}(\rho,x,t)&= \no\cdot \mathbf{a}(\rho,x,t)
  \end{align*}
  and
   \begin{align*}
     -\partial_\rho(2\nu(\theta_0(\rho))\partial_\rho \hat{\we}_\btau(\rho,x,t))&= \mathbf{a}_\btau(\rho,x,t),\\
    \no \cdot \hat{\we}(\rho,x,t)&= 0
  \end{align*}
  for all $\rho\in \R$, $x\in \Gamma_t(3\delta)$, $t\in [0,T_\eps]$. Since $\mathbf{a}$ depends linearly on $(h_{N+\frac12}, \partial_s h_{N+\frac12})\in L^2(0,T_\eps; H^1(\T^1))^2\cap H^{1/2}(0,T_\eps;L^2(\T^1))$, we have that
    \begin{equation*}
      \hat{\we}_j(\rho,x,t)= \hat{\mathbf{b}}_j(\rho,x,t)\cdot (h_{N+\frac12}(s,t), \partial_s h_{N+\frac12}(s,t))
      \end{equation*}
      for some smooth $\hat{\mathbf{b}}_j\colon \R\times \Gamma(3\delta)\to \R^2$, $j=1,2$, with uniformly bounded $C^k$-norms such that $\partial_\rho^k \hat{\mathbf{b}}_j\in (\mathcal{R}_{0,\alpha})^2 $ for every $k\in\N$ and $j=1,2$. Since $h_{N+\frac12}, \partial_s h_{N+\frac12}\in L^2(0,T_\eps; H^1(\T^1))\cap H^{\frac12}(0,T_\eps;L^2(\T^1)) $ are bounded by some $C(M)$, we also obtain
      \begin{equation*}
        \|\hat{\mathbf{w}}_j\|_{L^2(0;T_\eps; H^1(\R\times \T^1))} + \|\hat{\mathbf{w}}_j\|_{H^{\frac12}(0,T_\eps;L^2(\R\times \T^1))} \leq C(M)
      \end{equation*}
      for all $\eps\in (0,\eps_0)$. 
    Then
    \begin{align}\nonumber
      &\Div (\ve_A(x,t) + \mathbf{N}\bar{a}_\eps(t))  = \zeta(d_\Gamma(x,t))( G_\eps+\eps^{N+\frac12}((\Div_x\hat{\we})(\rho,x,t) -\nabla h_\eps(S(x,t),t)\partial_\rho \hat{\we}(\rho,x,t)) \\\label{eq:DivError}
      &\quad + \zeta'(d_\Gamma(x,t))\left(\ve_A^{in} - \ve^+_A(x,t)\chi_{\Omega^+(t)}(x) - \ve^-_A(x,t)\chi_{\Omega^-(t)}(x)\right) =O(\eps^{N+1})
    \end{align}
    in $L^\infty(0,T_\eps;L^2(\Omega))$
    because of \eqref{eq:RemainderApprox2}, the matching condition in Theorem~\ref{thm:Approx1} and $\no\cdot \hat\we(\rho,x,t)=0$.
    Hence
    \begin{align*}
      \sup_{0\leq t\leq T_\eps}|\bar{a}_\eps(t)|  &= \sup_{0\leq t\leq T_\eps} \frac1{\mathcal{H}^1(\partial\Omega)} \left|\int_{\partial \Omega}(\ve_A + \mathbf{N}\bar{a}_\eps(t)\, d\sigma\right|\\  
      &=\sup_{0\leq t\leq T_\eps} \frac1{\mathcal{H}^1(\partial\Omega)}\left|\int_{\Omega}\Div (\ve_A + \mathbf{N}\bar{a}_\eps) \, dx\right| \leq C\eps^{N+1}
    \end{align*}
    and similarly
    \begin{align*}
      \|\bar{a}_\eps\|_{H^{1/2}(0,T_\eps)}  &= \frac1{\mathcal{H}^1(\partial\Omega)}\left\|\int_{\Omega}\Div (\ve_A + \mathbf{N}\bar{a}_\eps) \, dx\right\|_{H^{1/2}(0,T_\eps)} \leq C\eps^{N+1}.
    \end{align*}
     This shows \eqref{eq:ApproxS2} for some (different) $G_\eps\colon \Omega \times (0,T_\eps)\to \R$, which is given by the sum of the right-hand side of \eqref{eq:DivError} and $-\Div (\mathbf{N}\bar{a}_\eps(t))$, such that
    \begin{equation*}
      \|G_\eps \|_{H^{1/2}(0,T_\eps;L^2(\Omega))}\leq C(M) \eps^{N+1}.
    \end{equation*}
    In the same manner as for the divergence equation one shows that
    \begin{align*}
      &\partial_t \ve_A -\Div (2\nu(c_A) D\ve_A) + \nabla p_A= -\eps\zeta(d_\Gamma)  \Div (\nabla \tilde{c}_A^{in}\otimes \nabla \tilde{c}_A^{in})\\
      &\quad  +\eps^{N-\frac32}\left(-\partial_\rho(2\nu(\theta_0(\rho))\partial_\rho \hat{\we}(\rho,x,t)+ \partial_\rho \hat{q}(\rho,x,t){\no}\right) + \eps^{N-\frac12}\mathbf{r}(\rho,x,t)+ O(\eps^{N+1})
    \end{align*}
    in $L^\infty(0,T_\eps;L^2(\Gamma_t(2\delta)))$ by using \eqref{eq:RemainderApprox1} and the matching condition in Theorem~\ref{thm:Approx1}. Moreover, one can use \eqref{eq:RemEstim1} in Lemma~\ref{lem:rescale} to show that $\eps^{N-\frac12}\mathbf{r}(\rho,x,t)$ is of order $O (\eps^{N+\frac12})$ in $L^2(0,T_\eps; (H^1(\Gamma_t(2\delta))')$. Hence
       \begin{align}\nonumber
         \partial_t \ve_A & -\Div (2\nu(c_A) D\ve_A) + \nabla p_A= -\eps\zeta(d_\Gamma)  \Div (\nabla \tilde{c}_A^{in}\otimes \nabla \tilde{c}_A^{in})\\\label{eq:veA1}
   &      +\eps^{N-\frac32}\left(-\partial_\rho(2\nu(\theta_0(\rho))\partial_\rho \hat{\we}(\rho,x,t))+ \partial_\rho \hat{q}(\rho,x,t){\no}\right)+ O(\eps^{N+\frac12})
    \end{align}
    in $L^2(0,T_\eps; (H^1(\Gamma_t(2\delta))')$.
    Now we use that
    \begin{align*}
      &\eps \nabla c_A^{in}\otimes \nabla c_A^{in} -\eps \nabla \tilde{c}_A^{in}\otimes \nabla \tilde{c}_A^{in} \\
      &= \eps^{N-\frac32}2\theta_0''(\rho)\theta_0'(\rho)\no\otimes \no h_{N+\frac12}(S(x,t),t) \\
      &\quad + \eps^{N-\frac12}\theta_0''(\rho)\theta_0'(\rho)\left(\nabla^\Gamma h_1(S(x,t),t)\otimes \no+ \no\otimes\nabla^\Gamma h_1(S(x,t),t)\right) h_{N+\frac12}(S(x,t),t) \\
           &\quad + \eps^{N-\frac12}\theta_0'(\rho)^2\left(\nabla^\Gamma h_{N+\frac12}(S(x,t),t)\otimes \no+ \no\otimes\nabla^\Gamma h_{N+\frac12}(S(x,t),t)\right)  \\
      &\quad + \eps^{N+\frac12} \mathbf{r}_\eps(\rho,x,t)\cdot \mathfrak{a}_\eps(s,t)
    \end{align*}
    for some $\mathbf{r}_\eps \in (\mathcal{R}_{0,\alpha})^N$, $\mathfrak{a}_\eps \in L^\infty(0,T;L^2(\T^1))^N$ and $N\in\N$ with uniformly bounded norms in $\eps\in (0,\eps_0)$. Here one uses that $h_{N+\frac12} \in X_T\hookrightarrow L^\infty(0,T; W^1_4(\T^1))$ is bounded and that $\partial_s h_{N+\frac12}$ enters at most quadratically. Hence, using Lemma~\ref{lem:rescale}, we obtain
     \begin{align*}
       &  -\eps \Div (\nabla c_A^{in}\otimes \nabla c_A^{in}) +\eps \Div (\nabla \tilde{c}_A^{in}\otimes \nabla \tilde{c}_A^{in})\\
       &\quad =  \eps^{N-\frac32} \Div \left(2\theta_0''(\rho)\theta_0'(\rho)(I-\no\otimes \no) h_{N+\frac12} \right) +\nabla p_\eps\\
       &\qquad  - \eps^{N-\frac12} \Div \left(\theta_0''(\rho)\theta_0'(\rho)(\nabla^\Gamma h_1\otimes \no+ \no\otimes\nabla^\Gamma h_1) h_{N+\frac12} \right)\\
       &\qquad - \eps^{N-\frac12}\Div \left(\theta_0'(\rho)^2\left(\nabla^\Gamma h_{N+\frac12}(S(x,t),t)\otimes \no+ \no\otimes\nabla^\Gamma h_{N+\frac12}(S(x,t),t)\right)\right)+ O(\eps^{N+1})\\
          &\quad = {\eps^{N-\frac32} \mathbf{a}(\rho,x,t)} + \nabla p_\eps + R_\eps^1+ O(\eps^{N+1})
  \end{align*}
  in $L^2(0,T_\eps, (H^1(\Gamma_t(2\delta))'))$, where $p_\eps(x,t) = \eps^{N-\frac32}{2} \theta_0''(\rho)\theta_0'(\rho) h_{N+\frac12}(S(x,t),t) $,
   $\mathbf{a}$ is as in \eqref{eq:Defna}
   and
   \begin{equation*}
     \|R_\eps^1\|_{L^2(0,T; (H^1(\Omega)^2)')}\leq C\eps^{N+\frac12} \|h_{N+\frac12}\|_{L^2(0,T;H^2(\T^1))}\leq C'\eps^{N+\frac12} \|\ue|_{\Gamma}\|_{L^2(\Gamma)}
   \end{equation*}
   for some $C>0$ independent of $\eps\in (0,\eps_0)$, $M>0$.
   Here one estimates the additional remainder terms in the same way as before with the aid of \eqref{eq:RemEstim1} and uses \eqref{eq:hN12a}-\eqref{eq:hN12b} together with standard estimates for parabolic equations.
  Altogether the terms in \eqref{eq:veA1} with the factor $\eps^{N-\frac32}$ cancel and all remaining terms are of order $O (\eps^{N+\frac12})$ in $L^2(0,T_\eps; (H^1(\Gamma_t(2\delta))')$.
  Therefore combining \eqref{eq:veA1} and the previous identity and replacing $p_A$ by $p_A+p_\eps$ we obtain \eqref{eq:ApproxS1}.
  \end{proof*}

\section{Proof of Main Result}
\subsection{Preparations and the Error in the Velocity}\label{subsec:LeadingErrorVelocity}

Throughout this subsection we assume that $(c_A, \ve_A, p_A)$ and $(\tilde{c}_A,\tilde{\ve}_A, \tilde{p}_A)$ are given as in Section~\ref{sec:ApproxSolutions}, where $(c_A, \ve_A, p_A)$ still depends on the choice of $\mathbf{u}$, which will be specified next, but $(\tilde{c}_A,\tilde{\ve}_A, \tilde{p}_A)$ do not. In the following we will often write $\no$ instead of $\no_\Gamma$.

Let $\tilde{\we} \colon \Om\times [0,T_0]\to \R^2$ be such that
 \begin{equation}\label{velocity decompose}
   \ve_\eps= \ve_A+ \tilde{\we}.
 \end{equation}
Then $\tilde{\we}$ solves
\begin{equation}\label{eq:w}
\begin{split}
 \partial_t \tilde{\we}{+ \ve_\eps \cdot \nabla \tilde{\we}}-\Div(2\nu(c_\eps) D\tilde{\we})+\nabla q&=-\eps\Div (\nabla u\otimes^s \nabla c_A)-\eps\Div (\nabla u\otimes  \nabla u)\\&\quad+\Div(2(\nu(c_\eps)-\nu(c_A)) D\ve_A){- \tilde{\we}\cdot \nabla \ve_A}- {R_\eps^1-R_\eps^2},\\
  \Div \tilde{\we}&= -G_\eps,\\
  \tilde{\we}|_{t=0}&=\ve_{0,\eps} - \ve_{A,0},\\
    \tilde{\we}|_{\partial\Omega}&= 0,
\end{split}
\end{equation}
for some $q\colon \Om\times [0,T_0]\to \R$,
where $u=c_\eps-c_A$, $a\otimes^s b=a\otimes b+b\otimes a$ and $R_\eps, G_\eps$ are as in Theorem~\ref{thm:approx}.
Now we choose
\begin{equation}\label{eq:Defue}
  \ue = \frac{\tilde{\we}}{\eps^{N+\frac12}}=:\we \in L^2(0,T_0; H^1(\Omega)^2)\cap L^\infty(0,T_0;L^2_\sigma(\Omega))^2.
\end{equation}
More precisely, since the right-hand side depends on $c_A$ and therefore on $\ue$, this $\ue$ is determined by a non-linear (and non-local) evolution equation with a globally Lipschitz nonlinearity, which can be solved by the same arguments as in \cite[Proof of Lemma~4.2]{StokesAllenCahn}.
 Then $c_A$ solves
   \begin{alignat}{2}\label{eq:cA}
     \partial_t c_A+ \ve_A \cdot \nabla c_A + \eps^{N+\frac12}\we|_{\Gamma}\cdot \nabla c_A&= \Delta c_A-\frac{f'(c_A)}{\eps^2} +S_\eps
   \end{alignat}
   in $\Omega\times (0,T_0)$.
For the following we consider the estimates
\begin{subequations}\label{assumptions}
  \begin{align}
    \sup_{0\leq t\leq \tau} \|c_\eps(t) -c_A(t)\|_{L^2(\Omega)} + \|\nabla (c_\eps -c_A)\|_{L^2(\Omega\times(0,\tau)\setminus \Gamma(\delta))} &\leq R\eps^{\order+\frac12},\\
    \|\nabla_\btau(c_\eps -c_A)\|_{L^2(\Omega\times(0,T_\eps)\cap \Gamma(2\delta))} +\eps\|\partial_\no(c_\eps -c_A)\|_{L^2(\Omega\times(0,\tau)\cap \Gamma(2\delta))} &\leq R\eps^{\order+\frac12},\\
    \|\nabla(c_\eps -c_A)\|_{L^\infty(0,\tau;L^2(\Omega))}+\|\nabla^2(c_\eps -c_A)\|_{L^2(\Omega\times(0,\tau))} &\leq R\eps^{\order-\frac32},\\
        \int_0^\tau\int_{\Gamma_t(\delta)}|\nabla u|^2 +\eps^{-2} f''(c_A(x,t)) u^2\, dx\, dt &\leq R^2\eps^{2N+1}.
\end{align}
\end{subequations}
hold true for some $\tau=\tau(\eps) \in (0,T_0]$, $\eps_0\in (0,1]$, and all $\eps \in (0,\eps_0]$, where $R=R(\theta)>0$ is chosen such that
 \begin{equation}\label{initial assumption-0}
  \|c_{0,\eps}-c_{A,0}\|_{L^2(\Omega)}^2+ \varepsilon^4\|\nabla(c_{0,\eps}-c_{A,0})\|_{L^2(\Omega)}^2+\|\ve_{0,\eps}-\ve_{A,0}\|_{L^2(\Omega)}^2\leq \theta^2 \frac{R^2}{4}\eps^{2\order+1}e^{-C_LT_0}
\end{equation}
for all $\eps\in (0,1]$, where $C_L>0$ is the constant from the spectral estimate in Theorem~\ref{thm:Spectral} and $\theta \in (0,1]$ is a suitable constant to be chosen later.
Moreover, we define
\begin{align}
T_{\varepsilon}:=\sup\{\tau\in[0,T_0]: \eqref{assumptions}  \ \text{holds true}.\}.\label{def:Teps}
\end{align}
Then $T_{\varepsilon}>0$.

The following theorem provides the essential estimate for the error in the velocity.
\begin{theorem}\label{thm:weEstim}
 Assume that $c_A, \tilde{\we}$ are as above and $\ue$ satisfies \eqref{assumptions} for some $R>0$, $\tau=T(\eps)$, and $N\geq 3$. Then there are some $C(R)>0$, $\eps_0>0$, and $M>0$ independent of $\eps \in (0,\eps_0)$ and $T\in (0,T_\varepsilon]$ such that $\|h_{N+\frac12}\|_{X_T}\leq M$ and
  \begin{equation}
    \|\tilde{\we}\|_{L^\infty(0,T;L^2(\Omega))}+\|\tilde{\we}\|_{L^2(0,T;H^1(\Omega))}\leq C(R) \eps^{N+\frac{1}{2}},\label{w1estimate}
  \end{equation}
  where $h_{N+\frac12}$ is as in Theorem~\ref{thm:hN12} with $\ue$ as in \eqref{eq:Defue}.
\end{theorem}
\begin{proof}
  First assume that $h_{N+\frac12}\in X_{T_\eps}$ is some function with $\|h_{N+\frac12}\|_{X_{T_\eps}}\leq M$ and $c_A, \ve_A, p_A$ is constructed with this choice of $h_{N+\frac12}$. We note that $\tilde{c}_A$ and $\tilde{\ve}_A$ are independent of $h_{N+\frac12}$.
  
To proceed we introduce $(\we_0,q_0)$ satisfying the following system
\begin{equation}\label{neweq:w0}
\begin{split}
\partial_t \we_0-\Delta\we_0+\nabla q_0&=0 \qquad\quad\text{in }\Omega\times [0,T_\eps],\\
  \Div \we_0&=G_\eps\qquad\text{ in }\Omega\times [0,T_\eps],\\
  \we_0|_{\partial\Omega}&= 0\qquad\quad\text{on }\partial\Omega\times [0,T_\eps],\\
  \we_0|_{t=0} &=0 \qquad\quad\text{in }\Omega.
\end{split}
\end{equation}
We note that
\begin{equation*}
  \int_\Omega G_\eps \,dx = \int_\Omega \Div \ve_A \,dx = 0 \qquad \text{for all } t\in [0,T_\eps]
\end{equation*}
since $G_\eps$ is as in Theorem~\ref{thm:approx}. Therefore the latter system possesses a unique solution for every $t\in[0,T_\eps]$, for which we have
\begin{equation}\label{est:w0}
\begin{split}
  &\|\we_0\|_{L^2(0,T_\eps;H^1(\Omega))}+ \|\partial_t\we_0\|_{L^2(0,T_\eps;V')}\\
  &\leq C \left(\|G_\eps \|_{L^2(\Omega\times (0,T_\eps))}+ \|G_\eps\|_{H^{1/2}(0,T_\eps; H^{-1}_{(0)}(\Omega))} \right)\leq C(M) \eps^{N+1}
\end{split}
\end{equation}
because of \cite[Theorem~3.3]{SchumacherInstNSt}, where $V= H^1_0(\Omega)^2\cap L^2_\sigma(\Omega)$ and $H^{-1}_{(0)}(\Omega)= (H^1_{(0)}(\Omega))'$.
Using
 \begin{equation*}
   L^2(0,T_\eps; V)\cap H^1(0,T_\eps;V')\hookrightarrow L^\infty(0,T_\eps;L^2_\sigma(\Omega)),
 \end{equation*}
 we obtain additionally
 \begin{align}
   \|\we_0\|_{L^\infty(0,T_\eps;L^2(\Omega))}\leq C(M) \eps^{N+1}.\label{est:w00}
 \end{align}
Now we define $\bar{\we}=\tilde{\we}+\we_0$ and then
the system for $\bar{\we}$ reads
\begin{equation}\label{neweq:w}
\begin{split}
 \partial_t \bar{\we}{+ \ve_\eps \cdot \nabla \bar{\we}}-\Div(2\nu(c_\eps) D\bar{\we})+\nabla q&=-\eps\Div (\nabla u\otimes^s \nabla c_A){- \bar{\we}\cdot \nabla \ve_A}+\we_0\cdot \nabla \ve_A\\&\quad+\Div(2(\nu(c_\eps)-\nu(c_A)) D\ve_A)\\
  &\quad -\eps\Div (\nabla u\otimes  \nabla u)- R_\eps^1-R_\eps^2\\
  &\quad +\partial_t\we_0+\ve_\eps \cdot \nabla \we_0-\Div(2\nu(c_\eps) D\we_0),\\
  \Div \bar{\we}&= 0,\\
  \bar{\we}|_{t=0}&=\ve_{0,\eps} - \ve_{A,0}+\we_0,\\
    \bar{\we}|_{\partial\Omega}&= 0,
\end{split}
\end{equation}
where the system has to be understood in a weak sense.
Since $\nabla \tilde{\ve}_A$ is uniformly bounded, we have
\begin{equation*}
  \left|\int_0^T\int_\Omega  (\bar{\we}\cdot \nabla \tilde{\ve}_A)\cdot \bar{\we}\, dxdt\right|\leq C\|\bar{\we}\|_{L^2(\Omega\times (0,T_\eps))}^2.
\end{equation*}
Moreover, Korn's inequality leads to
\begin{align}
\int_{\Omega} \nu(c_\eps) |D\bar{\we}|^2dx\geq C\int_{\Omega} |\nabla\bar{\we}|^2dx.\label{Korninequality}
\end{align}
Testing \eqref{neweq:w} with $\bar{\we}$ and integrating by parts yield that
\begin{align}
  &\frac{1}{2}\frac{d}{dt}\int_{\Omega} |\bar{\we}|^2dx-C \int_{\Omega} |\bar{\we}|^2dx+2\int_{\Omega} \nu(c_\eps) |D\bar{\we}|^2dx \nonumber\\&\leq\eps\int_{\Omega}\bigg(\nabla u\otimes^s  \nabla c_A:\nabla\bar{\we}\bigg)dx-2\int_{\Omega}\bigg(\big(\nu(c_\eps)-\nu(c_A)\big) D\ve_A:\nabla\bar{\we}\bigg)dx
  \nonumber\\
  &\quad +\varepsilon\int_{\Omega}\bigg(\nabla u\otimes\nabla u:\nabla\bar{\we}\bigg)dx
    -\int_{\Omega}R_\eps\cdot\bar{\we}x
  \nonumber\\&\quad-\int_{\Omega}(\bar{\we}\cdot \nabla (\ve_A-\tilde{\ve}_A))\cdot \bar{\we}\, dx +  \weight{\partial_t \we_0(t), \bar{\we}}_{V',V}  +
 \nonumber\\&\quad+\int_{\Omega}\bigg((\ve_\eps\cdot \nabla \we_0+\we_0\cdot\nabla \ve_A )\cdot\bar{\we}+\big(2\nu(c_\eps) D\we_0:\nabla\bar{\we}\big)\bigg)dx.\label{wenergyequ-old}
\end{align}
Thanks to Gronwall's inequality we get
\begin{align}
  &\sup_{0\leq t\leq T}\frac{1}{2}\int_{\Omega} |\bar{\we}|^2dx  +2\int_{0}^T\int_{\Omega} \nu(c_\eps) |D\bar{\we}|^2dx dt\nonumber\\&\leq e^{CT}\bigg( \frac{1}{2}\int_{\Omega} |\bar{\we}|^2|_{t=0}dx+\eps\int_{0}^T\bigg|\int_{\Omega}\big(\nabla u\otimes^s  \nabla c_A:\nabla\bar{\we}\big)dx\bigg|dt\nonumber\\&\quad+2\int_{0}^T\bigg|\int_{\Omega}\big(\big(\nu(c_\eps)-\nu(c_A)\big) D\ve_A:\nabla\bar{\we}\big)dx\bigg|dt
  \nonumber\\
  &\quad +\varepsilon\int_{0}^T\bigg|\int_{\Omega}\big(\nabla u\otimes\nabla u:\nabla\bar{\we}\big)dx\bigg|dt
    +\int_{0}^T\bigg|\int_{\Omega}R_\eps\cdot\bar{\we}dx\bigg|dt
  \nonumber\\&\quad+\int_{0}^T\bigg|\int_{\Omega}(\bar{\we}\cdot \nabla (\ve_A{-\tilde{\ve}_A})\cdot \bar{\we}dx\bigg|dt + \int_0^T\bigg| \weight{\partial_t \we_0(t), \bar{\we}}_{V',V} \bigg|dt
 \nonumber\\&\quad+\int_{0}^T\bigg|\int_{\Omega}\big((\ve_\eps\cdot \nabla \we_0+\we_0\cdot\nabla \ve_A )\cdot\bar{\we}+\big(2\nu(c_\eps) D\we_0:\nabla\bar{\we}\big)\big)dx\bigg|dt\bigg).\label{wenergyequ}
\end{align}
To proceed the proof  will be divided into three steps.

\noindent
\emph{Step 1:} {In this step we prove}
\begin{align}
\eps\int_0^T\bigg|\int_{\Omega}\nabla u\otimes \nabla c_A:\nabla\bar{\we}dx\bigg|dt
  \leq \left(C_0(R)\varepsilon^{N+\frac{1}{2}}+ C(R,M)\varepsilon^{N+1}\right)\|\nabla\bar{\we}\|_{L^2\big(0,T;L^2(\Omega)\big)}\label{w1est-1}
\end{align}
for some $C_0(R)$ independent of $M$.
We decompose $\nabla c_A=\nabla \tilde{c}_A\big|_{\Omega\backslash\Gamma_t(2\delta)}+\nabla \tilde{c}_A\big|_{\Gamma_t(2\delta)}+\nabla (c_A-\tilde{c}_A)$ and then the proof of \eqref{w1est-1} 
consists of three parts.
Firstly, obviously there holds
\begin{align}
&\varepsilon\int_0^{T}\bigg|\int_{\Omega\backslash\Gamma_t(2\delta)}\nabla u\otimes \nabla \tilde{c}_A:\nabla\bar{\we}dx\bigg| dt
                \nonumber\\&\qquad
  \leq C\varepsilon\|\nabla u\|_{L^2\big(0,T;L^2(\Omega\backslash\Gamma_t(2\delta))\big)}\|\nabla\bar{\we}\|_{L^2\big(0,T;L^2(\Omega)\big)}.\label{w1est-11}
\end{align}
Secondly, noting that
\begin{equation}\label{eq:LowerBdd}
\big|d_\Gamma(x,t) - \eps h_\eps(S(x,t),t)\big|\geq \frac{\delta}{2}\quad \text{for all } x\in \Gamma_t(2\delta)\backslash\Gamma_t(\delta),\ t\in[0,T], \eps\in (0,\eps_0]
\end{equation}
if $\eps_0\in (0,1)$ is sufficiently small,
we have for $x\in\Gamma_t(2\delta), t\in[0,T]$
\begin{align}
  \nabla \tilde{c}_A (x,t)&=\underbrace{\nabla(\zg )\(\tc_A^{in}(\rho,s,t)-c_A^+\chi_+-c_A^-\chi_- \)}_{=O(e^{-\frac{\alpha\delta}{2\varepsilon}})} +\varepsilon^{-1}\zg\theta_0'(\rho)\nn \nonumber\\&\quad
  -\underbrace{\zg\theta_0'(\rho)\sum_{k=0}^{N} \eps^{k}\nabla_{\btau}h_{k+1}(s,t)}_{=O(1)}
 +\underbrace{\zg\sum_{k=2}^{N+2}\eps^{k} \nabla_{\btau}\tc_k(\rho,s,t)}_{=O(\varepsilon^2)} \nonumber\\&\quad+\underbrace{\zg\sum_{k=2}^{N+2}\eps^{k-1} \partial_{\rho}\tc_k(\rho,s,t)\big)\big(\nn-\sum_{k=0}^{N} \eps^{k+1}\nabla_{\btau}h_{k+1}(s,t)\big)}_{=O(\varepsilon)}.
  \label{innappgrad}
\end{align}
Using \eqref{innappgrad} we obtain
\begin{align}
&\varepsilon\int_0^{T}\bigg|\int_{\Gamma_t(2\delta)}\nabla u\otimes \nabla \tilde{c}_A:\nabla\bar{\we}dx\bigg| dt
\nonumber\\&\leq \int_0^{T}\bigg|\int_{\Gamma_t(2\delta)}\zg\theta_0'(\rho)\partial_{\nn} u\big(\nn\otimes\nn:\nabla\bar{\we}\big)dx \bigg|dt+ C\|\nabla_{\btau} u\|_{L^2\big(0,T;L^2(\Gamma_t(2\delta))\big)}\|\nabla\bar{\we}\|_{L^2\big(0,T;L^2(\Omega)\big)}
\nonumber\\&\quad+ C(M)\varepsilon\|\nabla u\|_{L^2\big(0,T;L^2(\Gamma_t(2\delta))\big)}\|\nabla\bar{\we}\|_{L^2\big(0,T;L^2(\Omega)\big)}
\nonumber\\&=- \underbrace{\int_0^{T}\int_{\Gamma_t(2\delta)}\zg\theta_0'(\rho)\partial_{\nn} u\Div_{\btau}\bar{\we}dx dt}_{=: I}+ C\|\nabla_{\btau} u\|_{L^2\big(0,T;L^2(\Gamma_t(2\delta))\big)}\|\nabla\bar{\we}\|_{L^2\big(0,T;L^2(\Omega)\big)}
\nonumber\\&\quad+ C(M)\varepsilon\|\nabla u\|_{L^2\big(0,T;L^2(\Gamma_t(2\delta))\big)}\|\nabla\bar{\we}\|_{L^2\big(0,T;L^2(\Omega)\big)}.\label{w1est-3}
\end{align}
Here we have used the fact that  $0=\Div\bar{\we}=\big(\nn\otimes\nn:\nabla\bar{\we}\big)+\Div_{\btau}\bar{\we}$.
With the aid of \eqref{decompose u} in $I$ and using similar arguments as in the proof of Lemma 3.4 in \cite{AbelsMarquardt1} one shows 
\begin{align}\label{w1est-6}
I\leq& C\bigg(\|Z\|_{L^2\big(0,T;H^1(\Gamma)\big)}+\|\psi^R\|_{L^2\big(0,T;H^1(\Gamma_t(2\delta))\big)}\bigg)\|\nabla\bar{\we}\|_{L^2\big(0,T;L^2(\Omega)\big)}
\\\nonumber&+\frac{C}\varepsilon\bigg(\sup_{t\in[0,T]}\sup_{s\in\Gamma}\int_{I_{\varepsilon}^{s,t}}|\partial_{\tilde{\rho}} \Psi(\tilde{\rho},s,t)|^2J(\varepsilon(\tilde{\rho}+\hat{h}_\varepsilon),t)d\tilde{\rho}\bigg)^{\frac{1}{2}}\|Z\|_{L^2\big(0,T;H^1(\Gamma)\big)}\|\nabla\bar{\we}\|_{L^2(\Omega\times (0,T))}.
\end{align}
Thanks to \eqref{f2u estimate-2} we then arrive at
\begin{align}
I\leq C\bigg(\bigg\|\frac{\Lambda_\varepsilon}{\varepsilon}\bigg\|_{L^1(0,T)}^{\frac{1}{2}}+\|u\|_{L^2\big(0,T;L^2(\Gamma_t(2\delta))\big)}\bigg)\|\nabla\bar{\we}\|_{L^2(\Omega\times (0,T))}
,\label{w1est-7}
\end{align}
where
\begin{equation*}
  \Lambda_\eps = \int_{\Gamma_t(\delta)}\left(\eps |\nabla u|^2 + \frac1\eps f''(c_A^\eps(x,t))u(x,t)^2\right) dx.
\end{equation*}
Combining \eqref{w1est-7} with \eqref{w1est-3} one has
\begin{align}
&\varepsilon\int_0^{T}\bigg|\int_{\Gamma_t(2\delta)}\nabla u\otimes \nabla \tilde{c}_A:\nabla\bar{\we}dx\bigg| dt
\nonumber\\&\leq C\bigg(\bigg\|\frac{\Lambda_\varepsilon}{\varepsilon}\bigg\|_{L^1(0,T)}^{\frac{1}{2}}+\|u\|_{L^2\big(0,T;L^2(\Gamma_t(2\delta))\big)}\bigg)\|\nabla\bar{\we}\|_{L^2(\Omega\times (0,T))}
\nonumber\\& \quad+ C\|\nabla_{\btau} u\|_{L^2\big(0,T;L^2(\Gamma_t(2\delta))\big)}
\|\nabla\bar{\we}\|_{L^2(\Omega\times (0,T))}
  \nonumber\\&\quad+C(R)\varepsilon\|\partial_{\nn} u\|_{L^2\big(0,T;L^2(\Gamma_t(2\delta))\big)}\big)\|\nabla\bar{\we}\|_{L^2(\Omega\times (0,T))}
  \nonumber\\&\leq C(R)\varepsilon^{N+\frac12}\|\nabla\bar{\we}\|_{L^2(\Omega\times (0,T))}.\label{w1est-8}
\end{align}
{Thirdly}, we note that
\begin{align*}
\|\nabla u\|_{L^4(\Omega)}\leq C\bigg(\|\nabla u\|_{L^2(\Omega)}^{\frac{1}{2}}\|\Delta u\|_{L^2(\Omega)}^{\frac{1}{2}}+\|\nabla u\|_{L^2(\Omega)}\bigg).
\end{align*}
Therefore
\begin{align}
\|\nabla u\|_{L^4(\Omega\times(0,T))}&\leq C\bigg(\|\nabla u\|_{L^{\infty}(0,T;L^2(\Omega))}^{\frac{1}{2}}\|\Delta u\|_{L^2(\Omega\times(0,T))}^{\frac{1}{2}}+T_0^{\frac{1}{4}}\|\nabla u\|_{L^{\infty}(0,T;L^2(\Omega))}\bigg)
\nonumber\\ &\leq C(R)\varepsilon^{N-\frac{3}{2}}.\label{est:ul4l4}
\end{align}
Moreover, since
\begin{equation*}
  \nabla (c_A-\tilde{c}_A)= \zeta(d_\Gamma) \left(\eps^{N-\frac12}\theta_0'(\rho) +\eps^{N+\frac32}\partial_\rho \tc_2(\rho,s,t)\right) \nabla (h_{N+\frac12}(S(x,t),t)) + O(\eps^{N-\frac32})
\end{equation*}
in $L^\infty((0,T)\times \Omega)$, we obtain
\begin{equation*}
  \|\nabla (c_A-\tilde{c}_A)\|_{L^4(\Omega\times (0,T))}\leq C(M) \eps^{N-\frac32}.
\end{equation*}
Consequently, if $N\geq 3$, we conclude
\begin{align}\nonumber
  &\eps\int_{0}^T\bigg|\int_{\Omega}\nabla u\otimes^s \nabla (c_A-\tilde{c}_A):\nabla\bar{\we}dx\bigg|dt\\
  &\qquad \leq C\varepsilon\|\nabla u\|_{L^4(\Omega\times(0,T))}\|\nabla (c_A-\tilde{c}_A)\|_{L^4(\Omega\times (0,T))}\|\nabla \bar{\we}\|_{L^2(\Omega\times (0,T))}
 \nonumber\\&\qquad \leq C(R,M)\varepsilon^{2N-2}\|\nabla \bar{\we}\|_{L^2(\Omega\times (0,T))}\leq C(R,M)\varepsilon^{N+1}\|\nabla \bar{\we}\|_{L^2(\Omega\times (0,T))}
.\label{w2est-1}
\end{align}
Here we have used $X_T\hookrightarrow C\big([0,T];H^{\frac{3}{2}}(\T^1)\big)\hookrightarrow C\big([0,T];W^{1,p}(\T^1)\big)$ for all $1\leq p<+\infty$.

Since $N\geq 3$, \eqref{w1est-1} is a consequence of \eqref{w1est-11}, \eqref{w1est-8}, \eqref{w2est-1} and \eqref{assumptions}.
\smallskip

\noindent
\emph{Step 2:} {In this step we show}
\begin{align}
\int_0^T\bigg|\int_{\Omega}\bigg(\big(\nu(c_\eps)-\nu(c_A)\big) D\ve_A:\nabla\bar{\we}\bigg)dx\bigg|dt\leq C(R,M)\varepsilon^{N+1}\|\nabla\bar{\we}\|_{L^2\big(0,T;L^2(\Omega)\big)}.\label{w1est-0}
\end{align}
We use the decomposition $ D\ve_A=D\tilde{\ve}_A+D(\ve_A-\tilde{\ve}_A)$ and estimate each term separately.

{Firstly}, we note that  the leading order of {$\tilde{\ve}_A$} is
\begin{align}
\ve_0(x,t)=\zg \tilde{\ve}_0(\rho,x,t)+(1-\zg )\left(\ve_0^{+}(x,t)\chi_+ +\ve_0^-(x,t)\chi_-\right)\label{leading order velocity}
\end{align}
with $\tilde{\ve}_0$ defined in \eqref{inn-0-ve-express}. It follows from  \eqref{leading order velocity}, \eqref{inn-0-ve-express} and that $\ve_0^\pm, h_k$, $k=0,\ldots,N$, are smooth that one has
\begin{align}
|\nabla\ve_0|&\leq \frac{1}{\varepsilon}\zg\left|\frac{\ve_0^+-\ve_0^-}{d_{\Gamma}}\cdot d_{\Gamma}\eta_{\nu}'(\rho)\right|+\zg\big|(\ve_0^+-\ve_0^-)\eta_{\nu}'(\rho) \nabla^{\Gamma}h_{\varepsilon}\big|+C\nonumber\\
&\leq C\zg\big|\eta_{\nu}'(\rho)(\rho+h_{\varepsilon})\big|+C\leq \tilde{C}.
\end{align}
From this it easily follows that $\nabla \tilde{\ve}_A$ is uniformly bounded.
Accordingly we get
\begin{align}
\int_0^T\bigg|\int_{\Omega}\bigg(\big(\nu(c_\eps)-\nu(c_A)\big) D\tilde{\ve}_A:\nabla\bar{\we}\bigg)dx\bigg|dt\leq CT^{\frac{1}{2}}\| u\|_{L^\infty\big(0,T;L^2(\Omega)\big)}\|\nabla\bar{\we}\|_{L^2(\Omega\times (0,T))},\label{w1est-00}
\end{align}
here we have used $|\nu(c_\eps)-\nu(c_A)|\leq \|\nu'\|_{L^\infty(\R)}|u|$.

{Secondly}, since
\begin{equation*}
  \|\nabla (\ve_A-\tilde{\ve}_A)\|_{L^4(0,T;L^2(\Omega))} \leq C(M)\eps^{N}
\end{equation*}
due to  $X_T\hookrightarrow L^4(0,T;H^2(\T^1))$,
we conclude
\begin{align}
  &\eps\int_{0}^T\bigg|\int_{\Omega}\left((\nu(c_\eps)-\nu(c_A))D(\ve_A-\tilde{\ve}_A):\nabla\bar{\we}\right)dx\bigg|dt\nonumber\\
  &\leq C(R,M)\eps T_0^{\frac{1}{4}}\|\nabla (\ve_A-\tilde{\ve}_A)\|_{L^4(0,T;L^2(\Omega))} \|\bar{\we}\|_{L^2(0,T;H^1(\Omega))} \nonumber\\
  &\leq C(R,M) \eps^{N+1}\|\bar{\we}\|_{L^2(0,T;H^1(\Omega))}.\label{w2est-00}
\end{align}
Thanks to \eqref{w1est-00}, \eqref{w2est-00} and \eqref{assumptions} we derive \eqref{w1est-0}.

\smallskip

\noindent
\emph{Step 3:}
Similarly as in the proof of \eqref{w2est-1} we conclude
\begin{align}\nonumber
  &\varepsilon\int_{0}^T\bigg|\int_{\Omega}\bigg(\nabla u\otimes\nabla u:\nabla\bar{\we}\bigg)dx\bigg|dt\leq C\varepsilon\|\nabla u\|_{L^4(0,T;L^4(\Omega))}^2\|\nabla \bar{\we}\|_{L^2(\Omega\times (0,T))}
\nonumber\\&\qquad\qquad\qquad\qquad \qquad\qquad\qquad\leq C(R)\varepsilon^{2N-2}\|\nabla \bar{\we}\|_{L^2(\Omega\times (0,T))}
.\label{w2est-2}
\end{align}
Since $\|R_\eps^2\|_{L^2(0,T;(H^1(\Omega)^2)')}\leq  C(M) \varepsilon^{N+1}$ and
\begin{equation*}
  \|R_\eps^1\|_{L^2(0,T;(H^1(\Omega)^2)')}\leq  C \varepsilon^{N+\frac12}\|\we|_{\Gamma}\|_{L^2(0,T;L^2(\Gamma_t))}\leq C' \varepsilon^{N+\frac12}\|\we\|_{L^2(0,T;L^2(\Omega))}^{\frac12}\|\we\|_{L^2(0,T;H^1(\Omega))}^{\frac12},
\end{equation*}
cf. e.g.~\cite[Lemma~2.4]{Analyticity},
we obtain  
by noting that $\we=\frac{\tilde{\we}}{\eps^{N+\frac12}}=\frac{\bar{\we}-\we_0}{\eps^{N+\frac12}}$
\begin{align}
  &\int_{0}^T\bigg|\int_{\Omega}R_\eps^1\cdot\bar{\we}dx\bigg|dt\leq C\varepsilon^{N+\frac12}\|\we\|_{L^2(0,T;L^2(\Omega))}^{\frac12}\|\we\|_{L^2(0,T;H^1(\Omega))}^{\frac12}\|\bar{\we}\|_{L^2(0,T;H^1(\Omega))}
  \nonumber\\&
\qquad\leq C\|\tilde{\we}\|_{L^2(0,T;L^2(\Omega))}^{\frac12}\|\tilde{\we}\|_{L^2(0,T;H^1(\Omega))}^{\frac12}\|\bar{\we}\|_{L^2(0,T;H^1(\Omega))}
\nonumber\\&
\qquad\leq C\|\bar{\we}\|_{L^2(0,T;L^2(\Omega))}^{\frac12}\|\bar{\we}\|_{L^2(0,T;H^1(\Omega))}^{\frac12}\|\bar{\we}\|_{L^2(0,T;H^1(\Omega))}+C\|\we_0\|_{L^2(0,T;H^1(\Omega))}\|\bar{\we}\|_{L^2(0,T;H^1(\Omega))}
  \nonumber\\&\nonumber
  \qquad\leq C{\left(\|\bar{\we}\|_{L^2(0,T;L^2(\Omega))}+\|\we_0\|_{L^2(0,T;H^1(\Omega))}\right)}^{\frac12}\|\bar{\we}\|_{L^2(0,T;H^1(\Omega))}^{\frac32}\\
  &\qquad \quad +  C(M) \varepsilon^{N+1}\|\bar{\we}\|_{L^2(0,T;H^1(\Omega))}\label{w2est-3'}
\end{align}
and
\begin{align}
\int_{0}^T\bigg|\int_{\Omega}R_\eps^2\cdot\bar{\we}dx\bigg|dt&\leq \|R_\eps^2\|_{L^2(0,T;(H^1(\Omega)^2)')}\|\bar{\we}\|_{L^2(0,T;H^1(\Omega))} \nonumber\\&\leq C(M) \varepsilon^{N+1}\|\bar{\we}\|_{L^2(0,T;H^1(\Omega))}.\label{w2est-3}
\end{align}
Moreover,
\begin{align}
 &\int_0^T\left|\int_\Omega \big(\bar{\we}\cdot \nabla (\ve_A-\tilde{\ve}_A)\big)\cdot \bar{\we} dx\right|dt\leq \|\nabla (\ve_A-\tilde{\ve}_A)\|_{L^4(0,T;L^2(\Omega))}\|\bar{\we}\|^2_{L^{\frac83}(0,T; L^4(\Omega))}\nonumber\\&\qquad \leq C(M)\varepsilon^N\|\bar{\we}\|_{L^\infty(0,T;L^2(\Omega))}^{\frac12}\|\bar{\we}\|_{L^2(0,T;L^2(\Omega))}^{\frac12}\|\bar{\we}\|_{L^2(0,T;H^1(\Omega))}.\label{w2est-4}
\end{align}
Due to \eqref{est:w0} one has
\begin{align}\nonumber
  &\int_0^T \left|\weight{\partial_t \we_0, \bar{\we}}_{V',V} \right|dt +\int_{0}^T\bigg|\int_{\Omega}\bigg(  (\ve_\eps\cdot \nabla \we_0-\we_0\cdot\nabla \ve_A)\cdot\bar{\we}  +\big(2\nu(c_\eps) D\we_0:\nabla\bar{\we}\big)\bigg)dx\bigg|dt\\
  &\quad \leq C(M)\varepsilon^{N+1}\|\nabla\bar{\we}\|_{L^2\big(0,T;L^2(\Omega)\big)}+C(M) \eps^{N+1}\|\bar{\we}\|^{\frac{1}{2}}_{L^{\infty}\big(0,T;L^2(\Omega)\big)}\|\bar{\we}\|^{\frac{1}{2}}_{L^2\big(0,T;H^1(\Omega)\big)}.
 \label{w2est-5}
\end{align}
Here we have used
\begin{align*}
\int_{0}^T\bigg|\int_{\Omega}(\ve_\eps\cdot \nabla \we_0)\cdot\bar{\we}  dx\bigg|dt&\leq\|\ve_\eps\|_{L^4\big(0,T;L^4(\Omega)\big)}\|\nabla \we_0\|_{L^2\big(0,T;L^2(\Omega)\big)}\|\bar{\we}\|_{L^4\big(0,T;L^4(\Omega)\big)}
  \nonumber\\&\leq C(M) \eps^{N+1}\|\bar{\we}\|^{\frac{1}{2}}_{L^{\infty}\big(0,T;L^2(\Omega)\big)}\|\bar{\we}\|^{\frac{1}{2}}_{L^2\big(0,T;H^1(\Omega)\big)}
\end{align*}
{(due to the energy estimate \eqref{eq:EnergyEstim})} and
  \begin{align*}
&\int_{0}^T\bigg|\int_{\Omega}(\we_0\cdot\nabla \ve_A)\cdot\bar{\we}dx\bigg|dt\leq\|\we_0\|_{L^2\big(0,T;L^2(\Omega)\big)}\|\nabla\tilde{\ve}_A\|_{L^{\infty}\big(0,T;L^{\infty}(\Omega)\big)}\|\bar{\we}\|_{L^2\big(0,T;L^2(\Omega)\big)}
\nonumber\\&\qquad\quad +\|\we_0\|_{L^4\big(0,T;L^4(\Omega)\big)}\|\nabla(\ve_A-\tilde{\ve}_A)\|_{L^2\big(0,T;L^2(\Omega)\big)}\|\bar{\we}\|_{L^4\big(0,T;L^4(\Omega)\big)}
\nonumber\\&\qquad\leq C(M) T^{\frac{1}{2}}\eps^{N+1}\|\bar{\we}\|_{L^\infty\big(0,T;L^2(\Omega)\big)}+ \eps^{2N+\frac{3}{2}}\|\bar{\we}\|^{\frac{1}{2}}_{L^{\infty}\big(0,T;L^2(\Omega)\big)}\|\bar{\we}\|^{\frac{1}{2}}_{L^2\big(0,T;H^1(\Omega)\big)}
  \end{align*}
  {as well as
  \begin{align*}
    \int_{0}^T\bigg|\int_{\Omega}2\nu(c_\eps) D\we_0:\nabla\bar{\we}\, dx\bigg|dt&\leq C\|\we_0\|_{L^2\big(0,T;H^1(\Omega)\big)}\|\nabla \bar{\we}\|_{L^2\big(0,T;L^2(\Omega)\big)}\\
    &\leq C(M) \eps^{N+1}\|\nabla \bar{\we}\|_{L^2\big(0,T;L^2(\Omega)\big)}.
  \end{align*}}
Plugging \eqref{w1est-1}, \eqref{w1est-0}, \eqref{w2est-2}-\eqref{w2est-5} and utilizing \eqref{initial assumption}, \eqref{Korninequality} and Young's  inequality we can derive
 \begin{align}
   &\|\bar{\we}\|_{L^\infty(0,T;L^2(\Omega))}^2+\|\bar{\we}\|_{L^2(0,T;H^1(\Omega))}^2\nonumber\\
   &\quad \leq C_1(R) \eps^{2N+1}+C_2(R,M) \eps^{2N+2}+ C_3(M) \eps^{4N} \int_0^T \|\bar{\we}(t)\|_{L^2(\Omega)}^2\, dt.\nonumber 
 \end{align}
 Then the Gronwall inequality and \eqref{est:w0}-\eqref{est:w00} yield
 \begin{align}
   &\|\we\|_{L^\infty(0,T;L^2(\Omega))}+\|\we\|_{L^2(0,T;H^1(\Omega))}\nonumber\\
   &\quad \leq \exp (C_3(M)\eps^{4N} T_0) C_1(R) \eps^{N+\frac12}+C_2(R,M) \eps^{N+1}.\nonumber
 \end{align}
 for all $T\in (0,T_\eps]$.  Now we  choose $\eps_0\in (0,1)$ (in dependence of $M$) so small that $C_3(M) \eps_0^{4N}\leq 1$ and $C_2(R,M)\eps^{\frac12} \leq e C_1(R)$. Then
 \begin{align}
   &\|\we\|_{L^\infty(0,T;L^2(\Omega))}+\|\we\|_{L^2(0,T;H^1(\Omega))} \leq 2e C_1(R) \eps^{N+\frac12}.\label{w1barestimate}
 \end{align}
 for all $T\in (0,T_\eps]$ and $\eps \in (0,\eps_0]$. Moreover, since $h_{N+\frac12}\in X_{T}$ is determined by \eqref{eq:hN12a} with $\ue$ as in \eqref{eq:Defue} we have
 \begin{equation*}
   \|h_{N+\frac12}\|_{X_T}\leq C_4(T_0)\eps^{-N-\frac12} \|\we\|_{L^2(0,T;H^1(\Omega))}\leq 2e C_4(T_0) C_1(R).
 \end{equation*}
 Hence we can finally choose $M= 2e C_4(T_0) C_1(R)$. This also determines $\eps_0 \in (0,1)$.
Altogether this implies \eqref{w1estimate}.
\end{proof}

In the following let $M$ and $\eps_0>0$ be as in Theorem~\ref{thm:weEstim} (in dependence on $R$).
Noting  $\we=\frac{\tilde{\we}}{\varepsilon^{N+\frac{1}{2}}}$, we get for $T\in(0,T_\varepsilon)$
 \begin{equation}
    \|\we\|_{L^\infty(0,T;L^2(\Omega))}+\|\we\|_{L^2(0,T;H^1(\Omega))}\leq C(R).\label{neww1estimate}
  \end{equation}

\begin{theorem} For $T\in(0,T_\varepsilon)$ there holds
 \begin{alignat}{2}
\int_0^{T}\bigg|\int_{\Omega}\big(\we-\we|_{\Gamma}\big)\cdot \nabla c_Au\, dx\bigg|dt&\leq C(R)\varepsilon^{N+1}.\label{est:diffw1couple}
  \end{alignat}
\end{theorem}
\begin{proof}
We set
\begin{alignat}{1}
  c^{(0)}(x,t)&=\zg(x,t) \left(\theta_0(\rho) +\eps^{N-\frac12} \theta_0'(\rho)h_{N+\frac12}(S(x,t),t) \right)
  \nonumber\\&\quad+(1-\zg(x,t) )\(c_A^+(x,t)\chi_+(x,t) +c_A^-(x,t)\chi_-(x,t) \).\label{leading order app}
\end{alignat}
Using that
   \begin{align}
   \bigg(\theta_0(\rho)-\big(\pm1\big|_{\Omega^{\pm}(t)}\big)\bigg){\nabla \zeta(d_\Gamma)}=O(e^{-\frac{\alpha\delta}{2\varepsilon}}), \label{outer decay}
   \end{align}
due to \eqref{eq:LowerBdd},   we  get
 \begin{align}
&\int_{\Gamma_t(2\delta)}\big(\we-\we|_{\Gamma}\big)\cdot\nabla c^{(0)}udx=J_1- J_{21}+ J_{22}+O(e^{-\frac{\alpha\delta}{2\varepsilon}})\|\we\|_{H^1(\Omega)}\|u\|_{L^2(\Gamma_t(2\delta))},
\label{est:diffw1-1}
\end{align}
where
\begin{align*}
  J_1&:= \varepsilon^{-1}\int_{\Gamma_t(2\delta)}\zg\big(\we-\we|_{\Gamma}\big)\cdot\nn\big(\theta_0'(\rho)+\theta_0''(\rho)\varepsilon^{N-\frac{1}{2}}h_{N{+}\frac{1}{2}}\big)udx,\\
  J_{21}&:= \int_{\Gamma_t(2\delta)}\zg\big(\we-\we|_{\Gamma}\big)\cdot\nabla_{\btau}h_{\varepsilon}\big(\theta_0'(\rho)+\theta_0''(\rho)\varepsilon^{N-\frac{1}{2}}h_{N+\frac{1}{2}}\big)udx,\\
  J_{22} &:=\int_{\Gamma_t(2\delta)}\zg\big(\we-\we|_{\Gamma}\big)\cdot\varepsilon^{N-\frac{1}{2}}\nabla_{\btau}h_{N+\frac{1}{2}}\theta_0'(\rho)udx.
\end{align*}
It follows from $\|h_{N+\frac12}\|_{X_T}\leq M$ and $X_T\hookrightarrow BUC\big([0,T];C^0(\T^1)\big)$ that
\begin{align}
\big\|\eps^{N-\frac12} h_{N+\frac12} (s,t)\big\|_{C^0\big([0,T];C^0(\T^1)\big)}\leq C(M).\label{bound of heps}
\end{align}
Based on \eqref{bound of heps} and the procedure of proving \cite[Lemma 5.1]{StokesAllenCahn} one has 
 \begin{align}
   \int_{0}^T |J_1|dt&\leq C\varepsilon^{\frac{1}{2}}\int_{0}^T \|\we(\cdot,t)\|_{L^2(\Omega)}^{\frac{1}{2}}\|\we(\cdot,t)\|_{H^1(\Omega)}^{\frac{1}{2}}\big(\|\nabla_{\tau}u\|_{L^2(\Gamma_t(2\delta))}+\|u\|_{L^2(\Gamma_t(2\delta))}\big)dt
   \nonumber\\&\leq C(R,M)T^{\frac{1}{4}}\big(\|\nabla_{\btau}u\|_{L^2(0,T;L^2(\Omega))}+T^{\frac{1}{2}}\|u\|_{L^\infty(0,T;L^2(\Omega))}\big) \nonumber\\&\quad\cdot\|\we\|_{L^\infty(0,T;L^2(\Omega))}^{\frac{1}{2}}\|\we\|_{L^2(0,T;H^1(\Omega))}^{\frac{1}{2}}
   \leq C(R,M)T^{\frac{1}{4}}(1+T^{\frac{1}{2}})\varepsilon^{N+1}\label{est:j1}
  \end{align}
and
\begin{align}
 \int_{0}^T  |J_{21}+J_{22}|dt&\leq C\varepsilon\int_{0}^T \|\we\|_{H^1(\Gamma_t(2\delta))}\|u\|_{L^2(\Gamma_t(2\delta))}\big(1+\varepsilon^{N-\frac{1}{2}}\|h_{N+\frac{1}{2}}\|_{C^1(\T^1)}\big)dt
 \nonumber\\&\leq C\varepsilon\big(\|u\|_{L^2(0,T;L^2(\Omega))}+\varepsilon^{N-\frac{1}{2}}\|h_{N+\frac{1}{2}}\|_{L^2(0,T;C^1(\T^1))}\|u\|_{L^\infty(0,T;L^2(\Omega))}\big) \nonumber\\&\quad\cdot\|\we\|_{L^2(0,T;H^1(\Omega))}
  \leq C(R,M)\varepsilon^{N+\frac{3}{2}}.\label{est:j2}
  \end{align}
Using \eqref{est:j1} and \eqref{est:j2} in \eqref{est:diffw1-1}
 we conclude
 \begin{align}
\int_{0}^T \bigg|\int_{\Gamma_t(2\delta)}\big(\we-\we|_{\Gamma}\big)\cdot\nabla c^{(0)} u\, dx\bigg|dt\leq C(R,M,T_0)\varepsilon^{N+1}.
\label{est:diffw1-2}
\end{align}
The same estimate for $\nabla (c_A- c^{(0)})$ can be obtained in a straight forward manner since all terms are of higher order in $\eps$ compared to $\nabla c^{(0)}$. This yields \eqref{est:diffw1couple}.
\end{proof}

\subsection{Proof of Theorem \ref{thm:main}}

In order to estimate the error due to linearization of $f'(c_\eps)$ we use:
\begin{prop}
\begin{align}
\int_0^{T_\varepsilon}\int_{\Omega} |u|^3dxdt\leq C(R)T_\varepsilon^{\frac{1}{2}}\varepsilon^{3N+1}.\label{est:nonlinear}
 \end{align}
\end{prop}
\begin{proof}
 Thanks to \cite[Lemma 3.9]{AbelsMarquardt1}, H\"{o}lder's inequality and assumption \eqref{assumptions} one has
 \begin{align}
&\int_0^{T_\varepsilon}\int_{\Gamma_t(\delta)} |u|^3dx\, dt\leq C\big(\|u\|_{L^2(0,T_\varepsilon;L^2(\Gamma_t(\delta)))}+\|\nabla_{\btau}u\|_{L^2(0,T_\varepsilon;L^2(\Gamma_t(\delta)))}\big)^{\frac{1}{2}}
\nonumber\\&\quad\cdot\big(\|u\|_{L^2(0,T_\varepsilon;L^2(\Gamma_t(\delta)))}+\|\partial_{\nn} u\|_{L^2(0,T;L^2(\Gamma_t(2\delta)))}\big)^{\frac{1}{2}}\|u\|_{L^4(0,T_\varepsilon;L^2(\Gamma_t(\delta)))}^2
\nonumber\\&\leq C(R)T_\varepsilon^{\frac{1}{2}}\varepsilon^{3N+1}.\label{est:nonlinear-1}
 \end{align}
 Moreover, due to the Gagliardo-Nirenberg inequality and assumption \eqref{assumptions} one gets
  \begin{align}
\int_0^{T_\varepsilon}\int_{\Omega\setminus\Gamma_t(\delta)} |u|^3dxdt&\leq C\| u\|_{L^2(0,T_\varepsilon;H^1(\Omega\setminus\Gamma_t(\delta)))}\|u\|_{L^4(0,T_\varepsilon;L^2(\Omega\setminus\Gamma_t(\delta)))}^2
\nonumber\\&\leq C(R)\big(T_\varepsilon^{\frac{1}{2}}+T_\varepsilon\big)\varepsilon^{3N+\frac{3}{2}}.\label{est:nonlinear-2}
 \end{align}
 Consequently we arrive at \eqref{est:nonlinear}.
\end{proof}

Noting that  \begin{equation}\label{velocity decompose new}
   \ve_\eps= \ve_A+ \varepsilon^{N+\frac{1}{2}}\we
 \end{equation}
with the help of \eqref{eq:NSAC3} and \eqref{eq:cA} we then find  that
 \begin{alignat}{2}\label{eq:u}
     \partial_t u+\ve_\eps\cdot \nabla u+ \eps^{N+\frac12}\big(\we-\we|_{\Gamma}\big)\cdot \nabla c_A&= \Delta u-\frac{f''(c_A)}{\eps^2}u -\eps^{-2}\mathcal{N}(c_A,u)- S_\eps,
  \end{alignat}
where
\begin{align*}
  \mathcal{N}(c_A,u)=f'(c_\eps)-f'(c_A)-f''(c_A)u.
 \end{align*}
 Multiplying \eqref{eq:u} by $u$ and integrating by  parts yield that
  \begin{alignat}{2}\label{ineq:u energy-00}
    &\frac{1}{2}\frac{d}{dt}\int_{\Omega} u^2dx+ \int_{\Omega}\big( |\nabla u|^2+\frac{f''(c_A)}{\eps^2}u^2\big) \sd x\nonumber\\&\leq\eps^{N+\frac12}\bigg|\int_{\Omega}\big(\we-\we|_{\Gamma}\big)\cdot \nabla c_Audx\bigg|+\eps^{-2}C\int_{\Omega}|u|^3dx+\bigg|\int_{\Omega} S_\eps udx\bigg|.
  \end{alignat}
  Here we have used
  \begin{align*}
  \int_{\Omega} \mathcal{N}(c_A,u)udx\geq-C\int_{\Omega} |u|^3dx.
  \end{align*}
  By using  the spectral estimate of Lemma \ref{thm:Spectral} in \eqref{ineq:u energy-00}, one gets
  \begin{alignat}{2}\label{ineq:u energy}
    &\frac{1}{2}\frac{d}{dt}\int_{\Omega} u^2dx-C_L\int_\Om u^2 \sd x + \int_{\Om\setminus \Gamma_t(\delta)} |\nabla u|^2\sd x +  \int_{\Gamma_t({\delta})} |\nabla_\btau u|^2\sd x\nonumber\\&\leq\eps^{N+\frac12}\bigg|\int_{\Omega}\big(\we-\we|_{\Gamma}\big)\cdot \nabla c_Audx\bigg|+\eps^{-2}C\int_{\Omega}|u|^3dx+\bigg|\int_{\Omega} S_\eps udx\bigg|,
  \end{alignat}
Applying  Gronwall's inequality  and using \eqref{initial assumption-0}  we obtain for all $T\in [0,T_\varepsilon]$
  \begin{alignat}{2}\label{ineq:u energy-1}
    &\sup_{0\leq t\leq T}\frac{1}{2}\int_{\Omega} u^2dx + \int_{0}^T\int_{\Om\setminus \Gamma_t(\delta)} |\nabla u|^2dx dt +  \int_{0}^T\int_{\Gamma_t({\delta})} |\nabla_\btau u|^2d x dt\nonumber\\&\leq e^{C_LT_0}\bigg(\frac{1}{2}\int_{\Omega} u^2|_{t=0}dx +\eps^{N+\frac12}\int_{0}^T\bigg|\int_{\Omega}\big(\we-\we|_{\Gamma}\big)\cdot \nabla c_Au\,dx\bigg|dt\nonumber\\&\qquad\qquad+\eps^{-2}\int_{0}^T\int_{\Omega}|u|^3dxdt+\int_{0}^T\bigg|\int_{\Omega} S_\eps u\, dx\bigg|dt\bigg)
    \nonumber\\&\leq \theta^2\frac{R^2}{8}\eps^{2N+1}+C(R,T_0)\big(\eps^{2N+\frac32}+\eps^{3N-1}+\eps^{2N+\frac{3}{2}}\big).
  \end{alignat}
Hence, if $\varepsilon\in(0,\varepsilon_0)$ and $\varepsilon_0>0$ is sufficiently small, 
  \begin{align}
    \theta^2\frac{R^2}{8}\eps^{2N+1}+C(R,T_0)\big(\eps^{2N+\frac32}+\eps^{3N-1}+\eps^{2N+\frac{3}{2}}\big)\leq \theta^2\frac{R^2}{4}\eps^{2N+1}.\label{ineq:u energy-111}
  \end{align}
Thus there holds
 \begin{alignat}{2}\label{ineq:u energy-2}
    \sup_{0\leq t\leq T}\frac{1}{2}\int_{\Omega} u^2(x,t)dx + \int_{0}^T\int_{\Om\setminus \Gamma_t(\delta)} |\nabla u|^2dx dt +  \int_{0}^T\int_{\Gamma_t({\delta})} |\nabla_\btau u|^2d x dt\leq \theta^2\frac{R^2}{4}\eps^{2N+1}.
  \end{alignat}
We note that  \eqref{ineq:u energy-00} implies  for $0\leq T\leq T_\varepsilon$
  \begin{alignat}{2}\label{ineq:u energy-3}
    &\int_0^T \int_{\Omega}\big( |\nabla u|^2+\frac{f''(c_A)}{\eps^2}u^2\big) \sd x dt\nonumber\\&\leq\frac{1}{2}\int_{\Omega} u^2|_{t=0}dx +\eps^{N+\frac12}\int_{0}^T\bigg|\int_{\Omega}\big(\we_1-\we_1|_{\Gamma}\big)\cdot \nabla c_Audx\bigg|dt\nonumber\\&\qquad\qquad+\eps^{-2}\int_{0}^T\int_{\Omega}|u|^3dxdt+\int_{0}^T\bigg|\int_{\Omega}S_\eps udx\bigg|dt
    \nonumber\\&\leq\theta^2\frac{R^2}{8}\eps^{2N+1}+C(R,T_0)(\eps^{2N+\frac32}+\eps^{3N-1}).
  \end{alignat}
Then for small  $\varepsilon_0>0$, if $T\leq T_\varepsilon$ and $\varepsilon\in(0,\varepsilon_0)$, we have
\begin{alignat}{2}\label{ineq:u energy-4}
    \int_0^T \int_{\Omega}\big( |\nabla u|^2+\frac{f''(c_A)}{\eps^2}u^2\big) \sd x dt\leq\theta^2\frac{R^2}{4}\eps^{2N+1}
  \end{alignat}
  and
  \begin{alignat}{2}\label{ineq:u energy-5}
    \eps^2\int_0^T \int_{\Omega}|\partial_{\nn} u|^2\sd x dt&\leq\eps^2\int_0^T \int_{\Omega}|\nabla u|^2\sd x dt
    \nonumber\\&\leq\int_0^T \int_{\Omega}\big( \eps^2|\nabla u|^2+f''(c_A)u^2\big) \sd x dt+C_0\int_0^T \int_{\Omega}u^2\sd x dt\nonumber\\&\leq\frac{R^2}{4}\eps^{2N+3}+C_0T_0\theta^2\frac{R^2}{2}\eps^{2N+1}
    \leq\frac{3R^2}{4}\eps^{2N+1}
  \end{alignat}
  where $C_0= -\min_{s\in\R} f''(s)$ and we choose $\theta\in (0,1]$ so small that $C_0 T_0 \theta^2\leq 1$.
 Multiplying \eqref{eq:u} by $-\varepsilon^4\Delta u$ and integrating by  parts yield that
  \begin{alignat}{2}\label{ineq:u energy-6}
    &\sup_{0\leq t\leq T}\frac{\varepsilon^4}{2}\int_{\Omega} |\nabla u|^2(x,t)dx+ \varepsilon^4\int_0^T\int_{\Omega}|\Delta u|^2dxdt
    \nonumber\\&\leq\frac{\varepsilon^4}{2}\int_{\Omega} |\nabla u|^2|_{t=0}dx+\varepsilon^2\int_0^T\int_{\Omega}\big|f''(c_A)u\Delta u\big|\sd x dt
    \nonumber\\&\quad+\varepsilon^4\int_0^T\int_{\Omega}\big|\ve_\eps\cdot \nabla u\Delta u  \big|dx dt+\eps^{N+\frac92}\int_0^T\int_{\Omega}\big|\big(\we-\we|_{\Gamma}\big)\cdot \nabla c_A\Delta u\big|dxdt
    \nonumber\\&\quad+\eps^{2}\int_0^T\int_{\Omega}\big|\mathcal{N}(c_A,u)\Delta u\big|dx dt+\varepsilon^4\int_0^T\bigg|\int_{\Omega}S_\eps\Delta udx\bigg|dt.
  \end{alignat}
 Thanks to H\"{o}lder's inequality and \eqref{ineq:u energy-2}  one gets
   \begin{alignat}{2}\label{ineq:u energy-7}
\varepsilon^2\int_0^T\int_{\Omega}\big|f''(c_A)u\Delta u\big|\sd x dt&\leq C\varepsilon^2 T_0^{\frac{1}{2}}\|u\|_{L^\infty(0,T;L^2(\Omega))}\|\Delta u\|_{L^2(0,T;L^2(\Omega))}
\nonumber\\&\leq CR\varepsilon^{N+\frac{5}{2}} \theta T_0^{\frac{1}{2}}\|\Delta u\|_{L^2(0,T;L^2(\Omega))}
  \end{alignat}
Noting \eqref{velocity decompose new}
 we then have
 \begin{alignat}{2}\nonumber
   &\varepsilon^4\int_0^T\int_{\Omega}\big|\ve_\eps\cdot \nabla u\Delta u  \big|dx dt\\\label{ineq:u energy-8}
   &\qquad \leq \varepsilon^4\int_0^T\int_{\Omega}\big|\ve_A\cdot \nabla u\Delta u  \big|dx dt+\eps^{N+\frac92}\int_0^T\int_{\Omega}\big|\we\cdot \nabla u\Delta u \big| dxdt.
  \end{alignat}
  It follows from H\"{o}lder's inequality, the Gagliardo-Nirenberg interpolation inequality and \eqref{ineq:u energy-2} {and \eqref{ineq:u energy-5}}  that
 \begin{alignat}{2}\label{ineq:u energy-9}
    \varepsilon^4\int_0^T\int_{\Omega}\big|\ve_A\cdot \nabla u\Delta u  \big|dx dt&\leq C  \varepsilon^4\|\nabla u\|_{L^2(0,T;L^2(\Omega))}\|\Delta u\|_{L^2(0,T;L^2(\Omega))}\nonumber\\&\leq C R\varepsilon^{N+\frac{7}{2}}\|\Delta u\|_{L^2(0,T;L^2(\Omega))}
  \end{alignat}
 and
    \begin{alignat}{2}\label{ineq:u energy-10}
  &\eps^{N+\frac92}\int_0^T\int_{\Omega}\big|\we\cdot \nabla u\Delta u \big| dxdt  \nonumber\\&\leq  \eps^{N+\frac92}\|\we\|_{L^4(0,T;L^4(\Omega))}\|\nabla u\|_{L^4(0,T;L^4(\Omega))}\|\Delta u\|_{L^2(0,T;L^2(\Omega))}
    \nonumber\\&\leq\eps^{N+\frac92}\|\we\|_{L^\infty(0,T;L^2(\Omega))}^{\frac{1}{2}}\|\we\|_{L^2(0,T;H^1(\Omega))}^{\frac{1}{2}}\|\nabla u\|_{L^\infty(0,T;L^2(\Omega))}^{\frac{1}{2}}
    \nonumber\\&\quad\cdot\| u\|_{L^2(0,T;H^2(\Omega))}^{\frac{1}{2}}\|\Delta u\|_{L^2(0,T;L^2(\Omega))}
     \nonumber\\&\leq C(R)\eps^{N+\frac92}\|\nabla u\|_{L^\infty(0,T;L^2(\Omega))}^{\frac{1}{2}}\|\Delta u\|_{L^2(0,T;L^2(\Omega))}^{\frac{3}{2}}.
  \end{alignat}
  Using \eqref{ineq:u energy-9}-\eqref{ineq:u energy-10} in \eqref{ineq:u energy-8} implies that
   \begin{alignat}{2}\label{ineq:u energy-11}
    \varepsilon^4\int_0^T\int_{\Omega}\big|\ve_\eps\cdot \nabla u\Delta u  \big|dx dt&\leq C R\varepsilon^{N+\frac{7}{2}}\|\Delta u\|_{L^2(0,T;L^2(\Omega))}\nonumber\\&\quad+C(R)\eps^{N+\frac92}\|\nabla u\|_{L^\infty(0,T;L^2(\Omega))}^{\frac{1}{2}}\|\Delta u\|_{L^2(0,T;L^2(\Omega))}^{\frac{3}{2}}.
  \end{alignat}
  Moreover, we have
    \begin{alignat}{2}\label{ineq:u energy-12}
    \eps^{N+\frac92}\int_0^T\int_{\Omega}\big|\big(\we-\we|_{\Gamma}\big)\cdot \nabla c_A\Delta u\big|dxdt
    \leq C(R)\eps^{N+\frac72}\|\Delta u\|_{L^2(0,T;L^2(\Omega))}
  \end{alignat}
 and
    \begin{alignat}{2}\label{ineq:u energy-13}
    &\eps^{2}\int_0^T\int_{\Omega}\big|\mathcal{N}(c_A,u)\Delta u\big|dx dt \leq C\eps^{2}\int_0^T\int_{\Omega}\big(|u^2\Delta u|+|u^3\Delta u|\big)dx dt
   \nonumber\\&\quad\leq\eps^{2}\big(\|u\|_{L^4(0,T;L^4(\Omega))}^2+\|u\|_{L^6(0,T;L^6(\Omega))}^3\big)\|\Delta u\|_{L^2(0,T;L^2(\Omega))}
   \nonumber\\&\quad\leq\eps^{2}\big(\|u\|_{L^\infty(0,T;L^2(\Omega))}+\|u\|_{L^\infty(0,T;L^2(\Omega))}^2\big)\|u\|_{L^2(0,T;H^1(\Omega))}\|\Delta u\|_{L^2(0,T;L^2(\Omega))}
   \nonumber\\&\quad \leq CR^2\eps^{2N+2}\|\Delta u\|_{L^2(0,T;L^2(\Omega))}
  \end{alignat}
 and
 \begin{alignat}{2}\label{ineq:u energy-14}
 \varepsilon^4\int_0^T\bigg|\int_{\Omega}S_\eps\Delta udx\bigg|dt\leq \varepsilon^4\|S_\eps\|_{L^2(\Omega \times (0,T))}\|\Delta u\|_{L^2(\Omega\times (0,T))}\leq C\varepsilon^{N+4}\|\Delta u\|_{L^2(\Omega\times (0,T))}.
  \end{alignat}
 Combining \eqref{initial assumption-0}, \eqref{ineq:u energy-7} and \eqref{ineq:u energy-11}-\eqref{ineq:u energy-14} with \eqref{ineq:u energy-6} leads to
  \begin{alignat}{2}\label{ineq:u energy-15}
    &\sup_{0\leq t\leq T}\frac{\varepsilon^4}{2}\int_{\Omega} |\nabla u|^2(x,t)dx+ \varepsilon^4\int_0^T\int_{\Omega}|\Delta u|^2dxdt
    \nonumber\\&\leq \frac{R^2}{8}\eps^{2\order+1}+C(R)\eps^{N+\frac92}\|\nabla u\|_{L^\infty(0,T;L^2(\Omega))}^{\frac{1}{2}}\|\Delta u\|_{L^2(0,T;L^2(\Omega))}^{\frac{3}{2}}
    \nonumber\\&\quad+\bigg(CR\theta\varepsilon^{N+\frac{5}{2}}+C(R,T_0)(\varepsilon^{N+\frac{7}{2}}+\varepsilon^{N+4}+\varepsilon^{2N+2})\bigg)\|\Delta u\|_{L^2(0,T;L^2(\Omega))}.
  \end{alignat}
 Applying Young's inequality in \eqref{ineq:u energy-15} we get for $T\leq T_\varepsilon$ and $\varepsilon\in (0,\varepsilon_0)$
 \begin{alignat}{2}\label{ineq:u energy-16}
    &\sup_{0\leq t\leq T}\frac{3\varepsilon^4}{8}\int_{\Omega} |\nabla u|^2(x,t)dx+ \frac{3\varepsilon^4}{8}\int_0^T\int_{\Omega}|\Delta u|^2dxdt
    \nonumber\\&\qquad\leq \frac{R^2}{8}\eps^{2\order+1}+CR^2\theta^2\eps^{2N+1}+ C(R,T_0)\eps^{2N+3} \leq \frac{R^2}{4}\eps^{2\order+1}
  \end{alignat}
  if we first choose $\theta>0$ small enough (which finally determines $R=R(\theta)$) and afterwards $\varepsilon_0>0$ sufficiently small. 
  Thus there holds
 \begin{alignat}{2}\label{ineq:u energy-17}
    \sup_{0\leq t\leq T}\frac{1}{2}\int_{\Omega} |\nabla u|^2(x,t)dx + \int_{0}^T\int_{\Omega} |\Delta u|^2dx dt \leq \frac{2R^2}{3}\eps^{2N-3}.
  \end{alignat}
  According to the definition of $T_\varepsilon$ in \eqref{def:Teps}, then \eqref{ineq:u energy-2}, \eqref{ineq:u energy-4}, \eqref{ineq:u energy-5} and \eqref{ineq:u energy-17} implies  $T_\varepsilon=T_0$ and then \eqref{assumptions'} holds.

Finally,  \eqref{eq:convVelocityb} is a direct consequence of \eqref{velocity decompose} and \eqref{w1estimate}. The remaining two statements in Theorem \ref{thm:main} follow from the constructions of $c_A$ and  $\ve_A$. We then complete the proof of Theorem \ref{thm:main}.

\appendix

\section{Formally Matched Asymptotics}

In this appendix we discuss the construction of approximate solutions with the aid of the method of formally matched asymptotics. The scheme is similar to the scheme in \cite{AlikakosLimitCH} with adaptations similar to the schemes in \cite{ChenHilhorstLogak} and
\cite{StokesAllenCahn}. Moreover, it is an adaption of the scheme presented in \cite{AbelsMarquardt2} for the ``integer order part'' to the case of a Navier-Stokes/Allen-Cahn system instead of a Stokes/Cahn-Hilliard system, which can also be found in \cite{PhDMarquardt} in more detail.

First of all, we note that, since
\begin{align*}
\Div(\nabla c^{\eps}\otimes\nabla c^{\eps})=\frac{1}{2}\nabla\big(|\nabla c^{\eps}|^2\big)+\Delta c^{\eps}\nabla c^{\eps},
\end{align*}
we  can rewrite \eqref{eq:NSAC1}-\eqref{eq:NSAC3} as follows
\begin{alignat}{2}\label{eq:NSAC1-new}
  \partial_t \ve^\eps +\ve^\eps\cdot \nabla \ve^\eps-\Div(2\nu(c^\eps)D\ve^\eps)  +\nabla p^\eps & = -\eps\Delta c^{\eps}\nabla c^{\eps},\\\label{eq:NSAC2-new}
  \Div \ve^\eps& = 0,\\ \label{eq:NSAC3-new}
 \partial_t c^\eps +\ve^\eps\cdot \nabla c^\eps & =\Delta c^{\eps}-\e^{-2}f'(c^{\eps}),
\end{alignat}
in $\Omega\times (0,T_0)$ by changing $p^\eps$ by a scalar function.

\subsection{The Outer Expansion \label{subsec:The-Outer-Expansion}}

We assume that in $\Omega^{\pm}$
the solutions of  (\ref{eq:NSAC1-new})-(\ref{eq:NSAC3-new})
have the expansions
\begin{align*}
c^{\eps}(x,t)  \approx\sum_{k\geq 0}\eps^{k}c_{k}^{\pm}(x,t), \  \ \mathbf{v}^{\eps}(x,t) \approx\sum_{k\geq 0}\eps^{k}\mathbf{v}_{k}^{\pm}(x,t),\ \  p^{\eps}(x,t)  \approx\sum_{k\geq -1}\eps^{k}p_{k}^{\pm}(x,t),
\end{align*}
i.e., we have {up to} higher-order terms
\begin{align*}
c^{\eps}(x,t)  =\sum_{k=0}^{N+2}\eps^{k}c_{k}^{\pm}(x,t), \  \ \mathbf{v}^{\eps}(x,t) =\sum_{k=0}^{N+2}\eps^{k}\mathbf{v}_{k}^{\pm}(x,t),\ \  p^{\eps}(x,t) = \sum_{k=-1}^{N+2}\eps^{k}p_{k}^{\pm}(x,t),
\end{align*}
for some $N\in\N$, $c_{k}^{\pm}$, $\mathbf{v}_{k}^{\pm}$ and
$p_{k}^{\pm}$ are smooth functions defined in $\Omega^{\pm}$.
Plugging this ansatz into (\ref{eq:NSAC1-new}{)}-(\ref{eq:NSAC3-new}) yields in $\Omega^\pm$
\begin{align}
  &\sum_{k\geq0}\eps^{k}\partial_t\mathbf{v}_{k}^{\pm}+\sum_{k\geq0}\eps^{k}\sum_{0\leq j\leq k}\mathbf{v}_{j}^{\pm}\cdot\nabla \mathbf{v}_{k-j}^{\pm} -\Div\bigg(2\sum_{k\geq0}\eps^{k}\sum_{0\leq j\leq k}\nu_{j}^{\pm}D\ve_{k-j}^{\pm}\bigg) \nonumber\\&\qquad\qquad\qquad 
+\sum_{k\geq-1}\eps^{k}\nabla p_{k}^{\pm}=  -\sum_{k\geq0}\eps^{k+1}\sum_{0\leq j\leq k}\Delta c_{j}^{\pm}\nabla c_{k-j}^{\pm},\label{eq:Stokes-Outer1}\\
&\sum_{k\geq0}\eps^{k}\mbox{div}\mathbf{v}_{k}^{\pm}=0,\label{eq:Stokes-Outer2}\\
&\sum_{k\geq0}\eps^{k}\partial_{t}c_{k}^{\pm}+\sum_{k\geq0}\eps^{k}\sum_{0\leq j\leq k}\mathbf{v}_{j}^{\pm}\cdot\nabla c_{k-j}^{\pm} =-\frac{1}{\eps^2}f'(c_{0}^{\pm})-\frac{1}{\eps}f''\left(c_{0}^{\pm}\right)c_{1}^{\pm}\nonumber\\&\qquad\qquad\qquad\qquad\qquad\qquad+\sum_{k\geq0}\eps^k\bigg(\Delta c_{k}^{\pm}-f''\left(c_{0}^{\pm}\right)c_{k+2}^{\pm} -f_{k+1}(c_{0}^{\pm},\ldots,c_{k+1}^{\pm})\bigg).\label{eq:CH-Outer1}
\end{align}
Here we have used for $g=\nu, f$
\begin{align}
g(c^\eps)& =g(c_{0}^{\pm})+\eps g'\left(c_{0}^{\pm}\right)c_{1}^{\pm}+ \sum_{k=2}^{N+2}\eps^{k}\bigg(g'\left(c_{0}^{\pm}\right)c_{k}^{\pm} + {g}_{k-1}(c_{0}^{\pm},\ldots,c_{k-1}^{\pm})\bigg)
\nonumber\\&\quad+\eps^{N+3}g_{N+3}^\varepsilon(c_{0}^{\pm},\ldots,c_{N+2}^{\pm})=:\sum_{k\geq0}\eps^{k}g_{k}^{\pm},\label{eq:CH-Outer2nu}
\end{align}
where for fixed $c_{0}^{\pm}$ the functions $\nu_k,f_{k}$ are polynomials
in $(c_{1}^{\pm},\ldots,c_{k}^{\pm})$ and are the result
of a Taylor expansion. Moreover, they do not depend on $\eps$, except for the remainder term $g_{N+3}^\eps$.

 Matching the $\mathcal{O}(\eps^{-2}),\mathcal{O}(\eps^{-1})$ terms in \eqref{eq:CH-Outer1} yields
$f'(c_{0}^{\pm})=f''(c_{0}^{\pm})c_{1}^{\pm}=0$ and in view of the Dirichlet boundary
data for $c^{\eps}$ we set
\begin{equation}
c_{0}^{\pm}=\pm1.\label{eq:c0out}
\end{equation}
Then
\begin{equation}
c_{1}^{\pm}=0.\label{eq:c1out}
\end{equation}
Applying an induction argument to \eqref{eq:CH-Outer1} implies
\begin{equation}
c_{k}^{\pm}=0,\ \ k\geq 2.\label{eq:ckout}
\end{equation}
{Substituting} \eqref{eq:c0out}-\eqref{eq:ckout} into \eqref{eq:CH-Outer2nu} leads to
\begin{align}
\nu_{0}^{\pm} & =\nu(c_{0}^{\pm}),\quad \nu_{k}^{\pm}=0 \qquad \text{for } k\geq 1.\label{eq:CH-Outerknu}
\end{align}
Comparing the order terms $\mathcal{O}\left(\eps^{k}\right)$ (with $k\geq-1$) in \eqref{eq:Stokes-Outer1}-\eqref{eq:Stokes-Outer2}
yields in $\Omega^{\pm}$
\begin{alignat}{2}
\nabla p_{-1}^{\pm}  &=0, \label{eq:pressure-Outer-1}\\
 \partial_t\mathbf{v}_{k}^{\pm}+\mathbf{v}_{0}^{\pm}\cdot\nabla \mathbf{v}_{k}^{\pm}+\mathbf{v}_{k}^{\pm}\cdot\nabla \mathbf{v}_{0}^{\pm}  -\nu(c_{0}^{\pm})\Delta\ve_{k}^{\pm}+\nabla p_{k}^{\pm}  &=-\sum_{1\leq j\leq k-1}\mathbf{v}_{j}^{\pm}\cdot\nabla \mathbf{v}_{k-j}^{\pm} ,\ k\geq 0, \label{eq:stokes-Outerk}\\
\operatorname{div}\mathbf{v}_{k}^{\pm} & =0, \ k\geq 0.\label{eq:Outervpdefine}
\end{alignat}
For simplicity we take
\begin{align}
 p_{-1}^{\pm} =0.\label{expression:pressure-Outer-1}
\end{align}

\begin{rem}
\label{Outer-Rem}
We will need $(c_{k}^{\pm},\mathbf{v}_{k}^{\pm},p_{k}^{\pm})$,
for $k\geq0$, to not only be defined in $\Omega^{\pm}$,
but we have to extend them onto $\Omega^{\pm}\cup\Gamma(2\delta)$.
For $p_{k}^{\pm}$ we may use any smooth extension. One possibility is to use the extension operator defined in \cite[Part VI, Theorem 5]{Stein:SingInt}. It is trivial to extend $c_{k}^{\pm}$. For $\mathbf{v}_{k}^{\pm}$
we employ the same extension operator and then use the Bogovskii operator
to ensure that the extension is divergence free in $\Gamma_{t}(2\delta)$.
In particular we may construct a divergence free extension $\text{\ensuremath{\mathcal{E}}}^{\pm}(\mathbf{v}_{k}^{\pm})$
such that $\text{\ensuremath{\mathcal{E}}}^{\pm}(\mathbf{v}_{k}^{\pm})|_{\Omega^{\pm}(t)}=\mathbf{v}_{k}^{\pm}$
in $\Omega^{\pm}(t)$ and
\begin{equation}
\left\Vert \text{\ensuremath{\mathcal{E}}}^{\pm}(\mathbf{v}_{k}^{\pm})\right\Vert _{H^{2}(\Omega^{\pm}(t)\cup\Gamma_{t}(2\delta))}\leq C\Vert \mathbf{v}_{k}^{\pm}\Vert _{H^{2}(\Omega^{\pm}(t))}.\label{eq:sp=0000E4terwichtig}
\end{equation}
\end{rem}

For later use we define
\begin{align}
  \mathbf{W}_{k}^{\pm}(x,t) & =\partial_t\mathbf{v}_{k}^{\pm}(x,t)+{\sum_{j=0}^k \mathbf{v}_{k-j}^{\pm}\cdot\nabla \mathbf{v}_{j}^{\pm}} -\nu(c_{0}^{\pm})\Delta\mathbf{v}_{k}^{\pm}(x,t)+\nabla p_{k}^{\pm}(x,t) 
  \label{eq:Wkpm}
\\ \mathbf{W}^{\pm}  &=\sum_{k\geq0}\eps^{k}\mathbf{W}_{k}^{\pm},\label{eq:Wpm}
\end{align}
for $(x,t)\in\Omega^{\pm}\cup\Gamma(2\delta)$.
Note that by \eqref{eq:c0out}-\eqref{eq:CH-Outerknu} and \eqref{eq:stokes-Outerk}
we have $\mathbf{W}_{k}^{\pm}(x,t)=0$
for all $(x,t)\in\overline{\Omega^{\pm}}.$

\subsection{The Inner Expansion \label{subsec:The-Inner-Expansion}}

Close to the interface $\Gamma$ we introduce a stretched variable
\begin{equation}
\rho^{\eps}(x,t):=\frac{d_{\Gamma}(x,t)-\eps h^{\eps}(S(x,t),t)}{\eps}\qquad \text{for all }(x,t)\in\Gamma(2\delta)\label{eq:rhoeps}
\end{equation}
for $\eps\in\left(0,1\right)$. Here $h^{\eps}\colon \mathbb{T}^{1}\times [0,T_{0}]\rightarrow\mathbb{R}$
is a given smooth function and can heuristically be interpreted as
the distance of the zero level set of $c^{\eps}$ to $\Gamma$,
see also \cite[Chapter 4.2]{ChenHilhorstLogak}. In the following, we will often
drop the $\eps$\textendash dependence and write $\rho(x,t)=\rho^{\eps}(x,t)$.

Now assume that, in $\Gamma(2\delta)$,  the identities
\begin{align*}
u^{\eps}(x,t) & =\tilde{u}^{\eps}\big(\tfrac{d_{\Gamma}(x,t)}{\eps}-h^{\eps}(S(x,t),t),x,t\big)\qquad \text{for } u=c,\ve,p
\end{align*}
hold for the solutions of (\ref{eq:NSAC1-new})-\eqref{eq:NSAC3-new}
and some smooth functions $\tilde{c}^{\eps},\tilde{p}^{\eps}\colon \mathbb{R}\times\Gamma(2\delta)\rightarrow\mathbb{R}$,
$\tilde{\mathbf{v}}^{\eps}\colon \mathbb{R}\times\Gamma(2\delta)\rightarrow\mathbb{R}^{2}$.
Furthermore, we assume that we have the expansions
\begin{align*}
\tilde{c}^{\eps}(\rho,x,t)  \approx\sum_{k\geq0}\eps^{k}c_{k}(\rho,x,t),\quad\tilde{\mathbf{v}}^{\eps}(\rho,x,t) \approx\sum_{k\geq0}\eps^{k}\mathbf{v}_{k}(\rho,x,t),\quad
\tilde{p}^{\eps}(\rho,x,t) \approx\sum_{k\geq-1}\eps^{k}p_{k}(\rho,x,t),
\end{align*}
for all $(\rho,x,t)\in\mathbb{R}\times\Gamma(2\delta)$
and also
\begin{equation}
h^{\eps}(s,t)\approx\sum_{k\geq0}\eps^{k}h_{k+1}(s,t)\qquad \text{for all }s\in \T^1,t\in[0,T_0],\label{eq:Ansatz2}
\end{equation}
where $c_{k},p_{k}\colon \mathbb{R}\times\Gamma(2\delta)\rightarrow\mathbb{R}$,
$\mathbf{v}_{k}\colon \mathbb{R}\times\Gamma(2\delta)\rightarrow\mathbb{R}^{2}$
and $h_{k+1}\colon \mathbb{T}^{1}\times [0,T_{0}]\rightarrow\mathbb{R}$
are smooth functions for all $k\geq0$. When referring to $\tilde{c},\tilde{p},\tilde{\mathbf{v}}$
and the expansion terms we write $\nabla=\nabla_{x}$ and $\Delta=\Delta_{x}$ for the gradient, Laplacian, resp., with respect to $x$.
The expressions $\partial_{t}^{\Gamma}h^{\eps}(x,t)$,
$\nabla^{\Gamma}h^{\eps}(x,t)$, $\Delta^{\Gamma}h^{\eps}(x,t)$
and $D_{\Gamma}^{2}h^{\eps}(x,t)$ are for $(x,t)\in\Gamma(2\delta)$
to be understood in the sense of \eqref{Prelim:1.11}.

Moreover, for $g=\nu,f'$
\begin{align}
g(c^\eps)& =g(c_{0})+\eps g'(c_{0})c_{1}+ \sum_{k=2}^{N+2}\eps^{k}\bigg(g'(c_{0})c_{k} +g_{k-1}(c_{0},\ldots,c_{k-1})\bigg)
\nonumber\\&\quad+\eps^{N+3}g_{N+3}^\varepsilon(c_{0},\ldots,c_{N+2})=:\sum_{k\geq0}\eps^{k}g_{k},\label{inner:nu}
\end{align}
In order to match the inner and outer expansions, we require that
for all $k$ the so-called \emph{inner-outer matching conditions}
\begin{align}
\sup_{(x,t)\in\Gamma(2\delta)}\left|\partial_{x}^{m}\partial_{t}^{n}\partial_{\rho}^{l}\left(\varphi (\pm\rho,x,t)-\varphi^{\pm}(x,t)\right)\right| & \leq Ce^{-\alpha\rho},\label{eq:matchcon}
\end{align}
where $\varphi=c_{k},\mathbf{v}_{k},p_{k}$ and $k\geq0$
hold for constants $\alpha,C>0$ and all $\rho>0$, $m,n,l\geq0$.

If $\ve^\eps(x,t)=\tilde{\ve}^\eps(\rho,x,t)$, then one has by direct computations
\begin{align}  
&\Div(2\nu(c^\eps)D\ve^\eps)= \eps^{-2}\partial_{\rho}\big(\nu(\tilde{c}^\eps)\partial_{\rho}\tilde{\ve}^\eps\big)
+\eps^{-1}\partial_{\rho}\big(2\nu(\tilde{c}^\eps)D\tilde{\ve}^\eps\big)\big)\cdot\nabla d_{\Gamma}+\eps^{-1}\Div\big(2\nu(\tilde{c}^\eps)D_d\tilde{\ve}^\eps\big)\big)
\nonumber\\
&\qquad-\eps^{-1}\partial_{\rho}\big(\nu(\tilde{c}^\eps)\Div\tilde{\ve}^\eps\big)\nabla d_{\Gamma}
+\partial_{\rho}\big(\nu(\tilde{c}^\eps)\partial_{\rho}\tilde{\ve}^\eps\big)|\nabla^{\Gamma}h^{\eps}|^2-\partial_{\rho}\big(2\nu(\tilde{c}^\eps)D\tilde{\ve}^\eps\big)\big)\cdot\nabla ^{\Gamma}h^\eps
\nonumber\\
&\qquad+\Div\big(2\nu(\tilde{c}^\eps)(D\tilde{\ve}^\eps-D_h\tilde{\ve}^\eps)\big)+\partial_{\rho}\big(\nu(\tilde{c}^\eps)\Div\tilde{\ve}^\eps\big)\nabla^{\Gamma}h^{\eps},\label{divexpress}
\end{align}
where
\begin{align*}
 D\tilde{\ve}^\eps&=\frac 12(\nabla\tilde{\ve}^\eps+(\nabla\tilde{\ve}^\eps)^T),\ D_h\tilde{\ve}^\eps=\frac 12\big(\partial_\rho\tilde{\ve}^\eps\nabla^{\Gamma}h^{\eps}+(\partial_\rho\tilde{\ve}^\eps\nabla^{\Gamma}h^{\eps})^T\big),\\ D_d\tilde{\ve}^\eps&=\frac 12\big(\partial_\rho\tilde{\ve}^\eps\nabla d_{\Gamma}+(\partial_\rho\tilde{\ve}^\eps\nabla d_{\Gamma})^T\big).
 \end{align*}
{For the following we use}
\begin{align*}
  \eps\Delta c^{\eps}\nabla c^{\eps}=&\eps^{-2}\partial_{\rho}\cte\partial_{\rho\rho}\cte\nabla d_{\Gamma}-
                                                                    \eps^{-1} \partial_{\rho}\cte\partial_{\rho\rho}\cte \nabla^\Gamma h^\eps- (\partial_\rho \cte)^2 \left(\Delta d_\Gamma \nabla^\Gamma h^\eps + \Delta^\Gamma h^\eps\nabla d_\Gamma\right)\\
  &-2 \partial_\rho \nabla \cte \cdot \nabla^\Gamma h^\eps\partial_{\rho}\cte\nabla d_\Gamma+ 
                                                                    \eps^{-1}\mathbb{A}^{\eps}+\mathbb{B}^{\eps}+\eps\mathbb{C}^{\eps},
\end{align*}
where $\mathbb{A}^{\eps}, \mathbb{B}^{\eps}$ are polynomials in $\partial_\rho\cte, \nabla \cte, \nabla d_{\Gamma}$ and their derivaties and $\mathbb{C}^{\eps}$ is a polynomial in $\partial_\rho\cte, \nabla \cte, \nabla d_{\Gamma}, \nabla^{\Gamma} h^\eps$. Inserting the expansions for $\cte$ and $h^\eps$ one obtains
\begin{equation*}
  \mathbb{A}^\eps=\sum_{k\geq0}\eps^{k}\mathbb{A}_{k},\quad
  \mathbb{B}^\eps=\sum_{k\geq0}\eps^{k}\mathbb{B}_{k},\quad
    \mathbb{C}^\eps=\sum_{k\geq0}\eps^{k}\mathbb{C}_{k}.
\end{equation*}
In the new coordinate $(\rho,x,t)$ the system  (\ref{eq:NSAC1-new})-(\ref{eq:NSAC3-new})   reduces to
\begin{align}
\partial_{\rho}\big(\nu(c^\eps)\partial_{\rho}\tilde{\ve}^\eps\big)=&\partial_{\rho}\cte\partial_{\rho\rho}\cte\nabla d_{\Gamma}+ \eps\bigg(\partial_{\rho}\tilde{\ve}^\eps\partial_{t}d_{\Gamma}+\tilde{\ve}^\eps\cdot\nabla d_{\Gamma}\partial_{\rho}\tilde{\ve}^\eps-\partial_{\rho}\big(2\nu(\tilde{c}^\eps)D\tilde{\ve}^\eps\big)\cdot\nabla d_{\Gamma}
\nonumber\\
&\quad-\Div\big(2\nu(\tilde{c}^\eps)D_d\tilde{\ve}^\eps\big)+\partial_{\rho}\big(\nu(\tilde{c}^\eps)\Div\tilde{\ve}^\eps\big)\nabla d_{\Gamma}
+\partial_{\rho}\tilde{p}^\eps\nabla d_{\Gamma}-\partial_{\rho}\cte\partial_{\rho\rho}\cte \nabla^\Gamma h^\eps+\mathbb{A}^{\eps}\bigg)
\nonumber\\&+\eps^2\bigg(\partial_{t}\tilde{\ve}^\eps+\tilde{\ve}^\eps\cdot\nabla \tilde{\ve}^\eps -\tilde{\ve}^\eps\cdot\nabla ^{\Gamma}h^{\eps}\partial_{\rho}\tilde{\ve}^\eps -\partial_{\rho}\tilde{\ve}^\eps\partial_{t}^{\Gamma}h^{\eps}+\nabla\tilde{p}^\eps-\partial_{\rho}\tilde{p}^\eps\nabla _{\Gamma}h^\eps
\nonumber\\
&\qquad-\partial_{\rho}\big(\nu(\tilde{c}^\eps)\partial_{\rho}\tilde{\ve}^\eps\big)|\nabla^{\Gamma}h^{\eps}|^2+\partial_{\rho}\big(2\nu(\tilde{c}^\eps)D\tilde{\ve}^\eps\big)\cdot\nabla ^{\Gamma}h^\eps
+\partial_{\rho}\big(\nu(\tilde{c}^\eps)\Div\tilde{\ve}^\eps\big)\nabla^{\Gamma}h^{\eps}\nonumber\\
                                                                    &\qquad-\Div\big(2\nu(\tilde{c}^\eps)(D\tilde{\ve}^\eps-D_h\tilde{\ve}^\eps)\big)- (\partial_\rho \cte)^2 \left(\Delta d_\Gamma \nabla^\Gamma h^\eps + \Delta^\Gamma h^\eps\nabla d_\Gamma\right)\nonumber\\
  &\qquad -2 \partial_\rho \nabla \cte \cdot \nabla^\Gamma h^\eps\partial_\rho \cte \nabla d_\Gamma+\mathbb{B}^{\eps}\bigg)
             +\eps^3\mathbb{C}^{\eps},\label{eq:innerstokes} 
\end{align}
together with
  \begin{align}
\partial_{\rho}\tilde{\ve}^\eps\cdot\nn=&\eps\partial_{\rho}\tilde{\ve}^\eps\cdot\nabla _{\Gamma}h^\eps-\eps\Div_x\tilde{\ve}^\eps,\label{eq:innerdiv}\\
\partial^2_{\rho}\tilde{c}^\eps-f'(\tilde{c}^\eps)=&\eps\bigg(\partial_{\rho}\tilde{c}^\eps\partial_t d_{\Gamma}+\partial_{\rho}\tilde{c}^\eps\tilde{\ve}^\eps\cdot\nabla d_{\Gamma}-2\nabla\partial_{\rho}\tilde{c}^\eps\cdot\nabla d_{\Gamma}-\partial_{\rho}\tilde{c}^\eps\Delta d_{\Gamma}\bigg)
\nonumber\\&+\eps^2\bigg(2\nabla\partial_{\rho}\tilde{c}^\eps\cdot\nabla^{\Gamma}h^{\eps}+\partial_{\rho}\tilde{c}^\eps\Delta^{\Gamma}h^{\eps}-\partial^2_{\rho}\tilde{c}^\eps|\nabla^\Gamma h^\eps|^2-\Delta\tilde{c}^\eps
\nonumber\\&\qquad+\partial_t\tilde{c}^\eps-\partial_{\rho}\tilde{c}^\eps\partial_t^\Gamma h^\eps+\tilde{\ve}^\eps\cdot\big(\nabla\tilde{c}^\eps-\partial_{\rho}\tilde{c}^\eps\nabla^\Gamma h^\eps\big)\bigg).\label{eq:innerAL}
\end{align}

We interpret $\{ (x,t)\in\Gamma(2\delta)|d_{\Gamma}(x,t)=\eps h^{\eps}(S(x,t),t)\}$
as an approximation of the $0$-level set of $c^{\eps}$. Thus, we normalize {$c_{k}$}
such that
\begin{align}
  c_{k}(0,x,t)=0\qquad \text{for all }(x,t)\in\Gamma(2\delta), k\geq 0.\label{innnormalize}
\end{align}
 In a similar manner as in \cite{AlikakosLimitCH}, we introduce
auxiliary functions $g^{\eps}(x,t)$ , $\mathbf{u}^{\eps}(x,t)$,
and $\mathbf{l}^{\eps}(x,t)$ for $(x,t)\in\Gamma(2\delta)$.
Here $g^{\eps}$ is added to \eqref{eq:innerAL}
and $\mathbf{l}^{\eps}$ is added to \eqref{eq:innerstokes} in order to enable to fulfill the compatibility
conditions in $\Gamma(2\delta)\backslash\Gamma$. Furthermore, adding $\mathbf{u}^{\eps}$ to \eqref{eq:innerstokes} is important to ensure the matching conditions in $\Gamma(2\delta)\backslash\Gamma$.
Moreover, we choose $\eta\colon \mathbb{R}\rightarrow[0,1]$ such
that $\eta=0$ in $(-\infty,-1]$, $\eta=1$ in $[1,\infty)$
and $\eta'\geq0$ in $\mathbb{R}$ and
define
\[
\eta^{C,\pm}(\rho)=\eta(-C\pm\rho)
\]
for a suitable constant $C>0$ to be chosen later and $\rho\in\mathbb{R}$.

Now we may rewrite \eqref{eq:innerstokes}-\eqref{eq:innerAL} as
\begin{align}
\partial_{\rho}\big(\nu(\tilde{c}^\eps)\partial_{\rho}\tilde{\ve}^\eps\big)=&\, \partial_{\rho}\cte\partial_{\rho\rho}\cte\nabla d_{\Gamma}+ \eps\bigg(\partial_{\rho}\tilde{\ve}^\eps\partial_{t}d_{\Gamma}+\tilde{\ve}^\eps\cdot\nabla d_{\Gamma}\partial_{\rho}\tilde{\ve}^\eps-\partial_{\rho}\big(2\nu(\tilde{c}^\eps)D\tilde{\ve}^\eps\big)\big)\cdot\nabla d_{\Gamma}
\nonumber\\
&\quad-\Div\big(2\nu(\tilde{c}^\eps)D_d\tilde{\ve}^\eps\big)\big)+\partial_{\rho}\big(\nu(\tilde{c}^\eps)\Div\tilde{\ve}^\eps\big)\nabla d_{\Gamma}
+\partial_{\rho}\tilde{p}^\eps\nabla d_{\Gamma}-\partial_{\rho\rho}\cte \partial_{\rho}\cte \nabla^\Gamma h^\eps+\mathbb{A}^{\eps}\bigg)
\nonumber\\&+\eps^2\bigg(\partial_{t}\tilde{\ve}^\eps+\tilde{\ve}^\eps\cdot\nabla \tilde{\ve}^\eps-\tilde{\ve}^\eps\cdot\nabla ^{\Gamma}h^{\eps}\partial_{\rho}\tilde{\ve}^\eps-\partial_{\rho}\tilde{\ve}^\eps\partial_{t}^{\Gamma}h^{\eps}+\nabla\tilde{p}^\eps-\partial_{\rho}\tilde{p}^\eps\nabla _{\Gamma}h^\eps
\nonumber\\
&\qquad-\partial_{\rho}\big(\nu(\tilde{c}^\eps)\partial_{\rho}\tilde{\ve}^\eps\big)|\nabla^{\Gamma}h^{\eps}|^2+\partial_{\rho}\big(2\nu(\tilde{c}^\eps)D\tilde{\ve}^\eps\big)\big)\cdot\nabla ^{\Gamma}h^\eps
\nonumber\\
&\qquad-\Div\big(2\nu(\tilde{c}^\eps)(D\tilde{\ve}^\eps-D_h\tilde{\ve}^\eps)\big)+\partial_{\rho}\big(\nu(\tilde{c}^\eps)\Div\tilde{\ve}^\eps\big)\nabla^{\Gamma}h^{\eps}\nonumber\\
  &\qquad - (\partial_\rho \cte)^2 \left(\Delta d_\Gamma \nabla^\Gamma h^\eps + \Delta^\Gamma h^\eps\nabla d_\Gamma\right)-2 \partial_\rho \cte \partial_\rho \nabla \cte \cdot \nabla^\Gamma h^\eps\nabla d_\Gamma+\mathbb{B}^{\eps}\bigg)
\nonumber\\
                                                                                    &+\eps^3\mathbb{C}^{\eps}+\big(\nu(\theta_0)\eta'(\rho)\big)'\mathbf{u}^{\eps}\big(d_{\Gamma}-\eps(\rho+h^{\eps})\big) +\eps\mathbf{l}^{\eps}\eta'(\rho)\big(d_{\Gamma}-\eps (\rho+h^{\eps})\big)\nonumber \\& +\eps^{2}(\mathbf{W}^{+}\eta^{C_{S},+}+\mathbf{W}^{-}\eta^{C_{S},-}),\label{modieq:innerstokes}
\end{align}
and
\begin{align}
\partial_{\rho}\tilde{\ve}^\eps\cdot\nabla d_{\Gamma}=&\eps\partial_{\rho}\tilde{\ve}^\eps\cdot\nabla^ {\Gamma}h^\eps-\eps\Div\tilde{\ve}^\eps+\big(\mathbf{u^{\eps}}\cdot (\nabla d_{\Gamma}-\eps\nabla^{\Gamma}h^{\eps})\big)\eta'(\rho)\big(d_{\Gamma}-\eps (\rho+h^{\eps})\big),\label{modieq:innerdiv}\\
\partial^2_{\rho}\tilde{c}^\eps-f'(\tilde{c}^\eps)=&\eps\bigg(\partial_{\rho}\tilde{c}^\eps\partial_t d_{\Gamma}+\partial_{\rho}\tilde{c}^\eps\tilde{\ve}^\eps\cdot\nabla d_{\Gamma}-2\nabla\partial_{\rho}\tilde{c}^\eps\cdot\nabla d_{\Gamma}-\partial_{\rho}\tilde{c}^\eps\Delta d_{\Gamma}\bigg)
\nonumber\\&+\eps^2\bigg(2\nabla\partial_{\rho}\tilde{c}^\eps\cdot\nabla^{\Gamma}h^{\eps}+\partial_{\rho}\tilde{c}^\eps\Delta^{\Gamma}h^{\eps}-\partial^2_{\rho}\tilde{c}^\eps|\nabla^\Gamma h^\eps|^2-\Delta\tilde{c}^\eps
\nonumber\\&\qquad+\partial_t\tilde{c}^\eps-\partial_{\rho}\tilde{c}^\eps\partial_t^\Gamma h^\eps+\tilde{\ve}^\eps\cdot\big(\nabla\tilde{c}^\eps-\partial_{\rho}\tilde{c}^\eps\nabla^\Gamma h^\eps\big)\bigg)
\nonumber\\&+\eps g^{\eps}\myeta'(\rho)\left(d_{\Gamma}-\eps\left(\rho+h^{\eps}\right)\right).\label{modieq:innerAL}
\end{align}
Here the equations only have to hold in
\[
S^{\eps}:=\big\{ (\rho,x,t)\in\mathbb{R}\times\Gamma(2\delta):\rho=\tfrac{d_{\Gamma}(x,t)}{\eps}-h^{\eps}(S(x,t),t)\big\}.
\]
But in the following we consider them as ordinary differential equations in $\rho\in\mathbb{R}$,
where $(x,t)\in \Gamma(2\delta)$ are seen as fixed parameters. Thus we assume from now on that   \eqref{modieq:innerstokes}-\eqref{modieq:innerAL}
are fulfilled in $\mathbb{R}\times\Gamma(2\delta)$. The
term $\mathbf{W}^{\pm}$ (cf. \eqref{eq:Wpm})
are used here in order to ensure the exponential decay of the right
hand sides; in this context $C_{S}>0$ is a constant which will be
determined later on (see Remark \ref{WU}). We assume that the auxiliary
functions have expansions of the form
\begin{align}
\mathbf{u}^{\eps}(x,t) \approx\sum_{k\geq0}\mathbf{u}_{k}(x,t)\eps^{k},\quad  g^{\eps}(x,t)  \approx\sum_{k\geq0}g_{k}(x,t)\eps^{k}, \quad
\mathbf{l}^{\eps}(x,t) \approx\sum_{k\geq0}\mathbf{l}_{k}(x,t)\eps^{k},\label{eq:zusatzerme}
\end{align}
for $(x,t)\in\Gamma(2\delta)$.

The following lemma comes  from    Lemma 2.6.2 in \cite{PromotionStefan}.
\begin{lem}\label{ODEsolver}
Let $U\subset \mathbb{R}^n$ be an open subset and let $A:\mathbb{R}\times U\rightarrow \mathbb{R},(\rho,x)\mapsto A(\rho,x)$ be given and smooth. Assume that there exist $A^\pm(x)$ such that the decay property $A(\pm\rho,x)-A^\pm(x)=O(e^{-\alpha\rho})$ as $\rho\rightarrow\infty$ is fulfilled. Then for all $x\in U$
the system
\begin{equation}\label{eq:ODEsolver}
\begin{split}
&w_{\rho\rho}(\rho,x)-f'(\theta_0(\rho))w(\rho,x)=A(\rho,x),\ \rho\in\mathbb{R},\\
&w(0,x)=0,\ w(\cdot,x)\in L^\infty(\mathbb{R})
\end{split}
\end{equation}
has a solution if and only if
\begin{equation}\label{compatiblity:ODEsolver}
  \int_{\mathbb{R}}A(\rho,x)\theta_0'(\rho)d\rho=0.
\end{equation}
In addition, if the solution exists, then it is unique and satisfies for all $x\in U$
\begin{equation}\label{decay1:ODEsolver}
D_\rho^\ell\bigg(w(\pm\rho,x)+\frac{A^\pm(x)}{f'(\pm1)}\bigg)=O(e^{-\alpha\rho})~\text{as}~\rho\to \infty,\ l=0,1,2.
\end{equation}
Furthermore, if $A(\rho,x)$ satisfies for all $x\in U$
\begin{align*}
D_x^mD_\rho^\ell\bigg(A(\pm\rho,x)-A^\pm(x)\bigg)=O(e^{-\alpha\rho}),~\text{as}~ \rho\rightarrow\infty
\end{align*}
for all $m\in\{0,\cdots,M\}$ and $l\in\{0,\cdots,L\}$,
then
\begin{equation}\label{decay2:ODEsolver}
D_x^mD_\rho^\ell\bigg(w(\pm\rho,x)+\frac{A^\pm(x)}{f'(\pm1)}\bigg)=O(e^{-\alpha\rho})~\text{as}~\rho\to\infty
\end{equation}
for all $m\in\{0,\cdots,M\}$ and $l\in\{0,\cdots,L\}$.
\end{lem}

To proceed we show a  modified version of  Lemma 2.6.3 in \cite{PromotionStefan}.
\begin{lem}\label{modified version of Lemma 2.4}
Let $U\subset \mathbb{R}^n$ be an open subset and let $B\colon\mathbb{R}\times U\rightarrow \mathbb{R},(\rho,x)\mapsto B(\rho,x)$ be given and smooth. Assume that for all $x\in U$ the decay property $B(\pm\rho,x)=O(e^{-\alpha\rho})$ as $\rho\rightarrow\infty$ is fulfilled.
Then for each $x\in U$ the problem
\begin{align}
\partial_{\rho}\big(\nu(\theta_0)\partial_\rho w(\rho,x)\big)=B(\rho,x),\qquad \rho\in \mathbb{R},\label{modified version of Lemma 2.4-1}
\end{align}
has a solution $w(\cdot,x)\in C^2(\mathbb{R})\cap L^\infty(\mathbb{R})$ if and only if
\begin{align}
\int_{\mathbb{R}}B(\rho,x)d\rho=0.\label{modified version of Lemma 2.4-2}
\end{align}
Furthermore, if $w_{\ast}(\rho,x)$ is such a solution, then all the solutions can be written as
\begin{align}
 w(\rho,x)=w_{\ast}(\rho,x)+c(x),\quad \rho\in \mathbb{R}, x\in U,\label{modified version of Lemma 2.4-3}
\end{align}
where $c\colon U\rightarrow \mathbb{R}$ is an arbitrary function.

In particular, if \eqref{modified version of Lemma 2.4-2} holds,
\begin{align}
w_{\ast}(\rho,x)=\int_{0}^\rho\frac{1}{\nu(\theta_0)}\int_{-\infty}^r B(s,x)ds dr,\quad \rho\in \mathbb{R}, x\in U,\label{modified version of Lemma 2.4-4}
\end{align}
is a solution.
Additionally, if $\int_{\mathbb{R}}B(\rho,x)d\rho=0$ for all $x\in U$ and there exist $M,L\in \mathbb{N}$ such that
\begin{align}
D_x^mD_\rho^lB(\pm\rho,x)=O(e^{-\alpha\rho}) \ \text{as}\ \rho\rightarrow+\infty\label{modified version of Lemma 2.4-5}
\end{align}
for all $m\in\{0,\cdots,M\}$ and $l\in\{0,\cdots,L\}$ then there exist functions $w^+(x)$ and  $w^-(x)$ such that
\begin{align}
D_x^mD_\rho^l\big(w(\pm\rho,x)-w^{\pm}(x)\big)=O(e^{-\alpha\rho}) \ \text{as} \ \rho\rightarrow+\infty\label{modified version of Lemma 2.4-6}
\end{align}
for all $m\in\{0,\cdots,M\}$ and $l\in\{0,\cdots,L+2\}$.
\end{lem}
\begin{proof} Assume the problem
has a solution $w(\cdot,x)\in C^2(\mathbb{R})\cap L^\infty(\mathbb{R})$.  Along the same line of  the proof in Lemma 2.6.3  in \cite{PromotionStefan}  we can prove
\begin{equation*}
\lim_{\rho\rightarrow\pm\infty}\partial_\rho w(\rho,x)=0.
\end{equation*}
Integrating \eqref{modified version of Lemma 2.4-1} leads to \eqref{modified version of Lemma 2.4-2} which implies that the solution $w(\cdot,x)$ is independent of $\rho$ in the case of $B=0$. Then we have  \eqref{modified version of Lemma 2.4-3}.
 Moreover if \eqref{modified version of Lemma 2.4-2} holds, it is easy to check $w_{\ast}(\rho,x)$ defined in \eqref{modified version of Lemma 2.4-4} is a solution and satisfies \eqref{modified version of Lemma 2.4-6} if \eqref{modified version of Lemma 2.4-5} holds.
\end{proof}

\noindent
{\it{Matching $\eps^k$-order terms}:}  Firstly, by matching zero-order  term in \eqref{modieq:innerAL}, the matching conditions  \eqref{eq:matchcon} and \eqref{innnormalize} we get $c_{0}=\theta_0$ satisfying  \eqref{eq:OptProfile1}-\eqref{eq:OptProfile2}.
Matching the
$\eps^k(k=0,1)$-order terms, we derive the following ordinary differential equations
in $\rho$:
\begin{align}
  \partial_{\rho}\bigg(\nu(\theta_0(\rho))\big(\partial_{\rho}\ve_0-\mathbf{u}_0d_{\Gamma}\eta'(\rho)\big)\bigg)&=\big(\theta_0'(\rho)\theta_0''(\rho)+\partial_{\rho}p_{-1}\big)\nabla d_{\Gamma},\label{modieq:innerstokes0}
\end{align}
and
\begin{align}
&\partial_{\rho}\bigg(\nu(\theta_0(\rho))\big(\partial_{\rho}\ve_1+(\mathbf{u}_0h_{1}-\mathbf{u}_1d_{\Gamma})\eta'(\rho)\big)\bigg)= \big(\theta_0'(\rho)\partial_\rho^2c_1+\theta_0''(\rho)\partial_\rho c_1+\partial_{\rho}p_{0}\big)\nabla d_{\Gamma} + \partial_{\rho}\ve_{0}\partial_{t}d_{\Gamma}\nonumber\\&\quad+\ve_0\cdot\nabla d_{\Gamma}\partial_{\rho}\ve_{0}-\partial_{\rho}\big(2\nu(\theta_0(\rho))D\ve_{0}\big)\cdot\nabla d_{\Gamma}
-\Div\big(2\nu(\theta_0(\rho))D_d\ve_0\big)+\partial_{\rho}\big(\nu(\theta_0(\rho))\Div\ve_0\big)\nabla d_{\Gamma}
              \nonumber\\&\quad+\mathbf{l}_{0}d_{\Gamma}\eta'(\rho)-\partial_{\rho}\big(\nu_1\partial_{\rho}\ve_0\big)-\big(\nu(\theta_0)\eta'(\rho)\big)'\mathbf{u}_0\rho+\left(\partial_\rho p_{-1}- \partial_\rho \Big(\tfrac{(\theta_0'(\rho))^2}2\Big)\right)\nabla^\Gamma h_1+\mathbb{A}_0,\label{modieq:innerstokes1}
\end{align}
together with
  \begin{align}
\big(\partial_{\rho}\ve_0-\mathbf{u}_0d_{\Gamma}\eta'(\rho)\big)\cdot\nabla d_{\Gamma}&=0,\label{modieq:innerdiv0}\\
\partial^2_{\rho}c_1-f''(\theta_0)c_1&=\theta_0'(\rho)(\partial_t d_{\Gamma}+\ve_0\cdot\nabla d_{\Gamma}-\Delta d_{\Gamma})
+\myeta'(\rho)g_{0}d_{\Gamma}.\label{modieq:innerAL1}
\end{align}
Matching the
$\eps^k$-order terms for $k\geq2$, we obtain the following ordinary differential equations
in $\rho$:
\begin{align}
\partial_{\rho}\bigg(\nu(\theta_0(\rho))\big(\partial_{\rho}\ve_k+(\mathbf{u}_0h_{k}-\mathbf{u}_kd_{\Gamma})\eta'(\rho)\big)\bigg)&=\left(\partial_\rho p_{-1} -\partial_\rho \Big(\tfrac{(\theta_0'(\rho))^2}2\Big)\right)\nabla^\Gamma h_k+\mathbf{V}^{k-1},\label{modieq:innerstokesk}\\
\big(\partial_{\rho}\ve_k+(\mathbf{u}_0h_{k}-\mathbf{u}_kd_{\Gamma})\eta'(\rho)\big)\cdot\nabla d_{\Gamma}&=A^{k-1}+\nabla^\Gamma h_{k}\cdot(\partial_\rho\ve_0-\mathbf{u}_0d_\Gamma\eta'),\label{modieq:innerdivk}\\
\partial^2_{\rho}c_k-f''(\theta_0)c_k&=B^{k-1},\label{modieq:innerALk}
\end{align}
where
\begin{align*}
  \mathbf{V}^{k-1}&= -\partial_{\rho}\big(\nu_1\partial_{\rho}\ve_{k-1}\big)+\partial_{\rho}p_{k-1}\nabla d_{\Gamma}+\partial_{\rho}p_{0}\nabla^\Gamma h_{k-1}+ \partial_{\rho}\ve_{k-1}\partial_{t}d_{\Gamma}+\ve_0\cdot\nabla d_{\Gamma}\partial_{\rho}\ve_{k-1}+\ve_{k-1}\cdot\nabla d_{\Gamma}\partial_{\rho}\ve_{0}\\
                  &\quad-\partial_{\rho}\big(2\nu(\theta_0(\rho))D\ve_{k-1}\big)\cdot\nabla d_{\Gamma}+\partial_{\rho}\big(\nu(\theta_0(\rho))\Div\ve_{k-1}\big)\nabla d_{\Gamma}-\Div\big(2\nu(\theta_0(\rho))D_d\ve_{k-1}\big)\\
  &\quad  +\mathbf{l}_{k-1}d_{\Gamma}\eta'(\rho)-(\theta_0'(\rho))^2(\Delta d_\Gamma\nabla^\Gamma h_{k-1} + \Delta^\Gamma h_{k-1}\nabla d_\Gamma) +\widetilde{\mathbf{V}}_1^{k-1}+\widetilde{\mathbf{V}}_2^{k-1},
  \\
  \widetilde{\mathbf{V}}_1^{k-1}&= -\partial_{\rho}\ve_0\ve_0\cdot\nabla ^{\Gamma}h_{k-1}-\partial_{\rho}\ve_0\partial_{t}^{\Gamma}h_{k-1}-\partial_{\rho}p_0\nabla ^{\Gamma}h_{k-1} \\
  &\quad-\partial_{\rho}\big(\nu_0\partial_{\rho}\ve_{0}\big)\big(\nabla^{\Gamma}h_{1}\cdot\nabla^{\Gamma}h_{k-1}+\nabla^{\Gamma}h_{k-1}\cdot\nabla^{\Gamma}h_{1}\big)+\nabla^{\Gamma} h_{k-1}\cdot\partial_{\rho}\big(2\nu_0D\ve_{0}\big)
\nonumber\\
&\quad+\nabla^{\Gamma} h_{k-1}\cdot\partial_{\rho}\big(\nu_0\Div\ve_{0}\big)-\big(\nu(\theta_0)\eta'(\rho)\big)'\mathbf{u}_1h_{k-1} -\eta'(\rho)\mathbf{l}_0h_{k-1},
\end{align*}
and
\begin{align*}
  A^{k-1}=&\sum_{l=1}^{k-1}\partial_{\rho}\ve_l\cdot\nabla^ {\Gamma}h_{k-l}-\Div\ve_{k-1}-\eta'(\rho)\rho\mathbf{u}_{k-1}\cdot \nabla d_{\Gamma}- \eta'(\rho)d_{\Gamma}\sum_{l=1}^{k-1}\mathbf{u}_l\cdot \nabla^{\Gamma}h_{k-l}\nonumber\\& - \eta'(\rho)\sum_{l=1}^{k-1}\mathbf{u}_l\cdot \nabla d_{\Gamma} h_{k-l} + \eta'(\rho)\rho\sum_{l=0}^{k-2}\mathbf{u}_l\cdot \nabla^{\Gamma} h_{k-1-l}\\
  &+\eta'(\rho)\sum_{l=0}^{k-2}h_{l+1}\sum_{m=0}^{k-2-l}\mathbf{u}_m\cdot \nabla^{\Gamma} h_{k-1-l-m},
\end{align*}
and
\begin{align*}
  B^{k-1} =& \theta_0'(\rho)\big(\Delta^{\Gamma}h_{k-1}-\partial_t^\Gamma h_{k-1}-\ve_0\cdot\nabla^\Gamma h_{k-1}\big)-2\theta_0''(\rho)\nabla^\Gamma h_1\cdot\nabla^\Gamma h_{k-1}-\myeta'(\rho)g_{0}h_{k-1}\\
  &\quad +\myeta'(\rho)g_{k-1}d_{\Gamma}+ \theta_0'(\rho)\ve_{k-1}\cdot \nabla d_\Gamma
+\widetilde{B}^{k-1},
\end{align*}
here $\widetilde{\mathbf{V}}_2^{k-1}$ and $\widetilde{B}^{k-1}$  consist of ``known terms'' in the induction argument, which have an exponential decay as $|\rho|\to \infty$.
\begin{rem}\label{rem:B1}
  In the case $k=2$ one easily calculates that $\widetilde{B}^1= - \theta_0'(\rho)\rho g_0$ provided that $c_1\equiv 0$, which will be seen in Remark~\ref{rem:c1} below.
\end{rem}

We note that all functions with negative index, except $p_{-1}$,  are supposed to be zero. If the upper limit of the summations is less than the lower limit, the sum has to be understood as zero.  Furthermore, we remark that $f_{k-1}=f_{k-1}(c_{0},\ldots,c_{k-1})$ and $\nu_l= \nu_{l}(c_0\ldots, c_l)$  are terms from the Taylor expansion \eqref{inner:nu} and we use the convention $f_{0}(c_{0})=0$.

\subsection{Existence of the Expansion Terms}

\subsubsection{Solving zero order terms and derivation of jump conditions \eqref{eq:Limit3}-\eqref{eq:Limit5}}

It follows from \eqref{modieq:innerstokes0} and \eqref{modieq:innerdiv0} that
\begin{align}
\partial_{\rho}\ve_0-\mathbf{u}_0d_{\Gamma}\eta'(\rho)=\theta_0'(\rho)\theta_0''(\rho)+\partial_{\rho}p_{-1}=0.\label{modieq:innerstokes0'}
\end{align}
Therefore  one has
\begin{align}
p_{-1}(\rho,x,t)&=-\frac{\big(\theta_0'(\rho)\big)^2}{2}\label{inn-0-ve-2}
\end{align}
and
\begin{align}
\ve_0(\rho,x,z)=\overline{\ve}_0(x,t)+\mathbf{u}_0(x,t)d_{\Gamma}(x,t)\left(\eta(\rho)-\frac{1}{2}\right)\label{modieq:innerstokes00}
\end{align}
for all $(x,t)\in \Gamma(3\delta)$, $\rho \in\R$, and some function $\overline{\ve}_0(x,t)$.

We then get
\begin{align}
\mathbf{v}_{0}^+=\overline{\ve}_0+\frac{1}{2}\mathbf{u}_0d_{\Gamma},\ \ \mathbf{v}_{0}^-=\overline{\ve}_0-\frac{1}{2}\mathbf{u}_0d_{\Gamma}\qquad \text{in }\Gamma(3\delta)\label{modieq:innerstokes000}
\end{align}
and then
\begin{align}
\overline{\ve}_0=\frac{1}{2}\big(\mathbf{v}_{0}^++\mathbf{v}_{0}^-\big),\ \ \mathbf{u}_0d_{\Gamma}=\mathbf{v}_{0}^+-\mathbf{v}_{0}^-\qquad \text{in }\Gamma(3\delta),\label{modieq:innerstokes0000}
\end{align}
which immediately imply that $\eqref{eq:Limit4}$ and
\begin{align}
\ve_0(\rho,x,t)=\mathbf{v}_{0}^+(x,t)\eta(\rho)+\mathbf{v}_{0}^-(x,t)(1-\eta(\rho))\label{inn-0-ve-express}
\end{align}
and
\begin{eqnarray}\label{formula:u0}
\mathbf{u}_0=\left \{
\begin {array}{ll}
\frac{\mathbf{v}_{0}^+-\mathbf{v}_{0}^-}{d_{\Gamma}}&\text{for } (x,t)\in\Gamma(2\delta)\backslash\Gamma,\\
\\
\nn\cdot\nabla\big(\mathbf{v}_{0}^+-\mathbf{v}_{0}^-\big)&\text{for }  (x,t)\in\Gamma.
\end{array}
\right.
\end{eqnarray}
To proceed we show
\begin{prop}
On $\Gamma$ there hold 
\begin{align}
\int_{\R}\ve_0\cdot\nabla d_{\Gamma}\partial_{\rho}\ve_{0}d\rho&=0,\label{preeq:inner1-4}\\
\int_{\R}\Div\big(2\nu(\theta_0)D_d\ve_0\big)\big)d\rho&=\sigma_{\eta}\int_{\R}\widetilde{\Div}\ve_0d\rho=\sigma_{\eta}\mathbf{u}_0,\label{preeq:inner1-5}\\
\int_{\R}\mathbb{A}_0d\rho&=-\sigma H \no,\label{preeq:inner1-6}\\
\mathbf{u}_0\cdot\nn&=0,\label{preeq:inner1-66}
\end{align}
here $\sigma_{\eta}=\int_{\R}\nu(\theta_0)\eta'(\rho)d\rho$ and
\begin{align*}
\widetilde{\Div}\ve_0=(\mathbf{v}_{0}^+-\mathbf{v}_{0}^-)\cdot\nabla^2d_{\Gamma}+(\mathbf{v}_{0}^+-\mathbf{v}_{0}^-)\Delta d_{\Gamma}+(\nabla d_{\Gamma}\cdot\nabla)(\mathbf{v}_{0}^+-\mathbf{v}_{0}^-).
\end{align*}
\end{prop}
\begin{proof}
By direct computation one has
\begin{align}
\ve_0\cdot\nabla d_{\Gamma}\partial_{\rho}\ve_{0}&=\big(\mathbf{v}_{0}^+\eta\eta'+\mathbf{v}_{0}^-(1-\eta)\eta'\big)\cdot\nabla d_{\Gamma}(\mathbf{v}_{0}^+-\mathbf{v}_{0}^-),\label{preeq:inner1-1}\\
\Div\big(2\nu(\theta_0)D_d\ve_0\big)\big)&=\nu(\theta_0)\eta'(\rho)\Div\bigg((\mathbf{v}_{0}^+-\mathbf{v}_{0}^-)\nabla d_{\Gamma}+\big((\mathbf{v}_{0}^+-\mathbf{v}_{0}^-)\nabla d_{\Gamma}\big)^{T}\bigg)\nonumber\\&=:\nu(\theta_0)\eta'(\rho) \widetilde{\Div}\ve_0,\label{preeq:inner1-2}\\
\mathbb{A}_0&=-\theta_0'(\rho)\theta_0''(\rho)\nabla^{\Gamma}h_0+(\theta_0'(\rho))^2\Delta d_{\Gamma}\nabla d_{\Gamma}.\label{preeq:inner1-3}
\end{align}
Then we easily find that \eqref{preeq:inner1-4}-\eqref{preeq:inner1-6}. Moreover, noting that
  \begin{align*}0=\Div\mathbf{v}_{0}^{\pm}=\big(\nn\otimes\nn:\nabla\mathbf{v}_{0}^{\pm}\big)+\Div_{\btau}\mathbf{v}_{0}^{\pm},\end{align*}
we get
 \begin{align}\mathbf{u}_0\cdot\nn=\llbracket \nn\otimes\nn:\nabla\mathbf{v}_{0}^{\pm} \rrbracket=-\llbracket\Div_{\btau}\mathbf{v}_{0}^{\pm}\rrbracket=0.\label{preeq:inner1-33}\end{align}
\end{proof}

Applying Lemma \ref{modified version of Lemma 2.4} to \eqref{modieq:innerstokes1} we find
\begin{align}\label{formula:lg0}
&\big(p_0^+-p_0^-+\sigma\Delta d_{\Gamma}\big)\nabla d_{\Gamma}+\big(\ve_0^+-\ve_0^-\big)\partial_t d_{\Gamma}+\frac{1}{2}\big(\ve_0^++\ve_0^-\big)\cdot\nabla d_{\Gamma}\big(\ve_0^+-\ve_0^-\big)\nonumber\\&\quad-2\big(\nu^+ D\ve_0^+-\nu^- D\ve_0^-\big)\cdot\nabla d_{\Gamma}-\sigma_{\eta}\widetilde{\Div}\ve_0+\mathbf{l}_{0}d_{\Gamma}+\sigma_{\eta}\mathbf{u}_0
=0.
\end{align}
Using \eqref{preeq:inner1-4}-\eqref{preeq:inner1-6} we  can get
\begin{align}\label{formula:lg0000}
2\llbracket \nu^\pm D\ve_0^\pm \rrbracket\nn -\llbracket  p_0^\pm \rrbracket\nn =\sigma\Delta d_{\Gamma}\nn
\end{align}
i.e., we get
\eqref{eq:Limit3} and then
\begin{align*}
\mathbf{l}_{0}&=
                \frac{\big(p_0^--p_0^+-\sigma\Delta d_{\Gamma}\big)\nabla d_{\Gamma}+\big(\ve_0^--\ve_0^+\big)\partial_t d_{\Gamma}-\frac{1}{2}\big(\ve_0^++\ve_0^-\big)\cdot\nabla d_{\Gamma}\big(\ve_0^+-\ve_0^-\big)}{d_{\Gamma}}\\
  &\quad +\frac{2\big(\nu^+ D\ve_0^+-\nu^- D\ve_0^-\big)\cdot\nabla d_{\Gamma}+\sigma_{\eta}\big(\widetilde{\Div}\ve_0-\mathbf{u}_0\big)}{d_{\Gamma}}\qquad \text{for }(x,t)\in\Gamma(2\delta)\backslash\Gamma
\end{align*}
 and
\begin{align*}
  \mathbf{l}_{0}&=\nn\cdot\nabla\bigg(\big(p_0^--p_0^+-\sigma\Delta d_{\Gamma}\big)\nabla d_{\Gamma}-\frac{1}{2}\big(\ve_0^++\ve_0^-\big)\cdot\nabla d_{\Gamma}\big(\ve_0^+-\ve_0^-\big)\\
                &\quad +\big(\ve_0^--\ve_0^+\big)\partial_t d_{\Gamma}+2\big(\nu^+ D\ve_0^+-\nu^- D\ve_0^-\big)\cdot\nabla d_{\Gamma}+\sigma_{\eta}\big(\widetilde{\Div}\ve_0-\mathbf{u}_0\big)\bigg)\qquad \text{for }  (x,t)\in\Gamma.
\end{align*}
Applying Lemma~\ref{ODEsolver} to \eqref{modieq:innerAL1} we see that the compatibility condition is equivalent to \eqref{eq:Limit5}. To ensure $c_1\equiv 0$ in \eqref{modieq:innerAL1}, we define
\begin{eqnarray}\label{formula:g0}
g_0=\left \{
\begin {array}{ll}
\frac{\Delta d_{\Gamma}-\ve_0\cdot\nabla d_{\Gamma}-\partial_t d_{\Gamma}}{ d_{\Gamma}}&\quad \text{for } (x,t)\in\Gamma(2\delta)\backslash\Gamma,\\
\\
\nn\cdot\nabla\big(\Delta d_{\Gamma}-\ve_0\cdot\nabla d_{\Gamma}-\partial_t d_{\Gamma}\big)&\quad \text{for }  (x,t)\in\Gamma.
\end{array}
\right.
\end{eqnarray}
\begin{rem}\label{rem:c1}
  For this choice of $g_0$ we see that the right-hand side of \eqref{modieq:innerAL1} vanishes. Therefore $c_1\equiv 0$.
\end{rem}
Furthermore one has
\begin{align}\label{jumpcondition:ve0}
\ve_0^-|_{\partial\Omega}= 0\qquad \text{on }\partial\Omega.
\end{align}

In summary we can derive that $(\ve_0^\pm, p_0^\pm)$ together with $(\Gamma_t)_{t\in [0,T_0]}$ solve the sharp interface limit system \eqref{eq:Limit1}-\eqref{eq:Limit5}. Moreover,
\begin{lem}[The zeroth order terms]
 \label{zeroorder}~\\Let $(\mathbf{v}^{\pm},p^{\pm})$
be the extension of $(\mathbf{v}_0^{\pm},p_0^{\pm})$ satisfying the sharp interface limit system to $\Omega^{\pm}\cup\Gamma(2\delta)$ as in {Remark~\ref{Outer-Rem}}.
We define the terms of the outer expansion $(c_{0}^{\pm},\mathbf{v}_{0}^{\pm},p_{0}^{\pm})$
for $(x,t)\in\Omega^{\pm}\cup\Gamma(2\delta;T)$
as
\begin{equation}
c_{0}^{\pm}(x,t)=\pm1,\;\mathbf{v}_{0}^{\pm}(x,t)=\mathbf{v}^{\pm}(x,t),\;p_{0}^{\pm}(x,t)=p^{\pm}(x,t),\;p_{-1}^{\pm}(x,t)=0,\label{eq:0outdef}
\end{equation}
the terms of the inner expansion $(c_{0},\mathbf{v}_{0})$
as
\begin{align}
c_{0}(\rho,x,t) & =\theta_{0}(\rho),\label{eq:c0def}\\
\mathbf{v}_{0}(\rho,x,t) & =\mathbf{v}_{0}^{+}(x,t)\eta(\rho)+\mathbf{v}_{0}^{-}(x,t)\left(1-\eta(\rho)\right),\label{eq:v0def}
\end{align}
for all $(\rho,x,t)\in\mathbb{R}\times\Gamma(2\delta)$.
Then there are smooth and bounded $g_{0}\colon \Gamma(2\delta)\rightarrow\mathbb{R}$,
and $\mathbf{u}_{0}$, $\mathbf{l}_{0}\colon \Gamma(2\delta)\rightarrow\mathbb{R}^{2}$
such that the outer equations \eqref{eq:c0out},  \eqref{eq:stokes-Outerk},
\eqref{eq:Outervpdefine} (for $k=0$), the inner equations \eqref{modieq:innerstokesk},
\eqref{modieq:innerdivk}, \eqref{modieq:innerALk}
(for $k=0$), and the inner-outer matching conditions \eqref{eq:matchcon}
(for $k=0$) are satisfied.
\end{lem}

\subsubsection{Determination of the  Higher Order Terms}
In this part we firstly aim to derive the equations related to higher order terms and then give the detailed induction argument in Lemma \ref{k-thorder}.

Applying Lemma  \ref{ODEsolver} to \eqref{modieq:innerALk} we can get $c_k$. The compatibility condition \eqref{compatiblity:ODEsolver} for \eqref{modieq:innerALk} is equivalent to
\begin{align}
\partial_t^\Gamma h_{k-1}+\ve_0\cdot\nabla^\Gamma h_{k-1}-\Delta^{\Gamma}h_{k-1}&+\sigma_0^{-1}\int_{\R}\nabla d_\Gamma\cdot \ve_{k-1}\big(\theta_0'(\rho)\big)^2d\rho+g_{0}h_{k-1}-g_{k-1}d_{\Gamma}\nonumber\\&=\sigma_0^{-1}\int_{\R}\widetilde{B}^{k-1}\theta_0'(\rho)d\rho,\label{evolution law:hk-2}
\end{align}
where $\sigma_0=\int_{\R}(\myeta'(\rho))^2d\rho$.
Then one has
\begin{align}
&\partial_t^\Gamma h_{k-1}+\ve_0\cdot\nabla^\Gamma h_{k-1}-\Delta^{\Gamma}h_{k-1}+\sigma_0^{-1}\int_{\R}\no\cdot \ve_{k-1}\big(\theta_0'(\rho)\big)^2d\rho+g_{0}h_{k-1}\nonumber\\&\qquad\qquad=\sigma_0^{-1}\int_{\R}\widetilde{B}^{k-1}\theta_0'(\rho)d\rho\ \ \text{on}\ \ \Gamma\label{evolution law:hk-2'}\end{align}
 and takes
\begin{align}\nonumber
  g_{k-1}&=\frac{\big(\partial_t^\Gamma h_{k-1}+\ve_0\cdot\nabla^\Gamma h_{k-1}-\Delta^{\Gamma}h_{k-1}\big)+\sigma_0^{-1}\int_{\R}\nabla d_\Gamma\cdot \ve_{k-1}\big(\theta_0'(\rho)\big)^2d\rho+g_{0}h_{k-1}}{ d_{\Gamma}}\\\label{formula:gk-1}
  &\quad - \frac{\int_{\R}\widetilde{B}^{k-1}\theta_0'(\rho)d\rho}{\sigma_0 d_{\Gamma}}\qquad \text{for } (x,t)\in\Gamma(2\delta)\backslash\Gamma
\end{align}
and
\begin{align}\nonumber
g_{k-1}=&
\nn\cdot\nabla\bigg(\partial_t^\Gamma h_{k-1}+\ve_0\cdot\nabla^\Gamma h_{k-1}-\Delta^{\Gamma}h_{k-1}+\sigma_0^{-1}\int_{\R}\no\cdot \ve_{k-1}\big(\theta_0'(\rho)\big)^2d\rho+g_{0}h_{k-1}\\\label{formula:gk-1b} & \qquad-\int_{\R}\widetilde{B}^{k-1}\theta_0'(\rho)d\rho\bigg)\quad \text{for }  (x,t)\in\Gamma
\end{align}
such that the compatibility condition \eqref{compatiblity:ODEsolver} in $\Gamma(2\delta)\backslash\Gamma$ holds and $g_{k-1}$ is smooth.

The compatibility condition \eqref{modified version of Lemma 2.4-2} for  \eqref{modieq:innerstokesk} is equivalent to
\begin{align}
&(p_{k-1}^+-p_{k-1}^-)\nabla d_{\Gamma}+\big( \ve_{k-1}^+- \ve_{k-1}^-\big)\partial_td_{\Gamma}-\big( 2\nu^+ D\ve_{k-1}^+- 2\nu^- D\ve_{k-1}^-\big)\nabla d_{\Gamma}+\mathbf{l}_{k-1}d_{\Gamma}\nonumber\\
 & \qquad+\int_{\R}\big(\ve_0\cdot\nabla d_{\Gamma}\partial_{\rho}\ve_{k-1}+\ve_{k-1}\cdot\nabla d_{\Gamma}\partial_{\rho}\ve_{0}-\Div\big(2\nu(\theta_0(\rho))D_d\ve_{k-1}\big)\big)d\rho
 \nonumber\\
 & \quad =\sigma \big(\Delta^\Gamma h_{k-1}\nabla d_\Gamma +\Delta d_\Gamma \nabla^\Gamma h_{k-1}\big)- \int_{\R}\widetilde{\mathbf{V}}_1^{k-1}d\rho- \int_{\R}\widetilde{\mathbf{V}}_2^{k-1}d\rho.\label{oldmodieq:innerstokes1-solvability-k2}
\end{align}
If it is satisfied,
the solution to \eqref{modieq:innerstokesk} is given by
\begin{align}\label{formula:vek}
\ve_k+(\mathbf{u}_0h_{k}-\mathbf{u}_kd_{\Gamma})\eta(\rho)&=\ve_k^- +\int_{-\infty}^\rho\frac{1}{\nu(\theta_0(r))}\int_{-\infty}^r\mathbf{V}^{k-1}(s,x,t)\,ds\,dr\nonumber\\
&=:\ve_k^-+\mathbf{W}^{k-1}.
\end{align} 
By taking $\rho\to\infty$ in \eqref{formula:vek} on $\Gamma$ we then arrive at
\begin{align}\label{jumpcondition:vek}
{\llbracket \ve_k\rrbracket\cdot\nn}&{=\int_{\R}\frac{1}{\nu(\theta_0(r))}\int_{-\infty}^r\mathbf{V}^{k-1}(s,x,t)\cdot\nn \,ds\,dr-\mathbf{u}_0\cdot\nn h_{k}}\nonumber\\&{=\int_{\R}\frac{1}{\nu(\theta_0(r))}\int_{-\infty}^r\mathbf{V}^{k-1}(s,x,t)\cdot\nn\,ds\,dr=:\hat{a}_{k-1} }
\end{align}
on $\Gamma$, where we have used $\mathbf{u}_0\cdot\nn=0$ on $\Gamma$ and
\begin{align}\label{jumpcondition:vek-0}
\llbracket \ve_k\rrbracket\cdot\btau&=\int_{\R}\frac{1}{\nu(\theta_0(r))}\int_{-\infty}^r\mathbf{V}^{k-1}(s,x,t)\cdot\btau\, ds\, dr-\mathbf{u}_0\cdot\btau h_{k}
\nonumber\\&=: \check{a}_{k-1}-\mathbf{u}_0\cdot\btau h_{k}.
\end{align}
Furthermore it follows from \eqref{formula:vek} that $\mathbf{u}_{k}$ will be defined by
\begin{eqnarray}\label{formula:uk}
\mathbf{u}_k=\left \{
\begin {array}{ll}
\frac{\mathbf{v}_k^{+}-\mathbf{v}_k^{-}+\mathbf{u}_{0}h_{k}-\mathbf{W}^{k-1}(+\infty,x,t)}{d_{\Gamma}}&\ \text{for } (x,t)\in\Gamma(2\delta)\backslash\Gamma,\\
\\
\nn\cdot\nabla\big(\mathbf{v}_k^{+}-\mathbf{v}_k^{-}+\mathbf{u}_{0}h_{k}-\mathbf{W}^{k-1}(+\infty,x,t)\big)&\ \text{for } (x,t)\in\Gamma
\end{array}
\right.
\end{eqnarray}
and then
\begin{align}\label{newformula:vek}
\ve_k(\rho,x,t)=&\ve_k^+(x,t)\eta(\rho)+\ve_k^-(x,t)(1-\eta(\rho))+(1-\eta(\rho))\int_{-\infty}^\rho\frac{1}{\nu(\theta_0(r))}\int_{-\infty}^r\mathbf{V}^{k-1}(s,x,t)\,ds\,dr
\nonumber\\&-\eta(\rho)\int_\rho^{+\infty}\frac{1}{\nu(\theta_0(r))}\int_{-\infty}^r\mathbf{V}^{k-1}(s,x,t)\,ds\,dr,
\end{align}
which satisfies the inner-outer matching conditions \eqref{eq:matchcon}.

By induction one has
\begin{align}\label{newformula:vek-1}
\ve_{k-1}=&\ve_{k-1}^+\eta+\ve_{k-1}^-(1-\eta)+(1-\eta)\int_{-\infty}^\rho\frac{1}{\nu(\theta_0(r))}\int_{-\infty}^r\mathbf{V}^{k-2}(s,x,t)\,ds\,dr
\nonumber\\&-\eta\int_\rho^{+\infty}\frac{1}{\nu(\theta_0(r))}\int_{-\infty}^r\mathbf{V}^{k-2}(s,x,t)\,ds\,dr.
\end{align}
Using \eqref{jumpcondition:vek} (with $k-1$ instead of $k$) and \eqref{newformula:vek-1} in \eqref{evolution law:hk-2} leads to
\begin{align}
&\partial_t^\Gamma h_{k-1}+\ve_0\cdot\nabla^\Gamma h_{k-1}-\Delta^{\Gamma}h_{k-1}+\sigma_0^{-1}\sigma_2\no\cdot\ve_{k-1}^++(1-\sigma_0^{-1}\sigma_2)\no\cdot\ve_{k-1}^-+g_{0}h_{k-1}\nonumber\\&\qquad\qquad\qquad\qquad=\sigma_0^{-1}\int_{\R}\widetilde{B}^{k-1}\theta_0'(\rho)d\rho-\sigma_0^{-1}\mathbf{w}_{k-2},\ \ \text{on}\ \ \Gamma,\label{evolution law:hk-2-new}\end{align}
where  $\sigma_2=\int_{\R}\eta(\rho)\big(\theta_0'(\rho)\big)^2d\rho$ and
\begin{align*}
\mathbf{w}_{k-2}(x,t)=&\int_{-\infty}^{+\infty}(1-\eta(\rho))\big(\theta_0'(\rho)\big)^2\int_{-\infty}^\rho\frac{1}{\nu(\theta_0(r))}\int_{-\infty}^r\no\cdot\mathbf{V}^{k-2}(s,x,t)\,ds\,dr\,d\rho
\nonumber\\&-\int_{-\infty}^{+\infty}\eta(\rho)\big(\theta_0'(\rho)\big)^2\int_\rho^{+\infty}\frac{1}{\nu(\theta_0(r))}\int_{-\infty}^r\no\cdot\mathbf{V}^{k-2}(s,x,t)\,ds\,dr\,d\rho.
\end{align*}
Inserting  \eqref{newformula:vek-1} into \eqref{oldmodieq:innerstokes1-solvability-k2} implies  that 
\begin{align} 
\J_{k-1}+\mathbf{l}_{k-1}d_{\Gamma}=- \int_{\R}\widetilde{\mathbf{V}}_1^{k-1}d\rho- \int_{\R}\widetilde{\mathbf{V}}_2^{k-1}d\rho,\label{newmodieq:innerstokes1-solvability-k2}
\end{align}
where
\begin{align}
\J_{k-1}&=(p_{k-1}^+-p_{k-1}^-)\nabla d_{\Gamma}+\big( \ve_{k-1}^+- \ve_{k-1}^-\big)\partial_td_{\Gamma}-\big( 2\nu^+ D\ve_{k-1}^+- 2\nu^- D\ve_{k-1}^-\big)\nabla d_{\Gamma}\nonumber\\&
+\tfrac{ (\ve_{0}^+- \ve_{0}^-)\cdot\nabla d_{\Gamma}(\ve_{k-1}^+- \ve_{k-1}^-)}2+\tfrac{(\ve_{k-1}^+- \ve_{k-1}^-)\cdot\nabla d_{\Gamma}( \ve_{0}^+- \ve_{0}^-)}2-\sigma \big(\Delta^\Gamma h_{k-1}\nabla d_\Gamma +\Delta d_\Gamma \nabla^\Gamma h_{k-1}\big)\nonumber\\&-\sigma_{\eta}\big( \ve_{k-1}^+- \ve_{k-1}^-\big)\cdot\nabla^2d_{\Gamma}-\sigma_{\eta}\big( \ve_{k-1}^+- \ve_{k-1}^-\big)\Delta d_{\Gamma}-\sigma_{\eta}\nabla\big( \ve_{k-1}^+- \ve_{k-1}^-\big)\cdot\nabla d_{\Gamma}.
\end{align}
Noting that \begin{align*} 
 \int_{\R}\widetilde{\mathbf{V}}_1^{k-1}d\rho=&-\frac{1}{2}\big( \ve_{0}^+- \ve_{0}^-\big)\big( \ve_{0}^++\ve_{0}^-\big)\cdot\nabla^{\Gamma}h_{k-1}-\big( \ve_{0}^+- \ve_{0}^-\big)\partial_{t}^{\Gamma}h_{k-1}-( p_{0}^+- p_{0}^-)\nabla^{\Gamma}h_{k-1}\\&
+ 2(\nu^+D\ve_0^+-\nu^-D\ve_0^-)\cdot\nabla^{\Gamma}h_{k-1}-\mathbf{l}_0h_{k-1},
\end{align*}
we get
\begin{align}
&{\llbracket 2\nu^\pm D\ve_{k-1}^\pm -p_{k-1}^\pm \tn{I}\rrbracket\no_{\Gamma_t}} {+\llbracket \ve_{k-1}^\pm \rrbracket\big(\partial_td_{\Gamma}-\sigma_{\eta}\Delta d_{\Gamma}\big)-\sigma_{\eta}\llbracket \ve_{k-1}^\pm \rrbracket\cdot\nabla^2d_{\Gamma}}{-\sigma_{\eta}\llbracket \nabla\ve_{k-1}^\pm \rrbracket\cdot\nabla d_{\Gamma}}\nonumber\\&\qquad= \mathbf{b}_{k-1}-\llbracket 2\nu^\pm D\ve_{0}^\pm -p_{0}^\pm \tn{I}\rrbracket\no_{\Gamma_t}\cdot\nabla^\Gamma h_{k-1}-\sigma \big(\Delta^\Gamma h_{k-1}\nabla d_\Gamma +\Delta d_\Gamma \nabla^\Gamma h_{k-1}\big) \ \text{on }\Gamma\label{modieq:innerstokes1-solvability-k2}
\end{align}
and
\begin{eqnarray}\label{formula:lk-1}
\mathbf{l}_{k-1}=\left \{
\begin {array}{ll}
\frac{\mathbf{b}_{k-1}-\J_{k-1}- \int_{\R}\widetilde{\mathbf{V}}_1^{k-1}d\rho}{d_{\Gamma}},& \text{for } (x,t)\in\Gamma(2\delta)\backslash\Gamma,\\
\\
\nn\cdot\nabla\big(\mathbf{b}_{k-1}-\J_{k-1}- \int_{\R}\widetilde{\mathbf{V}}_1^{k-1}d\rho\big),& \text{for } (x,t)\in\Gamma,
\end{array}
\right.,
\end{eqnarray}
where \begin{align*} 
\mathbf{b}_{k-1}=- \int_{\R}\widetilde{\mathbf{V}}_2^{k-1}d\rho.
\end{align*}
Thanks to \eqref{modieq:innerstokes00} one gets $\partial_\rho\ve_0=\mathbf{u}_0d_\Gamma\eta'$. Then
 the equation \eqref{modieq:innerdivk} becomes
 \begin{align}
\big(\partial_{\rho}\ve_k+(\mathbf{u}_0h_{k}-\mathbf{u}_kd_{\Gamma})\eta'(\rho)\big)\cdot\nabla d_{\Gamma}=A^{k-1}.\label{modieq:innerdivk-new}
\end{align}
Multiplying \eqref{modieq:innerstokesk} by $\nabla d_{\Gamma}$ and utilizing \eqref{modieq:innerdivk-new} one has the equation for $p_{k-1}$
\begin{align}
 &\partial_{\rho}\bigg(p_{k-1}-\nu_1\partial_{\rho}\ve_{k-1}\cdot\nabla d_{\Gamma}+\ve_{k-1}\cdot\nabla d_{\Gamma}\partial_t d_{\Gamma}+\ve_{k-1}\cdot\nabla d_{\Gamma}\ve_{0}\cdot\nabla d_{\Gamma} \nonumber\\&\qquad-\big(2\nu(\theta_0(\rho))D\ve_{k-1}:\nabla d_{\Gamma}\otimes\nabla d_{\Gamma}\big)-\int_{-\infty}^\rho\Div\big(2\nu(\theta_0(r))D_d\ve_{k-1}(r,\cdot)\big)\cdot\nabla d_{\Gamma}dr
\nonumber\\&\qquad+\nu(\theta_0(\rho))\Div\ve_{k-1}+\mathbf{l}_{k-1}d_{\Gamma}\eta(\rho)\cdot\nabla d_{\Gamma}
  -\nu(\theta_0(\rho))A^{k-1}\nonumber\\
  &\qquad+\int_{-\infty}^\rho\big(\widetilde{\mathbf{V}}_1^{k-1}(r,\cdot)+\widetilde{\mathbf{V}}_2^{k-1}(r,\cdot)\big)\cdot\nabla d_{\Gamma}dr\bigg)=0.\label{modieq:innerstokesk-new}
\end{align}
Then we deduce
\begin{align}
p_{k-1}&=\nu_1\partial_{\rho}\ve_{k-1}\cdot\nabla d_{\Gamma}+\nu(\theta_0(\rho))A^{k-1}-\nu(\theta_0(\rho))\Div\ve_{k-1}
\nonumber\\&\ +\big(\ve_{k-1}^+\eta+\ve_{k-1}^-(1-\eta)\big)\cdot\nabla d_{\Gamma}\partial_t d_{\Gamma}-\ve_{k-1}\cdot\nabla d_{\Gamma}\partial_t d_{\Gamma}
\nonumber\\&\ +\big(2\nu(\theta_0(\rho))D\ve_{k-1}:\nabla d_{\Gamma}\otimes\nabla d_{\Gamma}\big)-\bigg(2\nu^+ D\ve_{k-1}^+\eta+ 2\nu^- D\ve_{k-1}^-(1-\eta):\nabla d_{\Gamma}\otimes\nabla d_{\Gamma}\bigg)
\nonumber\\&\ +\bigg(\eta\int_{+\infty}^\rho\Div\big(2\nu(\theta_0(r))D_d\ve_{k-1}\big)\cdot\nabla d_{\Gamma}dr+(1-\eta)\int_{-\infty}^\rho\Div\big(2\nu(\theta_0(r)D_d\ve_{k-1}\big)\cdot\nabla d_{\Gamma}dr\bigg)
\nonumber\\&\ +\bigg(\eta\int_\rho^{+\infty}\big(\widetilde{\mathbf{V}}_1^{k-1}+\widetilde{\mathbf{V}}_2^{k-1}\big)(r,\cdot)\cdot\nabla d_{\Gamma}dr-(1-\eta)\int_{-\infty}^\rho\big(\widetilde{\mathbf{V}}_1^{k-1}+\widetilde{\mathbf{V}}_2^{k-1}\big)(r,\cdot)\cdot\nabla d_{\Gamma}dr\bigg)
\nonumber\\&\ +p_{k-1}^+\eta+p_{k-1}^-(1-\eta)\label{modieq:innerstokesk-new1}
\end{align}
which satisfies  the inner-outer matching conditions \eqref{eq:matchcon}. Here we have used \eqref{formula:lk-1}.

Furthermore one has
\begin{align}\label{bccondition:vek}
{\ve_k^-|_{\partial\Omega}}= 0,\qquad \text{on }\partial\Omega.
\end{align}
\begin{rem}
\label{WU} We note that the terms $\mathbf{W}^{\pm}$ in \eqref{modieq:innerstokes}
are not multiplied by terms of the kind $\left(d_{\Gamma}-\eps\left(\rho+h^{\eps}\right)\right)$.
Therefore we have to make sure they vanish on the set $S^{\eps}$. This
is done by choosing the constant $C_{S}>0$ in a suitable
way. More precisely,  we choose
\[
C_{S}:=\left\Vert h_{1}\right\Vert _{C^{0}\left(\mathbb{T}^{1}\times[0,T_0]\right)}+2
\]
and assume that $\eps>0$ is so small that
\begin{equation}
\Big|\sum_{k\geq1}\eps^{k}h_{k+1}(S(x,t),t)\Big|\leq1\label{eq:remhglm}.
\end{equation}
It turns out that $h_{1}$
does not depend on the term $\eps^{2}(\mathbf{W}^{+}\eta^{C_{S},+}+\mathbf{W}^{-}\eta^{C_{S},-})$.
Therefore this choice of $C_{S}$ does not cause problems. Then it is possible to show as in \cite[Remark 4.2 (2)]{AlikakosLimitCH}
that for $\rho=\frac{d_{\Gamma}(x,t)}{\eps}-h^{\eps}\left(S(x,t),t\right)$
and $(x,t)\in\Gamma(2\delta)$ such that $d_{\Gamma}(x,t)\geq0$
it follows $\rho\geq-C_{S}+1$. Thus, $\eta^{C_{S},-}(\rho)=0$
and since $(x,t)\in\overline{\Omega^{+}}$ we have $\mathbf{W}^{+}(x,t)=0$
and therefore
\[
\eps^{2}\left(\mathbf{W}^{+}\eta^{C_{S},+}+\mathbf{W}^{-}\eta^{C_{S},-}\right)=0.
\]
An analoguous arguments show this  for $d_{\Gamma}(x,t)<0$.
\end{rem}

\begin{lem}[The $k$-th order terms]
 \label{k-thorder}~\\
Let $k\in\{ 1,\ldots,N+2\} $ be given. Then there are
smooth functions
\[
\mathbf{v}_{k},\mathbf{v}_{k}^{\pm},\mathbf{u}_{k-1},\mathbf{l}_{k-1},c_{k},c_{k}^{\pm},g_{k-1},h_{k},p_{k-1},p_{k}^{\pm},p_{-1}=-\big(\theta_0'(\rho)\big)^2
\]
which are bounded on their respective domains, such that for the $k$-th
order the outer equations  \eqref{eq:c0out},  \eqref{eq:stokes-Outerk},
\eqref{eq:Outervpdefine}, the inner equations \eqref{modieq:innerstokesk},
\eqref{modieq:innerdivk}, \eqref{modieq:innerALk},
the inner-outer matching conditions \eqref{eq:matchcon} are satisfied. Moreover, $(\mathbf{v}_{k}^{\pm},
p_{k}^{\pm},h_{k})$ satisfies
\begin{alignat}{2}
\partial_t\mathbf{v}_{k}^{\pm}-\nu(\pm 1)\Delta\ve_{k}^{\pm}&+\nabla p_{k}^{\pm}  ={-\sum_{j=0}^k \ve_{k-j}^\pm \cdot \nabla \ve_j^\pm} &  & \text{in }\Omega^{\pm}(t),\label{system:korder-1}\\
\operatorname{div}\mathbf{v}_{k}^{\pm} & =0  &&  \text{in }\Omega^{\pm}(t),\label{system:korder-2}\\
{\llbracket \ve_k\rrbracket\cdot\nn}& =\hat{a}_{k-1}   &&  \text{on }\Gamma_t,\label{system:korder-3}\\
{\llbracket \ve_k\rrbracket\cdot\btau}& =\check{a}_{k-1} -\mathbf{u}_0\cdot\btau h_{k}\circ X_0^{-1}  &&  \text{on }\Gamma_t,\label{system:korder-33}\\
\llbracket 2\nu^\pm D\ve_{k}^\pm -p_{k}^\pm \tn{I}\rrbracket\no_{\Gamma_t} &+\llbracket \ve_{k}^\pm \rrbracket\big(\partial_td_{\Gamma}-\sigma_{\eta}\Delta d_{\Gamma}\big) -\sigma_{\eta}\llbracket \ve_{k}^\pm \rrbracket\cdot\nabla^2d_{\Gamma}\nonumber\\-\sigma_{\eta}\llbracket \nabla\ve_{k}^\pm \rrbracket\cdot\nabla d_{\Gamma}&= -\sigma \big(\Delta^\Gamma h_{k-1}\circ X_0^{-1}\nabla d_\Gamma +\Delta d_\Gamma \nabla^\Gamma h_{k-1}\circ X_0^{-1}\big)+\mathbf{b}_{k} &&  \text{on }\Gamma_t,\label{system:korder-4}\\
\partial_t^\Gamma h_{k}+\ve_0\circ X_0\cdot\nabla_\Gamma &h_{k}-\Delta_{\Gamma}h_{k}+\left[\sigma_0^{-1}\sigma_2\no\cdot\ve_{k}^+ +(1-\sigma_0^{-1}\sigma_2)\no\cdot\ve_{k}^-\right]\circ X_0\nonumber\\+g_{0}h_{k}&{=\tfrac1{\sigma_0}\int_{\R}\widetilde{B}^{k}\theta_0'(\rho)d\rho-\tfrac1{\sigma_0}\mathbf{w}_{k-1}\circ X_0}  && \text{on}\ \T^1, \label{system:korder-6}\\
\mathbf{v}_k^{-} & =0  && \text{on }\partial\Omega\label{system:korder-5}
\end{alignat}
for every $t\in (0,T_0)$.
Here  \eqref{system:korder-4} and \eqref{system:korder-6} come from  \eqref{modieq:innerstokes1-solvability-k2} and \eqref{evolution law:hk-2-new}(with $k-1$ instead of $k$).
Additionally,
it holds $h_{k+1}(s,0)=0$ for all $s\in\mathbb{T}^{1}$.
Here $\mathbf{v}_{k}^{\pm}$,  $c_{k}^{\pm}$ and
$p_{k}^{\pm}$ are considered to be extended onto $\Omega^{\pm}\cup\Gamma(2\delta)$ as in Remark~\ref{Outer-Rem}.2.
\end{lem}
\begin{proof} We mainly give the induction procedure. For the induction step we assume that 
\begin{align*}
\{
\mathbf{v}_{i},\mathbf{v}_{i}^{\pm},\mathbf{u}_{i-1},\mathbf{l}_{i-1},c_{i},c_{i}^{\pm},g_{i-1},h_{i},p_{i-1},p_{i}^{\pm}
\}
\end{align*}
are known for $0\leq i\leq k-1$ and the matching conditions \eqref{eq:matchcon} for $\varphi=c_{i},\mathbf{v}_{i},p_{i}$ with $0\leq i\leq k-1$ hold. Next we obtain the terms for  $i=k$  by the following four steps.

\noindent
\emph{Step 1:} Noting that  $\widetilde{B}^{k-1}, \widetilde{\mathbf{V}}_1^{k-1}, \widetilde{\mathbf{V}}_2^{k-1},\mathbf{W}^{k-2}$ are known and then $\mathbf{l}_{k-1}$ is known by \eqref{formula:lk-1}. Collecting \eqref{eq:c1out} and \eqref{eq:ckout} we have $
c_{k}^{\pm}=0$ (for $k\geq 1$).

\noindent
\emph{Step 2:} We have seen that the compatibility condition for \eqref{modieq:innerALk} is equivalent to the evolution equation \eqref{evolution law:hk-2} for $h_{k-1}$ on $\T^1\times (0,T_0)$, which is satisfied by assumption. Then we can get $c_k$, which satisfies the inner-outer matching conditions \eqref{eq:matchcon}, by solving the equation \eqref{modieq:innerALk} and $g_{k-1}$ by \eqref{formula:gk-1}-\eqref{formula:gk-1b}. Moreover, since  $h_{k-1}$ is known,  
$\mathbf{u}_{k-1}$ can be obtained from \eqref{formula:uk} (with $k-1$ instead of $k$) and $A^{k-1}$ is known.

\noindent
\emph{Step 3:}  We then get $p_{k-1}$ by \eqref{modieq:innerstokesk-new1}. Therefore $\mathbf{V}^{k-1}$ and $\mathbf{W}^{k-1}$ are known.

\noindent
\emph{Step 4:} According to the procedure of the asymptotic expansion we have the closed system \eqref{system:korder-1}-\eqref{system:korder-5} for $(\mathbf{v}_{k}^{\pm},
p_{k}^{\pm},h_{k})$.  The details to solve the system will be given in the following Theorem \ref{solvingkorder}.  We then get $\mathbf{v}_{k}$ by \eqref{newformula:vek}.

Hence we completed the proof of this lemma.
\end{proof}

\noindent
\begin{proof*}{of Theorem~\ref{thm:Approx1}}
From the construction one can verify the statements of Theorem~\ref{thm:Approx1} in the same way in e.g.\ in \cite[Section~4]{AbelsMarquardt2}. In the present situation the estimates are even much simpler since no coefficients, which depend on  $\eps$ and are controlled in $\eps$ only in a limited way (as the $(M-\frac12)$-order terms in \cite{AbelsMarquardt2}), do not occur.  
\end{proof*}
\begin{thm}\label{solvingkorder} Let $T\in (0,\infty)$, $\ue_0\colon \Omega \times [0,T_0]{\to \R^2}$ {and $\ue\colon \Omega^\pm\to \R^2$} be smooth, and
  \begin{align*}
    \mathbf{f}&\in L^2(\Omega\times (0,T))^2, \quad
    g\in L^2(0,T;H^1(\Omega\setminus \Gamma_t)), \\
    \mathbf{a}_1&\in L^2(0,T;H^{\frac32}(\Gamma_t))\cap H^{\frac34}(0,T;L^2(\Gamma_t)), \\
    \mathbf{b}&\in L^2(0,T;H^{\frac12}(\Gamma_t))^2\cap H^{\frac14}(0,T;L^2(\Gamma_t))^2, \\
    \mathbf{w}& \in L^2(0,T;H^{\frac12}(\T^1)), \quad
    \mathbf{a}_2\in L^2(0,T;H^{\frac32}(\partial\Omega))^2\cap H^{\frac34}(0,T;L^2(\partial\Omega))^d,\\
    \mathbf{v}_0 &\in H^1(\Omega\setminus \Gamma(0))^2, \quad h_0 \in H^{\frac32}(\T^1)
  \end{align*}
  satisfy
\begin{align}\label{eq:CompCondStokes}
\int_{\Omega\setminus \Gamma(t)}g\,\sd x=-\int_{\Gamma_{t}}\mathbf{a}_{1}\cdot\mathbf{n}_{\Gamma_{t}}\sd\mathcal{\mathcal{H}}^{1}(p)+\int_{\partial\Omega}\mathbf{a}_{2}\cdot\mathbf{n}_{\partial\Omega}\sd\mathcal{H}^{1}(p)
\end{align}
for almost all $t\in (0,T)$, that $g=\Div R$ for some $R\in H^1(0,T;L^2(\Omega))^2\cap L^2(0,T;H^2(\Omega\setminus \Gamma(t)))^2$, $\Div \mathbf{v}_0|_{t=0} = g|_{t=0}$ and $\llbracket\mathbf{v}_0\rrbracket =\mathbf{a}_{1}|_{t=0}$, $\ve^-_0|_{\partial\Omega}= \mathbf{a}_2|_{t=0}$. Then there are unique
\begin{align*}
  \mathbf{v}&\in L^2(0,T; H^2(\Omega\setminus \Gamma_t))^2\cap H^1(0,T; L^2(\Omega))^2, \quad
  p\in L^2(0,T; H^1_{(0)}(\Omega\setminus \Gamma_t))\\
  h&\in L^2(0,T;H^{\frac52}(\T^1))\cap H^1(0,T;H^{\frac12}(\T^1))
\end{align*}
solving
\begin{align}\label{eq:CoupledStokes1}
\partial_t\mathbf{v}^{\pm}{+\ue^\pm \cdot \nabla \ve^\pm +\ve^\pm \cdot \nabla \ue^\pm}-\nu(\pm1)\Delta\ve^{\pm}&+\nabla p^{\pm}  =\mathbf{f}, &  &\text{in }\Omega^{\pm}(t), t\in (0,T),\\
\operatorname{div}\mathbf{v}^{\pm} & =g , &  &\text{in }\Omega^{\pm}(t), t\in (0,T),\\
\llbracket\mathbf{v}\rrbracket\cdot\nn & =\mathbf{a}_{1}\cdot \nn,  &  & \text{on }\Gamma_t, t\in (0,T),\\
\llbracket\mathbf{v}\rrbracket\cdot\btau & =\mathbf{a}_{1}\cdot\btau-\mathbf{u}_0\cdot\btau h,  &  & \text{on }\Gamma_t, t\in (0,T),\\
\llbracket 2\nu D\ve-p \tn{I}\rrbracket\no_{\Gamma_t} +\llbracket \ve \rrbracket\big(\partial_td_{\Gamma}-\sigma_{\eta}\Delta d_{\Gamma}\big)\\ -\sigma_{\eta}\llbracket \ve \rrbracket\cdot\nabla^2d_{\Gamma}-\sigma_{\eta}\llbracket \nabla\ve\rrbracket\cdot\nabla d_{\Gamma}&= \mathbf{b},  &  & \text{on }\Gamma_t, t\in (0,T),\\
{\partial_t^\Gamma h+\ve_0\circ X_0\cdot\nabla_\Gamma h-\Delta_{\Gamma}h+\sigma_0^{-1}\sigma_2\no\cdot\ve^+}&\circ X_0\\ {+(1-\sigma_0^{-1}\sigma_2)\no\cdot\ve^-\circ X_0+\sigma_0^{-1}\sigma_1g_{0}h}&=\mathbf{w}, &  & \text{on }\T^1\times(0,T),\\\label{eq:CoupledStokesLast}
  \mathbf{v}^{-} & =\mathbf{a},  &  &\text{on }\partial\Omega\times (0,T),\\\label{eq:CoupledIV1}
  \mathbf{v}^\pm|_{t=0} &= \mathbf{v}_0^\pm. &&\text{in }\Omega,\\\label{eq:CoupledIV2}
    h|_{t=0} &= h_0. &&\text{on }\T^1.
\end{align}
\end{thm}
\begin{proof}
  First of all, we can easily reduce to the case $\mathbf{a}_1\equiv 0$ and $\mathbf{a}_2$ by subtracting a suitable extension.

  \noindent
  \textbf{Step 1: Existence for $T=T_1>0$ sufficiently small:} In this case the proof can be done in the same way as the proof of \cite[Theorem~3.2]{StrongNSMS}. In the present case the proof is even simpler and one simply has to omit the terms related to the convective term $u\cdot\nabla_h u$ and the (given) surface tension term $\sigma \tilde{H}_h$.

    \noindent
    \textbf{Step 2: Existence for general $T>0$:} Since the system is linear we get that the existence time $T_1>0$ is independent of the norms of the data. Moreover, we obtain that for any $t_0\in [0,T)$ there is some $T_1(t_0)>t_0$ such that the system has a unique solution for $t\in (t_0,T_1(t_0))$ for a given initial value $\ve|_{t=t_0}=\tilde{\ve}_0$ at $t=-1$. Because of the compactness of $[0,T]$, we can concatenate these solutions and obtain a solution on $(0,T)$. 
  \end{proof}
  \begin{rem}
    With the aid of Theorem~\ref{thm:Approx1} one can obtain a smooth solution of \eqref{eq:CoupledStokes1}-\eqref{eq:CoupledStokesLast} for smooth $\mathbf{f}, g, a_{1,2}, \mathbf{b}, \mathbf{w}, \mathbf{a}$ (without precribed initial values $\ve_0$, $h_0$) as follows: Extend the smooth data $\mathbf{f}, g, a_{1,2}, \mathbf{b}, \mathbf{w}, \mathbf{a}$ and $\Omega^\pm(t), \Gamma_t$ on a time interval $[-1,T_0]$ in a smooth manner such that these functions vanish in $[-1,\frac12]$. Then one can apply Theorem~\ref{thm:Approx1} to obtain a solution of \eqref{eq:CoupledStokes1}-\eqref{eq:CoupledStokesLast} on a time intervall $(-1,T_0)$ instead of $(0,T_0)$ and with initial values $\ve_0 =0$, $h_0=0$. Then one can apply the parameter-trick in time to obtain that the solution is smooth in $(-1,T_0]$. Combining this with elliptic regularity theory for the (time-independent) Stokes system and elliptic equations, one obtains that $\ve^\pm$ is smooth in $\bigcup_{t\in (-1,T_0]}\Omega^\pm(t)\times \{t\}$ and $h$ is smooth in $\bigcup_{t\in (-1,T_0]}\Gamma_t\times \{t\}$. Restriction to $[0,T_0]$ in time yields the existence of a smooth solution to \eqref{eq:CoupledStokes1}-\eqref{eq:CoupledStokesLast}, which satisfy \eqref{eq:CoupledIV1}-\eqref{eq:CoupledIV2} for some $\ve_0^\pm:= \ve^\pm|_{t=0}$, $h_0 := h|_{t=0}$. 
  \end{rem}

\section*{Acknowledgments}
M.~Fei was partially supported by NSF of China under Grant No.11871075 and 11971357. The support is gratefully acknowledged. Moreover, we thank Maxilian Moser for several helpful comments on a previous version of this contribution and Yuning Liu for helpful discussions on this topic. {Moreover, we are grateful to the anonymous referees for their careful reading of the manuscript and many valuable comments to improve the presentation.}

\def\ocirc#1{\ifmmode\setbox0=\hbox{$#1$}\dimen0=\ht0 \advance\dimen0
  by1pt\rlap{\hbox to\wd0{\hss\raise\dimen0
  \hbox{\hskip.2em$\scriptscriptstyle\circ$}\hss}}#1\else {\accent"17 #1}\fi}

\bigskip

\noindent
{\it
  (H. Abels) Fakult\"at f\"ur Mathematik,
  Universit\"at Regensburg,
  93040 Regensburg,
  Germany}\\
{\it E-mail address: {\sf helmut.abels@mathematik.uni-regensburg.de} }\\[1ex]
{\it
(M. Fei) School of  Mathematics and Statistics, Anhui Normal University, Wuhu 241002, China}\\
{\it E-mail address: {\sf mwfei@ahnu.edu.cn} }\\[1ex]
\end{document}